\documentclass[10pt,a4paper]{article}
\usepackage[utf8]{inputenc}
\usepackage{amsmath}
\usepackage{amsfonts}
\usepackage{amssymb}
\usepackage{amsthm}
\usepackage{mathtools}
\usepackage{color}
\usepackage{enumitem}
\usepackage[dvipsnames]{xcolor}
\usepackage{hyperref}

\hypersetup{colorlinks = true,
            linkcolor = blue,
            citecolor = LimeGreen}

\DeclareMathOperator{\Exp}{Exp}
\DeclareMathOperator{\Div}{div}

\DeclareMathOperator{\D}{d}
\DeclareMathOperator{\re}{Re}
\DeclareMathOperator{\im}{Im}
\DeclareMathOperator{\res}{Res}
\DeclareMathOperator{\diam}{diam}

\newtheorem{lemma}{Lemma}[section]
\newtheorem{theorem}{Theorem}
\newtheorem{proposition}[lemma]{Proposition}
\theoremstyle{definition}

\theoremstyle{remark}
\newtheorem{remark}[lemma]{Remark}

\newcommand{\set}[1]{\left\{ #1 \right\}}
\newcommand{\n}[1]{\left\| #1 \right\|}
\newcommand{\T}{\mathcal{T}}
\newcommand{\F}{\mathcal{F}}

\title{Properties of resonant states for generic smooth expanding maps}
\date{}
\author{Malo J\'ez\'equel\thanks{CNRS, Univ. Brest, UMR6205, Laboratoire de Math{\'e}matiques de Bretagne Atlantique, France. email: malo.jezequel@math.cnrs.fr}}

\begin{document}

\maketitle

\begin{abstract}
We prove that the resonances for a generic smooth expanding map are simple and that zero is a regular value for the associated resonant states. Moreover, the real-valued resonant states are Morse functions. Using Nash-Moser theory, we also prove that the resonant states for a generic smooth expanding map of large enough degree (depending on the dimension of the manifold the map is acting on) have the same properties as generic smooth functions.
\end{abstract}

\section{Introduction}

Let $M$ be a compact boundaryless connected $C^\infty$ manifold. For convenience, we endow $M$ with a smooth Riemannian metric. We let $\Exp(M)$ be the set of smooth ($C^\infty$) expanding maps from $M$ to itself. We recall that a smooth map $T : M \to M$ is expanding if there are constants $C,\theta > 0$ such that for every\footnote{We use a calligraphic $\mathcal{T}$ in the notation for tangent spaces and bundles to avoid ambiguity with the letter $T$ which is often used to denote a smooth map.} $x \in M, v \in \T_x M$ and $n \in \mathbb{N}$, we have
\begin{equation}\label{eq:definition_expanding}
|D T^n(x) \cdot v | \geq C e^{n \theta} |v|. 
\end{equation}
Smooth expanding maps are among the simplest examples of hyperbolic dynamical systems. Recall indeed that a dynamical system is said to be hyperbolic if it is contracting in a direction and expanding in a supplementary direction. In the case of expanding maps, the contracted direction is trivial.

Hyperbolic dynamical systems are known for their rich statistical properties, and expanding maps are no exception. If $T \in \Exp(M)$ then there is a unique Borel probability measure $\mu$ on $M$ which is both invariant by $T$ and absolutely continuous with respect to the class of Lebesgue measure \cite{acim_expanding}. The measure $\mu$ plays a central role in the understanding of the statistical properties of $T$. Since $T$ is smooth, we even know that, if $\mathrm{d}x$ denotes the Riemannian volume density on $M$, then $\mu = \rho \mathrm{d}x$ with $\rho$ a smooth positive function on $M$. 

Moreover, the measure $\mu$ is known to be exponentially mixing: there is $\delta \in (0,1)$ such that for every $f,g \in C^\infty(M,\mathbb{C})$ we have (see e.g. \cite[Chapter 2]{baladi_book2})
\begin{equation}\label{eq:exponential_mixing}
\int_M f. g \circ T^n \mathrm{d}\mu \underset{n \to + \infty}{=} \int_M f \mathrm{d}\mu \int_M g \mathrm{d}\mu + \mathcal{O}(\delta^n).
\end{equation}
The notion of \emph{Ruelle resonances} \cite{ruelle_expanding_maps} allows to sharpen this estimate. There is a discrete subset $\res(T)$ in $\mathbb{C}^*$ such that for each $\lambda \in \res(T)$ there is a finite dimensional subspace $E_{T,\lambda}$ of $C^\infty(M,\mathbb{C})$, a continuous projection $\Pi_{T,\lambda} : C^\infty(M,\mathbb{C}) \to E_\lambda$ and a nilpotent endomorphism $N_{T,\lambda} : E_{T,\lambda} \to E_{T,\lambda}$ with the properties that for every $\delta > 0$ and $f,g \in C^\infty(M,\mathbb{C})$ we have the following asymptotic of correlations formula, which is a consequence of \cite[Theorem 3.2]{ruelle_expanding_maps}:
\begin{equation}\label{eq:asymptotic_correlation}
\int_M f. g \circ T^n \mathrm{d}x \underset{n \to + \infty}{=} \sum_{\substack{\lambda \in \res(T) \\ |\lambda| \geq \delta}} \int_{M} \left((\lambda I + N_{T,\lambda})^n \Pi_{T,\lambda} (f)\right) g \mathrm{d}x + \mathcal{O}(\delta^n).
\end{equation}

The elements of $\res(T)$ are called the resonances for $T$ and $\res(T)$ itself is the Ruelle spectrum of $T$. For $\lambda \in \res(T)$, the non-zero elements of $E_{T,\lambda}$ are called the \emph{generalized resonant states} for $T$ associated to $\lambda$, and the non-zero elements of $\ker N_{T,\lambda}$ are called \emph{resonant states}. The dimension of $E_{T,\lambda}$ is called the multiplicity of $\lambda$ (as a resonance for $T$). Notice that if $\lambda$ is a simple resonance (i.e. of multiplicity $1$), then the term corresponding to $\lambda$ in \eqref{eq:asymptotic_correlation} takes the form
\begin{equation*}
\lambda^n \nu_{T,\lambda}(f) \int_M \phi_{T,\lambda} g \mathrm{d}x,
\end{equation*}
where $\nu_{T,\lambda}$ is a distribution on $M$ and $\phi_{T,\lambda}$ a smooth function on $M$ (that spans $E_{T,\lambda}$).

The aim of this paper is to describe the generic properties of resonances and resonant states for smooth expanding maps. Notice that it follows from \eqref{eq:exponential_mixing} that $1$ is always a simple resonance for $T$ with (up to rescaling) $\nu_{T,1} : f \mapsto \int_{M} f\mathrm{d}x$ and $\phi_{T,1} = \rho$. Moreover, all the other resonances of $T$ have strictly smaller modulus. Our first result is the following (the topology on $\Exp(M)$ is induced by the usual topology on $C^\infty(M,M)$).

\begin{theorem}\label{theorem:generic_simple}
Let $\delta > 0$. The set of smooth expanding maps whose resonances of moduli larger than or equal to $\delta$ are all simple is open and dense in $\Exp(M)$.
\end{theorem}

In particular, the set of smooth expanding maps whose resonances are all simple is a $G_\delta$ dense set in $\Exp(M)$. Once we know that the resonances of an expanding maps are simple, we can investigate the generic properties of the associated resonant states, and we find:

\begin{theorem}\label{theorem:generic_morse}
Let $\delta > 0$. Let $\mathcal{U}_\delta$ be the set of smooth expanding maps $T$ with the following properties:
\begin{itemize}
\item the resonances of $T$ of moduli larger than or equal to $\delta$ are simple;
\item if $f$ is a real-valued resonant state associated to a real resonance of $T$ of absolute value larger than or equal to $\delta$, then $f$ is a Morse function and $0$ is a regular value for $f$;
\item if $f$ is a resonant state associated to a resonance for $T$ in $\mathbb{C} \setminus \mathbb{R}$ of modulus larger than or equal to $\delta$, then $0$ is a regular value for $f$.
\end{itemize}
Then $\mathcal{U}_\delta$ is open and dense in $\Exp(M)$.
\end{theorem}

We can go further in the description of resonant states. Let us start with the resonant state associated to the resonance $1$, which is just (after suitable normalization) the density of the invariant absolutely continuous probability measure.

\begin{theorem}\label{theorem:generic_density}
Let $U \subseteq \set{\rho \in C^\infty(M,\mathbb{R}_+^*) : \int_M \rho \mathrm{d}x = 1}$ be open and dense. Let $\mathcal{U}$ be the set of $T$ in $\Exp(M)$ such that the density of the invariant absolutely continuous probability measure for $T$ belongs to $U$. The set $\mathcal{U}$ is open and dense in $\Exp(M)$.
\end{theorem}

To put it loosely, ``the density of the absolutely continuous invariant probability measure for a generic smooth expanding map is a generic smooth positive function of integral $1$''. Theorem \ref{theorem:generic_density} is obtained by proving that the map that associates to a smooth expanding map the density of its absolutely continuous invariant probability measure is locally surjective (it has local right inverses).

We need extra definitions in order to discuss the other resonances. Since there is in general no preferred element in $E_{T,\lambda}$ (except in the case $\lambda = 1$), we will study the generic properties of the space $E_{T,\lambda}$ itself. In order to do so, for $\mathbb{K} = \mathbb{R}$ or $\mathbb{C}$, we let $C_0^\infty(M,\mathbb{K})$ denote the space of zero-mean\footnote{It follows from \eqref{eq:asymptotic_correlation} and the fact that $\mu$ is invariant for $T$ that the generalized resonant states associated to resonances different from $1$ have zero average.} (for $\mathrm{d}x$) smooth functions from $M$ to $\mathbb{K}$. We define then $PC_0^\infty(M,\mathbb{K})$ as the space of one-dimensional subspaces of $C_0^\infty(M,\mathbb{K})$. We put a topology on $PC_0^\infty(M,\mathbb{K})$ by identifying it with the quotient of $C_0^\infty(M,\mathbb{K}) \setminus \set{0}$ under the action of $\mathbb{K}^*$. If $\phi$ is an element of $C_0^\infty(M,\mathbb{K}) \setminus \set{0}$, then we will denote by $[ \phi ]$ its equivalence class (or the line that it spans depending on the point of view).

Let us recall that if $T \in \Exp(M)$ then $T$ is a covering map of $M$ by itself. Since $M$ is compact, $T$ has a finite number of sheets, we call this number the degree of $T$, that we write $\deg T$. Notice that any point in $M$ has exactly $\deg T$ antecedents by $T$ and that, if $M$ is orientable, then $\deg T$ is the absolute value of the topological degree of $T$. For $m \geq 2$, we will denote by $\Exp_{\geq m}(M)$ the set of expanding maps on $M$ whose degree is larger than or equal to $m$. With these notations, we have:

\begin{theorem}\label{theorem:generic_real_resonance}
Let $\delta > 0$. Let $U$ be an open subset of $PC_0^\infty(M,\mathbb{R}) \times \mathbb{R}$. Let $m = \dim M + 1$. Let $\mathcal{U}_\delta$ be the set of $T \in \Exp_{\geq m}(M)$ such that if $\lambda \in \res(T)$ and $\lambda \in \mathbb{R}\setminus ((-\delta,\delta) \cup \set{1})$ then $\lambda$ is simple and\footnote{The intersection with $C_0^\infty(M,\mathbb{R})$ is here because, as we defined it, $E_{T,\lambda}$ is a complex linear subspace of $C_0^\infty(M,\mathbb{C})$. However, in the context of Theorem \ref{theorem:generic_real_resonance}, $E_{T,\lambda(T)}$ is a one-dimensional subspace spanned by a real-valued function, so that $E_{T,\lambda}\cap C_0^\infty(M,\mathbb{R})$ belongs to $PC_0^\infty(M,\mathbb{R})$. We might sometimes be slightly less accurate and identify $E_{T,\lambda}$ with a subspace of $C_0^\infty(M,\mathbb{R})$ when $\lambda$ is real.} $(E_{T,\lambda} \cap C_0^\infty(M,\mathbb{R}),\lambda) \in U$. The set $\mathcal{U}_\delta$ is open and dense in $\Exp_{\geq m}(T)$.
\end{theorem}

Notice that in dimension $1$, the degree condition in Theorem \ref{theorem:generic_real_resonance} is empty. For non-real resonances, this result becomes:

\begin{theorem}\label{theorem:generic_complex_resonance}
Let $\delta > 0$. Let $U$ be an open subset of $PC_0^\infty(M,\mathbb{C})\times \mathbb{C}$. Let $m = \dim M + 2$. Let $\mathcal{U}_\delta$ be the set of $T \in \Exp_{\geq m}(M)$ such that if $\lambda \in \res(T)$ and $\lambda \in \mathbb{C}\setminus (\mathbb{D}(0,\delta) \cup \mathbb{R})$ then $\lambda$ is simple and $(E_{T,\lambda},\lambda) \in U$. The set $\mathcal{U}_\delta$ is open and dense in $\Exp_{\geq m}(T)$.
\end{theorem}

Theorems \ref{theorem:generic_density}, \ref{theorem:generic_real_resonance} and \ref{theorem:generic_complex_resonance} are particular cases of a more general result (Theorem \ref{theorem:general_statement}) that allows to deal with several resonances simultaneously. Theorems \ref{theorem:generic_real_resonance} and \ref{theorem:generic_complex_resonance} may be stated informally as `` the resonant states for a generic smooth expanding map of large enough degree have the generic properties of smooth maps of average zero''.

\subsection*{Context}

The concept of resonances for hyperbolic dynamical systems originated in the work of Ruelle \cite{ruelle_resonances} and Pollicott \cite{pollicott_resonances}. In the context of smooth expanding maps, the theory of resonances is developed in \cite{ruelle_expanding_maps}. For a textboof presentation of this topic, one may refer to \cite[Chapter 2]{baladi_book1} or \cite[Part I]{baladi_book2}.

The specific case of real-analytic expanding maps (in particular in low dimension) has been extensively studied, see for instance\footnote{Notice that some references in this list does not directly deal with expanding maps but study instead some model transfer operators associated to a family of contractions.} \cite{bandtlow_jenkinson_2007, bandtlow_jenkinson_2008, bandtlow_jenkinson_2008_2, naud_ruelle_spectrum, bjs_2013, bjs_2017, bandtlow_naud}. To this list we may add the seminal work on Ruelle on the related topic of zeta functions \cite{ruelle_zeta}. The paper \cite{bjs_2013} by Bandtlow, Just and Slipantschuk is of specific interest to us as it gives examples of expanding maps on the circle with explicit Ruelle resonances. In particular, the authors exhibit expanding maps with non-real resonances, which proves that Theorem \ref{theorem:generic_complex_resonance} is not empty. However, they do not give examples of expanding maps with simple real resonances distinct from $1$. We will see in Appendix \ref{appendix:existence_simple_real_resonances} that such maps exist, which implies that Theorem \ref{theorem:generic_real_resonance} is not empty\footnote{Theorems \ref{theorem:generic_real_resonance} and \ref{theorem:generic_complex_resonance} are a priori coherent with the absence of resonances distinct from $1$. The results from \cite{bjs_2013} and Appendix \ref{appendix:existence_simple_real_resonances} shows that the domain of application of Theorems \ref{theorem:generic_real_resonance} and \ref{theorem:generic_complex_resonance} is non trivial.}.

The result from \cite{bandtlow_naud} is also relevant for the study of generic properties of resonances for expanding maps. In this reference, Bandtlow and Naud prove a lower bound on the number of resonances\footnote{More precisely, they show that a lower bound on the growth of the number of resonances of modulus larger than $r$ as $r$ goes to $0$.} for a dense set of analytic expanding maps of the circle. This is a bit different from the generic properties that we study in the present paper: we focus on ``local'' properties of the resonance spectrum (i.e. properties that involve only a finite number of resonances simultaneously), while the lower bound from \cite{bandtlow_naud} is a ``global'' property of resonances (it only makes sense when considering the whole spectrum at once). Such global properties of the resonance spectrum usually require regularity assumptions beyond $C^\infty$ to be dealt with\footnote{See \cite{jezequel_ultradifferentiable_expanding} for a systematic study of resonances for expanding maps of the circle in classes of regularity between $C^\infty$ and analytic.}. On the other hand, local properties are accessible in the smooth category\footnote{And even in finite regularity with some extra care. Notice however that the Nash--Moser theory methods that we will use in this paper would not apply directly in finite regularity. Hence, it is likely that Theorems \ref{theorem:generic_simple} and \ref{theorem:generic_morse} may be adapted to deal with finitely differentiable expanding maps (with extra technicalities) but it is probably much harder to adapt Theorems~\ref{theorem:generic_density}, \ref{theorem:generic_real_resonance} and \ref{theorem:generic_complex_resonance}.}, which is thus a natural setting for the properties that we study in the present paper.

The concept of Ruelle resonances also apply to more general hyperbolic dynamical systems, in particular to smooth Anosov diffeomorphisms (see e.g. \cite[Part II]{baladi_book2}), and some of the results mentioned above have analogues in this context. The existence of Anosov maps with non-trivial (i.e. distinct from $1$) resonances is not obvious. It was first established by Adam \cite{adam_anosov} who proved the existence of Anosov diffeomorphisms with at least one non-trivial resonance by perturbative methods. Then, Bandtlow, Just and Slipantschuk \cite{complete_Anosov, resonances_rational} produced many examples of Anosov diffeomorphisms of the two-dimensional torus with explicit resonance spectra (see also \cite{pollicott_sewell}). We used these examples to adapt the result of \cite{bandtlow_naud} and prove a lower bound on the number of resonances for a dense subset of the space of analytic Anosov diffeomorphisms of the two-dimensional torus.

Beware that the resonant states in the Anosov case are distributions rather than smooth functions. Hence, it is not clear what could be the equivalent of a statement such as Theorem \ref{theorem:generic_morse} in this context. Moreover, the resonant states associated to an Anosov diffeomorphism are not any distributions, they have regularity properties that depend on the Anosov map (in particular through its stable and unstable direction) and on the associated resonances (resonant states associated to small resonances can be more irregular). Hence, the precise regularity properties of a resonant state are a priori not preserved under a small deformation of the Anosov map, which is a tricky feature to deal with when studying their generic properties.

There is a question however that still makes sense in the Anosov case (even in the context of Anosov flow) which is the simplicity of resonances (Theorem \ref{theorem:generic_simple}). Related questions are addressed for geodesic flows on the unit tangent bundle of a generic perturbation of a hyperbolic $3$-manifold in \cite{three_manifolds} by Ceki\'c, Delarue, Dyatlov and Paternain. The methods in this paper are similar to our proof of Theorem \ref{theorem:generic_simple} as the authors also rely on a deformation argument by computing the first order variation of some spectral quantities. The geometric context however is much more involved in \cite{three_manifolds} than here.

Some of the methods that we will use could be related to recent works in linear response theory. In particular, in \cite{galatolo_pollicott_control} Galatolo and Pollicott solve the linearized problem associated to Theorem \ref{theorem:generic_density} in dimension $1$. The higher dimensional case is dealt with by Kloeckner in \cite{kloeckner_linear_request}. Notice that our proof of Theorem \ref{theorem:generic_density} implies to solve the linearized problem.

Finally, let us mention that a motivation for the current paper is the work of Uhlenbeck \cite{uhlenbeck_generic} on generic properties of eigenfunctions for elliptic differential operators. A crucial element in the article by Uhlenbeck is the use of infinite-dimensional transversality theorem \cite{smale_infinite,abraham_transversality, quinn_transversal} (in particular, an infinite-dimensional analogue of Sard's theorem). The core ideas of the current paper originated from the reading of \cite{uhlenbeck_generic}. However, the actual tools from \cite{uhlenbeck_generic} are not really suited for the analysis of resonant states of generic expanding maps. The reason for that is that \cite{uhlenbeck_generic} only deals with a Banach setting, while some specifities of the perturbation theory for resonances make it more convenient to work in a Fr\'echet framework. Fortunately, the Nash--Moser theory as it is exposed by Hamilton in \cite{hamilton_ift} turned out to be very well suited to our dynamical context.

\subsection*{Methods}

The main tool in this paper is the theory of perturbations for Ruelle resonances and resonant states. The specific case of the leading resonance $1$ is often called linear response\footnote{Linear response is a term from statistical physics that denotes the first order variation of an observed quantity under a small perturbation of a system.} theory \cite{baladi_or_else}. As explained in \S \ref{subsection:resonances}, the resonances for a smooth expanding map $T$ are eigenvalues for a \emph{transfer operator} \eqref{eq:definition_transfer_operator} acting on $C^k(M,\mathbb{C})$ for different values of $k > 0$. However, this operator does not depend smoothly on $T$ as a bounded operator on $C^k(M,\mathbb{C})$, and the standard perturbation theory \cite{kato_book} does not apply. Dynamicists have developed methods to bypass this difficulty. There are in particular abstract results of Gou\"ezel, Keller and Liverani that apply perfectly to the case of smooth expanding maps, see \cite[I.2.5 and III.A.3]{baladi_book2}. Once the specific difficulties related to the dynamical context are dealt with, the resulting perturbation theory is very similar to what would happen for bounded operators on a Banach space (in terms of formula for the derivative of the spectral data for instance).

Once we have a good perturbation theory for resonances, we need to produce relevant perturbations for the proof of our different theorems. The perturbation theory exposed in \S \ref{subsection:linear_response_theory} implies that the set $\mathcal{U}_\delta$ from Theorem \ref{theorem:generic_simple} is open. In order to prove that it is dense, we want to prove that if $T_0 \in \Exp(M)$ has a resonance $\lambda_0$ that is not simple, then we can produce a smooth perturbation $(T_t)_{t \in (-\epsilon,\epsilon)}$ of $T_0$ such that for $t$ small but non-zero the resonance $\lambda_0$ splits into several resonances\footnote{If there are Jordan blocks associated to $\lambda_0$, we can also reduce the size of the largest Jordan block.}. By applying this process several times, we end up finding that $\mathcal{U}_\delta$ is dense. The condition on $(T_t)_{t \in (-\epsilon,\epsilon)}$ for a splitting of $\lambda_0$ to happen involves the coresonant states for $\lambda_0$. As an intermediary step in the proof, we will show that the coresonant states for smooth expanding maps are fully supported (Proposition \ref{proposition:full_support}). Notice that Weich proved that resonant states for Anosov systems are fully supported \cite{weich_support}.

The proof of Theorem \ref{theorem:generic_morse} is based on a standard transversality argument (a corollary of Sard's theorem). Considering a smooth expanding map $T_0$ and a real resonant state $f_0$ for $T_0$ (associated to a simple resonance), we build a smooth deformation $(-\epsilon,\epsilon)^N \ni t \mapsto T_t$ of $T_0$ such that, if $f_t$ denotes the deformation of $f_0$ into a resonant state of $T_t$, then $0$ is a regular value of the map $(x,t) \mapsto f_t(x)$. It follows then that, for almost all $t$ in $(-\epsilon,\epsilon)^N$, the number $0$ is a regular value of the map $f_t$. By similar consideration, we can make $f_t$ Morse. Once again, the point here is to produce enough interesting perturbations of the resonant state $f_0$. These perturbations are obtained in Lemma \ref{lemma:key_deformation}. This lemma may seem involved, but this is because it is designed so that it applies in different situations that appear along the paper. For the proof of Theorem \ref{theorem:generic_morse}, the only point that matters is that for every $x_0 \in M$ and $g \in C^\infty_0(M,\mathbb{C})$, we can produce a deformation of $T_0$ for which the first order change in $f_0$ is given by $g$ near $x_0$. This perturbation is obtained by direct inspection of the formula for the first order change in $f_0$. The specific ideas behind the proof of Lemma \ref{lemma:key_deformation} are difficultly explained without referring to some notions from \S \ref{section:background}. Their exposition is deferred to \S \ref{section:key_lemma}.

Let us now explain the structure of the proofs of the Theorems \ref{theorem:generic_density}, \ref{theorem:generic_real_resonance} and \ref{theorem:generic_complex_resonance}. As above, we consider a smooth expanding map $T_0$ and a simple, say real and distinct from $1$, resonance $\lambda_0$. For $T$ close to $T_0$, there is a single resonance $\lambda(T)$ for $T$ close to $\lambda_0$. Our goal is to find a local right inverse for the map
\begin{equation}\label{eq:map_to_invert}
T \mapsto (E_{T,\lambda(T)} \cap C_0^\infty(M,\mathbb{R}),\lambda(T)) \in PC_0^\infty(M,\mathbb{R}) \times \mathbb{R}.
\end{equation}
We start by giving a sufficient condition on $T_0$ for the existence of a right inverse for the \emph{derivative} of \eqref{eq:map_to_invert}. Then, we apply Nash--Moser theory to show that, when this condition is satisfied, there is a local right inverse for \eqref{eq:map_to_invert}. This is the content of Lemma \ref{lemma:local_solvability}. Nash--Moser theory is often advertised as a way to solve a non-linear problem when there is a deregularizing operator solving the linearized problem, but this is not the case here. Actually, an inspection of the proof of Lemma \ref{lemma:local_solvability} shows that the solution that we give to the linearized problem is given by a regularizing operator. The linearized problem however is given by a deregularizing operator. This feature of our problem makes the approach of Nash--Moser theory exposed in \cite{hamilton_ift} very convenient. The main result of the theory is stated there as an inverse function theorem in the category of so-called smooth tame map. Hence, there is a symmetry between the linearized problem and its solution, so that which one is deregularizing does not matter.

Once we have a sufficient condition for the existence of a local right inverse for \eqref{eq:map_to_invert}, the final ingredient of the proof is to check that a generic expanding map (with large enough degree) satisfies this condition. This is the content of Lemma \ref{lemma:generic_solvability}, whose proof is based on a succession of transversality arguments (sketched in Remark \ref{remark:sketch_proof}). The condition on the degree of the map appears in these arguments: it is related to the fact that a transverse intersection between two manifolds of large enough codimensions is empty. These transversality arguments will require relevant perturbations of resonant states, and this is where we will need the full generality of the deformations from Lemma \ref{lemma:key_deformation}.

\subsection*{Structure of the paper}

In \S \ref{section:background} we recall standard results on resonances and their perturbation theory.
In \S \ref{section:simplicity}, we prove Theorem \ref{theorem:generic_simple}.
In \S \ref{section:key_lemma}, we prove our key Lemma \ref{lemma:key_deformation} that provides many useful deformations of expanding maps.
In \S \ref{section:first_generic_properties} we use Lemma \ref{lemma:key_deformation} to prove Theorem \ref{theorem:generic_morse}.
In \S \ref{section:general_results}, we state and prove a general result that implies Theorems \ref{theorem:generic_real_resonance} and \ref{theorem:generic_complex_resonance}. We also prove Theorem \ref{theorem:generic_density}.

This paper contains two appendices. In Appendix \ref{appendix:existence_simple_real_resonances}, we prove that there exists expanding maps with simple real resonances (proving that Theorem \ref{theorem:generic_real_resonance} is not empty). In Appendix \ref{appendix:weighted_transfer_operator}, we discuss the potential adaptation of our analysis to the case of weighted transfer operators.

This paper is organized so that the notions of tame Fréchet manifolds and smooth tame maps only appear in \S \ref{section:general_results}. Hence, the proofs of Theorems \ref{theorem:generic_simple} and \ref{theorem:generic_morse} may be read without any knowledge in Nash--Moser theory.

\subsection*{Acknowledgements}

The author benefits from the support of the French government “Investissements d’Avenir” program integrated to France 2030, bearing the following reference ANR-11-LABX-0020-01.

\section{Background and notations}\label{section:background}

In this first section, we recall some bacgkground facts and fix notations for the rest of the paper. We start with generalities on resonances in \S \ref{subsection:resonances} and then pay a particular attention to the perturbation theory of resonances in \S \ref{subsection:linear_response_theory}.

\subsection{Transfer operator and resonances}\label{subsection:resonances}

Let us fix $T \in \Exp(M)$. The main character in the theory of Ruelle resonances for $T$ is the transfer operator $\mathcal{L}_T$ defined by
\begin{equation}\label{eq:definition_transfer_operator}
\mathcal{L}_Tf(x) = \sum_{\substack{y \in M \\ Ty = x}} \frac{f(y)}{|\det DT(y)|}
\end{equation}
for $f : M \to \mathbb{C}$ and $x \in M$. Here, the Jacobian determinant $|\det DT|$ is defined using the density $\mathrm{d}x$. The operator $\mathcal{L}_T$ is the adjoint of the composition operator associated to $T$ in the following sense: if $f,g \in C^\infty(M,\mathbb{C})$ then the change of variable formula yields
\begin{equation}\label{eq:change_of_variable}
\int_M f. g \circ T \mathrm{d}x = \int_M \mathcal{L}_T(f) g\mathrm{d}x.
\end{equation}

The expansion \eqref{eq:asymptotic_correlation} is a consequence of the following result (see \cite[Corollary 2.6]{baladi_book2} for a proof that Proposition \ref{proposition:essential_spectral_radius} implies \eqref{eq:asymptotic_correlation}).

\begin{proposition}[Theorem 3.2 in \cite{ruelle_expanding_maps}]\label{proposition:essential_spectral_radius}
Let $T \in \Exp(M)$. There are $C, \theta > 0$ such that for every $k \in \mathbb{N}$ the operator $\mathcal{L}_T$ induces a bounded operator from $C^k(M,\mathbb{C})$ to itself with essential spectral radius less than $C e^{-k \theta}$.
\end{proposition}

We recall that if $L : \mathcal{B} \to \mathcal{B}$ is a bounded endomorphism of Banach space $\mathcal{B}$, then the essential spectral radius of $L$ is the smallest $r > 0$ such that the intersection of the spectrum of $L$ with $\set{z \in \mathbb{C} : |z|> r}$ only contains isolated eigenvalues of finite multiplicity. Hence, Proposition \ref{proposition:essential_spectral_radius} implies that for every $k \in \mathbb{N}$ the resolvent $(z - \mathcal{L}_T)^{-1} : C^k(M,\mathbb{C}) \to C^k(M,\mathbb{C})$ is meromorphic on $\set{z \in \mathbb{C} : |z| > C \lambda^{-k}}$ with coefficients of finite ranks in the Laurent series expansion at each pole.

Consequently, if we define for $|z| \gg 1$ the operator $R_T(z) : C^\infty(M,\mathbb{C}) \to C^0(M,\mathbb{C})$ by $R_T(z) f = \sum_{k = 0}^{+ \infty} z^{-k-1} \mathcal{L}_T^k f$, then $R_T(z)$ actually maps $C^\infty(M,\mathbb{C})$ into itself, and has a meromorphic extension, as an operator from $C^\infty(M,\mathbb{C})$ to itself, to $\mathbb{C}^*$, with coefficients of finite ranks in the Laurent series at each pole. The meromorphic extension of $R_T(z)$ will still be denoted by $R_T(z)$.

The set $\res(T)$ of resonances for $T$ are then the poles of $R_T(z)$. Equivalently, we find that $\lambda \in \mathbb{C}^*$ is a resonance for $T$ if there is $f \in C^\infty(M,\mathbb{C})$ non-zero such that $\mathcal{L}_T f = \lambda f$. If $\lambda$ is a resonance, then the operator $\Pi_{T,\lambda}$ from \eqref{eq:asymptotic_correlation} is the residue of $R_T(z)$ at $\lambda$:
\begin{equation}\label{eq:definition_projector}
\Pi_{T,\lambda} = \frac{1}{2 i \pi} \int_{\partial \mathbb{D}(\lambda,\epsilon)} R_T(z) \mathrm{d}z,
\end{equation}
where $\epsilon > 0$ is small enough so that $\lambda$ is the only resonance for $T$ within $\overline{\mathbb{D}}(\lambda,\epsilon)$. The space $E_{T,\lambda}$ is then the range of $\Pi_{T,\lambda}$, but also the characteristic space of $\mathcal{L}_T$ (acting on $C^\infty(M,\mathbb{C})$) for $\lambda$:
\begin{equation*}
E_{T,\lambda} = \Pi_{T,\lambda}(C^\infty(M,\mathbb{C})) = \set{ f \in C^\infty(M,\mathbb{C}): \exists N \in \mathbb{N}^*, (\mathcal{L}_T - \lambda)^N f = 0}.
\end{equation*}
The operator $N_{T,\lambda}$ is just the operator induced by $\mathcal{L}_T - \lambda I$ on $E_{T,\lambda}$. The resonant states for $\mathcal{L}_T$ are then the eigenvectors (associated to non-zero eigenvalues) of $\mathcal{L}_T$ acting on $C^\infty(M,\mathbb{C})$ and the generalized resonant states are the corresponding generalized eigenvectors. The dimension of $E_{T,\lambda}$ is finite, it is called the multiplicity of $\lambda$ has a resonance for $T$.

We let $\mathcal{D}'(M)$ denote the space of distributions on $M$, that we identify with the dual of $C^\infty(M,\mathbb{C})$ using the density $\mathrm{d}x$ to identify $C^\infty(M,\mathbb{C})$ with the space of smooth sections of the complexification of the density bundle on $M$. If $A$ is a bounded operator from $C^\infty(M,\mathbb{C})$ to itself, we let $A^*$ be the operator from $\mathcal{D}'(M)$ to itself defined by $A^* \nu(\varphi) = \nu(A \varphi)$ for $\nu \in \mathcal{D}'(M)$ and $\varphi \in C^\infty(M,\mathbb{C})$.

We will be particularly interested in the operator\footnote{Notice that, considering \eqref{eq:change_of_variable}, it would be legitimate to write $\nu \circ T$ instead of $\mathcal{L}_T^* \nu$.} $\mathcal{L}_{T}^*$. The (generalized) eigenvectors for $\mathcal{L}_T^*$ (associated to non-zero eigenvalues) are called \emph{(generalized) coresonant states} for $T$. For $\lambda \in \mathbb{C}^*$, the associated eigenspace is
\begin{equation*}
\mathcal{E}_{T,\lambda} \coloneqq \set{\nu \in \mathcal{D}'(M): \exists N \in \mathbb{N}^*, (\mathcal{L}_T^* - \lambda)^N \nu = 0}.
\end{equation*}
Using Proposition \ref{proposition:essential_spectral_radius}, we find that $\mathcal{E}_{T,\lambda}$ is the image of $C^\infty(M,\mathbb{C})$ by $\Pi_{T,\lambda}^*$. Hence, $\mathcal{E}_{T,\lambda}$ is non-trivial if and only if $\lambda$ is a resonance for $T$. Moreover, the pairing of distributions with functions induce a non-degenerate pairing between $E_{T,\lambda}$ and $\mathcal{E}_{T,\lambda}$. In particular, the dimension of $\mathcal{E}_{T,\lambda}$ is the multiplicity of $\lambda$ as a resonance for $T$.

\begin{remark}
There are some definitions above (starting with the definition of $\mathcal{L}_T$) that depends on the choice of the density $\mathrm{d}x$ on $M$. The impact of this choice is mostly irrelevant. Indeed, we could define the operator $\mathcal{L}_T$ as the pullback operator by $T$ acting on sections of the density bundle on $M$ to get more intrinsic statements. With this point of view, our choice of reference density is just a way to trivialize the density bundle.

The only statement that would change if we were working directly with densities is Theorem \ref{theorem:generic_morse}. Indeed, whether a resonant state is a Morse function or not depend on the choice of the reference density. Having zero as a regular value however is well-defined for densities.
\end{remark}

\begin{remark}\label{remark:1_resonance}
It follows from \eqref{eq:change_of_variable} that we have $\int_M \mathcal{L}_T f \mathrm{d}x = \int_M f \mathrm{d}x$ for every $f \in C^\infty(M,\mathbb{C})$. Hence, the distribution $f \mapsto \int_M f \mathrm{d}x$ is a coresonant state for $T$ associated to the resonance $1$. It can be shown (see e.g. \cite[Proposition 2.5]{baladi_book2}) that $1$ is always a simple resonance, and that there is an everywhere positive resonant state associated to $1$: this is the density of the absolutely continuous invariant measure for $T$.

It is important to remember that, since  $f \mapsto \int_M f \mathrm{d}x$ is a coresonant state for $T$ associated to $1$, any resonant state $g$ for $T$ associated to a resonance $\lambda \neq 1$ has zero average: $\int_M g \mathrm{d}x = 0$. This is why the spaces $C_0^\infty(M,\mathbb{R})$ and $C_0^\infty(M,\mathbb{C})$ appear in Theorems \ref{theorem:generic_real_resonance} and \ref{theorem:generic_complex_resonance}, instead of  $C^\infty(M,\mathbb{R})$ and $C^\infty(M,\mathbb{C})$. Notice also that $C_0^\infty(M,\mathbb{R})$ and $C_0^\infty(M,\mathbb{C})$ are stable under the action of $\mathcal{L}_T$.
\end{remark}

\begin{remark}\label{remark:real_valued}
The space $C^\infty(M,\mathbb{R})$ is stable under the action of $\mathcal{L}_T$. This fact has the following consequences:
\begin{itemize}
\item If $\lambda$ is a real resonance for $T$, then the space $E_{T,\lambda}$ is preserved by taking the real or the imaginary parts. Hence, $E_{T,\lambda}$ is spanned by the real-valued elements of $E_{T,\lambda}$.
\item If $\lambda$ is a resonance for $T$ then $\bar{\lambda}$ is a resonance for $T$ and $\mathfrak{C} : f \mapsto \bar{f}$  induces an isomorphism between $E_{T,\lambda}$ and $E_{T,\bar{\lambda}}$.
\item If $z \in \mathbb{C}^*$ is not a resonance for $T$ then $R_T(\bar{z}) = \mathfrak{C} \circ R_T(z) \circ \mathfrak{C}$.
\item If $\lambda$ is a resonance for $T$ then $\Pi_{T,\bar{\lambda}} = \mathfrak{C} \circ \Pi_{T,\lambda} \circ \mathfrak{C}$.
\item If $\lambda$ is a real resonance for $T$, then $C^\infty(M,\mathbb{R})$ is stable under the action of $\Pi_{T,\lambda}$. 
\end{itemize}
These facts (in particular the first two) are important because they are limitations on what the resonances and resonant states for $T$ can be. Hence, we will have to keep them in mind when constructing deformation of $T$ with specific resonances and resonant states. Notice for instance that in \S \ref{section:key_lemma} below, where we study the possible deformations of a family of resonances, we impose that his family that does not contain a pair of complex conjugates.
\end{remark}

\subsection{Spectral stability for transfer operators}\label{subsection:linear_response_theory}

As explained in the introduction, we will need to understand how the resonances and resonant states behave under a perturbation of $T$. The difficulty in this matter is that the operator $\mathcal{L}_T$ does not depend smoothly on $T$ as a bounded operator on $C^k(M,\mathbb{C})$ for $k \geq 0$. Since this problem is crucial for linear response theory, dynamicists developed tools to bypass this difficulty. In particular, Gou\"ezel, Keller and Liverani developed an abstract framework that can be used to study the spectral stability of many dynamical sytems.  A textbook presentation of this method can be found in \cite[\S A.3]{baladi_book2} (see \S 2.5 in this book for the application in the case of expanding maps).

We start with a stability statement which is a direct consequence of \cite[Theorem 2.35]{baladi_book2} (see also \cite[Theorem A.4]{baladi_book2}):

\begin{proposition}\label{proposition:stability_resonances}
Let $T_0 \in \Exp(M)$. Let $\delta > 0$. Let $\lambda_1,\dots,\lambda_N$ denote the resonances for $T_0$ of modulus larger than or equal to $\delta$, and $m_1,\dots,m_N$ be their multiplicities. Let $\epsilon \in(0,\delta)$ be such that the disks $\overline{\mathbb{D}}(\lambda_1,\epsilon),\dots,\overline{\mathbb{D}}(\lambda_N,\epsilon)$ are disjoint and for $j = 1,\dots,N$ the only resonance for $T_0$ in $\overline{\mathbb{D}}(\lambda_j,\epsilon)$ is $\lambda_j$. There is a neighbourhood $U$ of $T_0$ in $\Exp(M)$ such that:
\begin{itemize}
\item for every $T \in U_0$, the resonances for $T$ of modulus larger than or equal to $\delta$ are contained in $\bigcup_{j = 1}^N \mathbb{D}(\lambda_j,\epsilon)$;
\item for every $T \in U$ and for $j = 1,\dots,N$ there are exactly $m_j$ resonances (counted with multiplicities) for $T$ in $\mathbb{D}(\lambda_j,\epsilon)$. Moreover, for $j = 1,\dots,N$, the map
\begin{equation}\label{eq:deformation_projector}
T \mapsto \frac{1}{2 i \pi} \int_{\partial \mathbb{D}(\lambda_j,\epsilon)} R_T(z) \mathrm{d}z
\end{equation}
is continuous from $\Exp(M)$ to the space of continuous operators from $C^\infty(M,\mathbb{C})$ to itself\footnote{This space is endowed with the topology of uniform convergence on bounded sets.}.
\end{itemize}
\end{proposition}

Notice that the operator in the right hand side of \eqref{eq:deformation_projector} is the sum of the spectral projectors $\Pi_{T,\lambda}$ for $\lambda$ the resonances of $T$ within $\mathbb{D}(\lambda_j,\epsilon)$.

We will need a sharper result than Proposition \ref{proposition:stability_resonances} in order to understand the behaviour of resonances and resonant states along smooth families of expanding maps. Notice that \cite[Theorem 2.36]{baladi_book2} can be used to this end. However, we will sometimes need slightly more accurate statement, and it will be convenient to state smoothness results using the structure of Fréchet manifold structure of $\Exp(M)$.

Indeed, $\Exp(M)$ is an open subset of $C^\infty(M,M)$. Hence, the standard structure of Fréchet manifold of $C^\infty(M,M)$ (see for instance \cite[Example I.4.1.3]{hamilton_ift}) induces a structure of Fréchet manifold on $\Exp(M)$. If $F \in C^\infty(M,M)$, then the tangent space to $C^\infty(M,M)$ at $F$ identifies with $\Gamma(F^* \mathcal{T}M)$, the space of smooth sections of the pullback of the tangent bundle to $M$ by $F$ (see for instance \cite[Example I.4.3.3]{hamilton_ift}). However, if $F \in \Exp(M)$, since $F$ is a local diffeomorphism, there is an identification between $\Gamma(F^* \mathcal{T}M)$ and $\Gamma(\mathcal{T}M)$: to $Y \in \Gamma(F^* \mathcal{T}M)$ we associate the vector field given by $x \mapsto D F(x)^{-1} \cdot Y(x)$. Hence, there is an identification of $\mathcal{T}_F \Exp(M)$ with $\Gamma(\T M)$. Concretely, if $(F_t)_{t \in (-\epsilon_0,\epsilon_0)}$ is a smooth curve in $\Exp(M)$ with $F_0 = F$, the tangent vector to $(F_t)_{t \in (-\epsilon_0,\epsilon_0)}$ at $t = 0$ is given under this identification by the vector field $x \mapsto DF_0(x)^{-1} \cdot \frac{\mathrm{d}}{\mathrm{d}t} F_t(x)_{t = 0}$.

If $F_0 \in \Exp(M)$, then for $X \in \Gamma(\T M)$, we may define a smooth map $F_X : M \to M$ by
\begin{equation}\label{eq:coordinate_expm}
F_X(x) = F_0(\exp_x(X(x))) \textup{ for } x \in M.
\end{equation}
Here, we use the exponential map associated to the Riemannian metric on $M$. One can then see that the map $X \mapsto F_X$ induces a diffeomorphism between a neighbourhood of $0$ in $\Gamma(\T M)$ and a neighbourhood of $F_0$ in $\Exp(M)$. We could even use such maps to define an atlas on $\Exp(M)$. The identification of $\Gamma(\T M)$ with $\T_{F_0} \Exp(M)$ is then the derivative at $0$ of the map $X \mapsto F_X$.

The identification of each tangent space of $\Exp(M)$ with $\Gamma(\T M)$ induces a parallelization of $\Exp(M)$:
\begin{equation*}
\mathcal{T} \Exp(M) \simeq \Exp(M) \times \Gamma(\T M).
\end{equation*}
This identification is actually an isomorphism of Fréchet vector bundles: using coordinates of the form $X \mapsto F_X$ as defined by \eqref{eq:coordinate_expm}, the identification is given by nonlinear partial differential operators, which are smooth maps \cite[Example I.3.6.6]{hamilton_ift}. All the derivatives computed below are expressed using this identification.

We start by studying the smoothness of the map $T \mapsto \mathcal{L}_T$ itself. To do so, for $T \in \Exp(M)$ and $X \in \Gamma(\T M)$, we introduce the operator $P_T(X) : C^\infty(M,\mathbb{C}) \mapsto C^\infty(M,\mathbb{C})$ defined by
\begin{equation}\label{eq:derivative_transfer}
P_T(X) f = - \mathcal{L}_T(\Div(fX))
\end{equation}
for $f \in C^\infty(M,\mathbb{C})$. We have then the following formula:

\begin{lemma}\label{lemma:explicit_derivative}
The map $(T, f) \mapsto \mathcal{L}_T f$ is smooth\footnote{Here, we use the notion of smoothness for maps between Fréchet manifolds discussed in \cite[I.4.4]{hamilton_ift}.} from $\Exp(M) \times C^\infty(M,\mathbb{C})$ to $C^\infty(M,\mathbb{C})$. The derivative with respect to $T$ of this map is\footnote{Concretely, if $t \mapsto T_t$ is a smooth map from a neighbourhood $I$ of $0$ in $\mathbb{R}$ to $\Exp(M)$ and $f \in C^\infty(M,\mathbb{C})$ then the derivative at $0$ of the map $t \mapsto \mathcal{L}_{T_t} f$ is $P_{T_{0}}(X)f$, where the vector field $X$ is defined by
\begin{equation*}
X(x) = DT_{0}(x)^{-1} \cdot \frac{\mathrm{d}}{\mathrm{d}t} T_{t}(x)_{|t = 0} \textup{ for } x \in M.
\end{equation*}} $(T,f,X) \mapsto P_T(X)f$.
\end{lemma}

\begin{proof}
Let us start by proving that $(T,f ) \mapsto \mathcal{L}_T f$ is continuous. Let $T_0 \in \Exp(M)$. Pick a point $x_0 \in M$. Let $m$ be the degree of $T_0$ and $y_1,\dots,y_m$ be the antecedents of $x_0$ by $T_0$. Let us apply the implicit function theorem to the $C^1$ map
\begin{equation*}
\begin{array}{ccc}
C^1(M,M) \times M & \to & M \\
(T,x)& \mapsto Tx
\end{array}
\end{equation*}
at the points $(T_0,y_1),\dots,(T_0,y_m)$. We find that there is an open neighbourhood $U$ of $T_0$ in $C^1(M,M)$, an open neighbourhood $V$ of $x_0$ in $M$ and open neighbourhoods $V_1,\dots,V_m$ of $y_1,\dots,y_m$ respectively such that for every $x \in V,T \in U$ and $j \in \set{1,\dots,m}$, there is a unique point $G_j(T,x)$ in $V_j$ such that $T(G_j(T,x)) = x$. Moreover, the map $G_j$ is $C^1$ from $U \times W$ to $V_j$. Up to taking $U$ and $V$ smaller, we may assume that for every $x \in W$ and $T \in U$ the points $G_1(T,x),\dots,G_m(T,x)$ are distinct and are all the antecedents of $x$ by $T$. Applying the implicit function theorem to the map $(T,x) \mapsto Tx$ from $C^k(M,M) \times M$ to $M$, we find that for every $k \geq 1$ and $j \in \set{1,\dots,m}$ the restriction of $G_j$ to $(U \cap C^k(M,M)) \times W$ is $C^k$.

Let $U_0 = U \cap \Exp(M)$ (provided $U$ is small enough, this is also $U \cap C^\infty(M,M)$). If $T \in U_0, f \in C^\infty(M,\mathbb{C})$ and $x \in W$, then we have
\begin{equation}\label{eq:local_transfer}
\mathcal{L}_{T} f(x) = \sum_{j = 1}^{m} f(G_j(T,x)) |\det D_x G_j(T,x)|.
\end{equation}
If we use this formula and the chain rule to compute the derivatives of order $k \in \mathbb{N}$ (in any coordinates system) of $\mathcal{L}_T f$ at a point $x$ in $W$, we get a polynomial in the derivatives of order up to $k$ of $f$ evaluated at the $G_j(T,x)$'s, the derivative with respect to $x$ of the $G_j$'s and some smooth functions that do not depend on $x$ and $T$ (they depend on the Riemannian metric on $M$). All these quantities depend continuously on $f$ and $T$ in the $C^k$ topology. Covering $M$ by a finite number of open sets with the properties of $W$ (only changing $x_0$), we find (up to reducing the size of $U_0$) that $(T,f) \mapsto \mathcal{L}_T f$ is continuous from $U_0 \times C^\infty(M,\mathbb{C})$ to $C^k(M,\mathbb{C})$. Since $k$ and $T_0$ are arbitrary, we find that $(T,f) \mapsto \mathcal{L}_T f$ is continuous from $\Exp(M) \times C^\infty(M,\mathbb{C})$ to $C^\infty(M,\mathbb{C})$.

Let us now consider a $C^1$ map $t \mapsto T_t$ from an open interval $I \subseteq \mathbb{R}$ containing $0$ to $\Exp(M)$. Let $f \in C^\infty(M,\mathbb{C})$. We want to prove that $t \mapsto \mathcal{L}_{T_t} f \in C^\infty(M,\mathbb{C})$ is differentiable at $0$ and compute its derivative. To do so, we fix a point $x_0 \in M$ and apply the construction from the previous paragraphs to $T_0$. The formula \eqref{eq:local_transfer} becomes, for $x \in W$ and $t$ near $0$:
\begin{equation*}
\mathcal{L}_{T_t}f(x) = \sum_{j = 1}^m f(g_j(t,x))|\det D_x g_j(t,x)|,
\end{equation*}
where the $g_j$'s are defined by $G_j(T_t,x)$. Notice that the $g_j$'s are $C^\infty$ functions on $I \times W$ (up to making $I$ smaller). Hence, we may apply Taylor's formula to find for $t$ small and $x \in W$ that
\begin{equation}\label{eq:taylor_local}
\mathcal{L}_{T_t} f(x) = \mathcal{L}_{T_0}f(x) + t P_{T_0}(X)f(x) + t^2 Q_t f(x).
\end{equation}
Here, $X$ is the vector field $y \mapsto D T_0(y)^{-1} \cdot \frac{\mathrm{d}}{\mathrm{d}t} T_t(y)_{|t = 0}$ and $Q_t f(x)$ is the integral remainder in Taylor's formula. This is an integral between $0$ and $1$ that involves the derivatives of $f$ up to order $2$ and some smooth functions of $t$ and $x$. Hence, the derivatives of $Q_t f$ of order up to $k$ are controlled by the derivatives of $f$ of order up to $k+2$. It follows then from \eqref{eq:taylor_local} that $t \mapsto \mathcal{L}_{T_t} f$ is differentiable at $0$ with derivative $P_{T_0}(X)f$ (as above, we work with a finite number of open sets that cover the manifold and on which we have a formula \eqref{eq:taylor_local}).

We proved that $(T,f) \mapsto \mathcal{L}_{T}f$ is continuous and has a partial derivative with respect to $T$ given by $(T,f,X) \mapsto P_T(X)f$. However, recalling the definition \eqref{eq:derivative_transfer}, we deduce from the continuity of $(T,f) \mapsto \mathcal{L}_{T}f$ that the map $(T,f,X) \mapsto P_T(X)f$ is continuous. Hence, it follows from \cite[Corollary I.3.4.4]{hamilton_ift} that the map $(T,f) \mapsto \mathcal{L}_{T}f$ is $C^1$. Notice that the operator $(X,f) \mapsto \Div(fX)$ is smooth (see for instance \cite[Example I.3.6.6]{hamilton_ift}). In view of the formula \eqref{eq:derivative_transfer} for the derivative of $(T,f) \mapsto \mathcal{L}_T f$, we find that if $(T,f) \mapsto \mathcal{L}_T f$ is $C^k$ for some $k \geq 1$ then it is $C^{k+1}$. It follows by induction that $(T,f) \mapsto \mathcal{L}_T f$ is smooth.
\end{proof}

A key ingredient to understand the perturbations of resonances and resonant states is the understanding of the perturbations of the resolvent $R_T(z)$.

\begin{lemma}\label{lemma:useful_linear_response}
Let $T_0 \in \Exp(M)$. Let $V$ be an open relatively compact subset of $\mathbb{C}^*$ such that $\overline{V} \cap \res(T_0) = \emptyset$. There is an open neighbourhood $U$ of $T_0$ in $\Exp(M)$ such that for every $T \in U$ the map $T$ has no resonance in $\overline{V}$. Moreover, the map $(T,z,f) \mapsto R_T(z) f$ is smooth from $U \times V \times C^\infty(M,\mathbb{C})$ to $C^\infty(M,\mathbb{C})$ and its derivative with respect to $T$ is given by $(T,z,f,X) \mapsto R_T(z) P_T(X) R_T(z)f$.
\end{lemma}

\begin{proof}
The existence of the set $U$ is a direct consequence of Proposition \ref{proposition:stability_resonances}. It follows from Lemma \ref{lemma:explicit_derivative} and \cite[Theorem I.5.3.1]{hamilton_ift} that, in order to prove the smoothness of $(T,z,f) \mapsto R_{T}(z)f$, we only need to prove that this map is continuous (from $U \times V \times C^\infty(M,\mathbb{C})$ to $C^\infty(M,\mathbb{C})$). The formula for the derivative is then obtained by considering a smooth curve $t \mapsto T_t$ in $U$ and differentiating the relation $(z - \mathcal{L}_{T_t}) R_{T_t}(z) f = f$ with respect to $t$.

In order to prove the continuity of $(T,z,f) \mapsto R_{T}(z)f$, let us mention that for $T_1,T_2 \in U, z_1,z_2 \in \overline{V}$ and $f_1,f_2 \in C^\infty(M,\mathbb{C})$, we have
\begin{equation*}
\begin{split}
R_{T_1}(z_1) f_1 - R_{T_2}(z_2)f_2 & = R_{T_2}(z_2) (f_1-f_2) + (z_2 - z_1) R_{T_2}(z_1) R_{T_2}(z_2) f_1 \\ & \qquad \qquad + R_{T_1}(z_1)(\mathcal{L}_{T_1} - \mathcal{L}_{T_2}) R_{T_2}(z_1)f_1. 
\end{split}
\end{equation*}
Considering this relation and Lemma \ref{lemma:explicit_derivative}, we see that we only need to prove that for every $T_1 \in U$ $z_1 \in V$ and $k \in \mathbb{N}$ there are neighbourhoods $\widetilde{U}$ and $\widetilde{V}$ respectively of $T_1$ in $U$ and $z_1 \in V$ and constants $C > 0$ and $\ell \in \mathbb{N}$ such that for every $T_2 \in \widetilde{U}, z_2 \in \widetilde{V}$ and $f \in C^\infty(M,\mathbb{C})$ we have $\n{R_{T_2}(z_2)f}_{C^k} \leq C \n{f}_{C^\ell}$.

Such a bound is given in the proof of \cite[Theorem 2.35]{baladi_book2}, but using Sobolev spaces instead of $C^k$ spaces. This difference is not an issue thanks to Sobolev injections. The required bound is established in the beginning of the proof of \cite[Theorem 2.35]{baladi_book2}, when the author applies \cite[Theorem A.4]{baladi_book2} with $N = 1$ (see the penultimate bound in the statement of this theorem).
\end{proof}

From Lemma \ref{lemma:useful_linear_response} we deduce a smoothness result (see \cite[Theorem 2.36]{baladi_book2} for a more complete statement but using another language).

\begin{proposition}\label{proposition:smoothness_spectral_projector}
In the setting of Proposition \ref{proposition:stability_resonances}, for $j = 1,\dots,N$ the map
\begin{equation*}
(T,f) \mapsto \frac{1}{2 i \pi} \int_{\partial \mathbb{D}(\lambda_j,\epsilon)} R_{T}(z)f \mathrm{d}z
\end{equation*}
is smooth from $U \times C^\infty(M,\mathbb{R})$ to $C^\infty(M,\mathbb{C})$.
\end{proposition}

\begin{proof}
In order to deduce Proposition \ref{proposition:smoothness_spectral_projector} from Lemma \ref{lemma:useful_linear_response}, one only needs a result of differentiation under the integral that applies in this context. Such a result can be deduced from \cite[Lemma I.3.3.1]{hamilton_ift}.
\end{proof}

\begin{remark}\label{remark:useful_smoothness}
In Proposition \ref{proposition:smoothness_spectral_projector} the notion of smoothness that we use is the one from \cite[I.4.4]{hamilton_ift}. However, we have the following consequence in terms of the standard notion of smoothness of maps between open sets of Banach spaces. Let $j \in \set{1,\dots,N}$. If $V$ is an open subset of $\mathbb{R}^n$ and $t \mapsto F_t$ is a smooth map from $V$ to $U$, then for every $f \in C^\infty(M,\mathbb{C})$ and $k \in \mathbb{N}$, the map
\begin{equation}\label{eq:smooth_family}
t \mapsto \frac{1}{2 i \pi} \int_{\partial \mathbb{D}(\lambda_j,\epsilon)} R_{F_t}(z)f \mathrm{d}z
\end{equation}
is smooth from $V$ to $C^k(M,\mathbb{C})$. To prove this fact, just notice that Proposition \ref{proposition:smoothness_spectral_projector} and \cite[Theorem I.3.6.4]{hamilton_ift} implies that the map \eqref{eq:smooth_family}, seen as a map from $V$ to $C^k(M,\mathbb{C})$ as continuous partial derivatives of any order.

From the smoothness of \eqref{eq:smooth_family} we may deduce another handy result: the sum of the resonances for $F_t$ in $\mathbb{D}(\lambda_j,\epsilon)$ counted with multiplicities is a smooth function of $t \in V$. Let us denote the operator in the right hand side of \eqref{eq:smooth_family} by $\Pi(t)$. Let $t_0 \in V$. Let $f_1,\dots,f_m$ be a basis of the range of $\Pi(t_0)$ (with $m = m_j$). Let $l_1,\dots,l_m$ be measures on $M$ such that $l_i(f_k) = \delta_{i,k}$ for $i,k \in \set{1,\dots,m}$ (the existence of such measures is guaranteed by the Hahn--Banach theorem). For $t \in V$ and $i \in \set{1,\dots,m}$ let $g_{i,t} = \Pi(t) f_i$. Notice that the $g_{i,t}$'s depend smoothly on $t$ as continuous functions, hence the matrix $N(t) = (l_i(g_{k,t}))_{1 \leq i,k \leq t}$ is a smooth function of $t \in V$. In particular, there is a neighbourhood $V_0$ of $t_0$ such that for every $t \in V_0$ the matrix $N(t)$ is invertible. Consequently, for $t \in V_0$ the functions $g_{1,t},\dots,g_{m,t}$ are linearly independent and thus form a basis of the range of $\Pi(t)$ (i.e. of the sum of the characteristic spaces for $\mathcal{L}_{F_t}$ associated to resonances in $\mathbb{D}(\lambda_j,\epsilon)$). Hence, there is a matrix $A(t) = (a_{i,k}(t))_{1 \leq i,k \leq m}$ such that for every $t \in V_0$ and $i \in \set{1,\dots,m}$ we have
\begin{equation*}
\mathcal{L}_{F_t} g_{i,t} = \sum_{k = 1}^m a_{i,k}(t) g_{k,t}.
\end{equation*}
Notice that $A(t)$ is just the transpose of the matrix of the operator induced by $\mathcal{L}_{F_t}$ on the range of $\Pi(t)$. Hence, the eigenvalues of $A(t)$ are the resonances for $\mathcal{L}_{F_t}$ in $\mathbb{D}(\lambda_j,\epsilon)$. For $i = 1,\dots,m$, since $t \mapsto g_{i,t}$ is smooth from $V$ to $C^\infty(M,\mathbb{C})$, it follows from Lemma \ref{lemma:explicit_derivative} that the map $t \mapsto \mathcal{L}_{F_t} g_{i,t}$ is smooth from $V$ to $C^0(M,\mathbb{C})$ (for instance). Hence, the matrix $Q(t) = (l_i(\mathcal{L}_{F_t} g_{k,t}))_{1 \leq i,k \leq t}$ is a smooth function of $t \in V$. For $t \in V_0$, we have $Q(t) = A(t) N(t)^\top$ and thus $A(t) = Q(t)(N(t)^\top)^{-1}$, which proves that $A(t)$ is a smooth function of $t \in V_0$. Hence, the sum of the resonances for $F_t$ in $\mathbb{D}(\lambda_j,\epsilon)$, which is just the trace of $A(t)$, is a smooth function of $t$.
\end{remark}

\section{Generic simplicity of resonances (Theorem \ref{theorem:generic_simple})}\label{section:simplicity}

This section is dedicated to the proof of Theorem \ref{theorem:generic_simple}. To do so, we start by studying coresonant states. The following result implies that the coresonant states associated to a smooth expaniding maps have full support. A similar result is proven in the Anosov case in \cite{weich_support}. Notice however that in the Anosov case there is a symmetry between resonant and coresonant states, and thus both of them are fully supported. However, in the expanding case it it not clear whether the resonant states are always fully supported or not (it follows from Theorem \ref{theorem:generic_morse} that the resonant states for a generic smooth expanding maps are fully supported).

\begin{proposition}\label{proposition:full_support}
Let $T \in \Exp(M)$. Let $\nu$ be a coresonant state for $T$ associated to a resonance $\lambda$ not equal to $1$. Let $U$ be an open subset of $M$. There is $\varphi \in C_0^\infty(M,\mathbb{R})$ supported in $U$ such that $\nu(\varphi) \neq 0$.
\end{proposition}

The proof of Proposition \ref{proposition:full_support} requires the following standard lemma. We recall its proof since we will also need this result in the proof of Lemma \ref{lemma:linear_algebra} below.

\begin{lemma}\label{lemma:expanding_open_sets}
Let $T \in \Exp(M)$. Let $U$ be a non-empty open subset of $M$. There is $N \geq 0$ such that $T^N(U) = M$.
\end{lemma}

\begin{proof}
Choose a point $x_0 \in U$. Let $x$ be a point of $M$. For each $N \geq 1$, let us choose a $C^1$ path $\gamma_N : [0,1] \to M$ from $T^N(x_0)$ to $x$ (recall that $M$ is connected) of length less than $2 \diam M$. Since $T: M \to M$ is expanding, it is a local diffeomorphism and, since $M$ is compact, it follows that $T$ is a covering. Consequently for each $N \geq 1$, there is a unique $C^1$ path $c_N$ starting at $x_0$ such that $T^N \circ c_N= \gamma_N$. With $C$ and $\theta$ the constants from the definition \eqref{eq:definition_expanding} of an expanding map, we notice that the length of $c_N$ is smaller than $2 C^{-1} e^{-N \theta} \diam M$. Since $U$ is open we have $d(x_0, M \setminus U) >0$, and we find that for $N$ strictly larger than $\theta^{-1}\log\left( \frac{2 \diam M}{C d(x_0, M \setminus U)} \right)$ the point $c_N(1)$ belongs to $U$ and thus $x = T^N(c_N(1)) \in T^N(U)$.
\end{proof}

\begin{proof}[Proof of Proposition \ref{proposition:full_support}]
Let $A$ be the set of all open subsets $V$ of $M$ such that for every $\varphi \in C_0^\infty(M,\mathbb{R})$ supported in $V$ we have $\nu(\varphi) = 0$. Let $\mathcal{V}$ be the union of the elements of $A$. By a partition of unity argument, we find that $\mathcal{V} \in A$. Let $x_0$ be a point in $\mathcal{V}$. By the Inverse Function Theorem, there is a neighbourhood $V_0 \subseteq \mathcal{V}$ of $x_0$ such that $T$ induces a diffeomorphism from $V_0$ to $T(V_0)$. Let $\varphi \in C_0^\infty(M,\mathbb{R})$ be supported in $T(V_0)$ and define the function $\psi : M \to \mathbb{R}$ by
\begin{equation*}
\psi(x) = \begin{cases} \varphi(T x) |\det DT(x)| & \textup{ if } x \in V_0, \\
                        0 & \textup{ if } x \in M \setminus V_0. \end{cases}
\end{equation*}
Notice that $\psi \in C_0^\infty(M,\mathbb{R})$ is supported in $V_0$ and $\mathcal{L}_T\psi = \varphi$. Hence, we have $\nu(\psi) = 0$ and thus $\nu(\varphi) = \lambda \nu(\psi) = 0$. Consequently, $T(V_0) \in A$, which proves that $T(x_0) \in \mathcal{V}$.

We just showed that $T(\mathcal{V}) \subseteq \mathcal{V}$. It follows then from Lemma \ref{lemma:expanding_open_sets} that $\mathcal{V} = \emptyset$ or $\mathcal{V} = M$. If $\mathcal{V} = M$, then $\nu$ is identically zero on $C^\infty_0(M,\mathbb{C})$, and thus $\nu$ is a multiple of the distribution $f \mapsto \int_M f \mathrm{d}x$, which contradicts $\lambda \neq 1$. We must consequently have $\mathcal{V} = \emptyset$, proving the result.
\end{proof}

We are now ready to prove Theorem \ref{theorem:generic_simple}.

\begin{proof}[Proof of Theorem \ref{theorem:generic_simple}]
Let $\mathcal{U}_\delta$ denote the set of $T \in \Exp(M)$ such that all resonances of $T$ of moduli larger than or equal to $\delta$ are simple. It follows from Proposition \ref{proposition:stability_resonances} that $\mathcal{U}_\delta$ is open, so we only need to prove that it is dense. Let $U$ be a non-empty open subset of $\Exp(M)$, we want to prove that $\mathcal{U}_\delta \cap U \neq \emptyset$. 

For $T \in U$, let $n(T)$ be the number of resonances for $T$ in $\mathbb{C} \setminus \mathbb{D}(0,\delta)$ counted with multiplicities minus the number of resonances for $T$ in $\mathbb{C} \setminus \mathbb{D}(0,\delta)$ counted without multiplicities. I.e. each resonance for $T$ in $\mathbb{C} \setminus \mathbb{D}(0,\delta)$ contribute to $n(T)$ by its multiplicity minus one. Notice that $T$ belongs to $\mathcal{U}_\delta$ if and only if $n(T) = 0$. Let $n_0$ be the minimal value of $n(T)$ for $T \in U$ and consider $\widetilde{U} = \set{T \in U : n(T) = n_0}$. It follows from Proposition \ref{proposition:stability_resonances} that $\widetilde{U}$ is open. 

Let us assume by contradiction that $n_0 > 0$. For $T \in \widetilde{U}$, let $N(T)$ denote the size of the largest Jordan block for a resonance for $T$ in $\mathbb{C} \setminus \mathbb{D}(0,\delta)$. I.e. $N(T)$ is the largest integer such that there is $\lambda \in \mathbb{C} \setminus \mathbb{D}(0,\delta)$ and $f \in C^\infty(M,\mathbb{C}) \setminus \set{0}$ such that $(\mathcal{L}_T - \lambda)^{N(T)-1}f \neq 0$ and $(\mathcal{L}_T - \lambda)^{N(T)}f = 0$. The number $N(T)$ is well-defined since our assumption $n_0 > 0$ implies that any $T \in \widetilde{U}$ has at least one resonance in $\mathbb{C} \setminus \mathbb{D}(0,\delta)$. Let $N_0$ be the maximal value of $N(T)$ for $T \in \widetilde{U}$ (which is well-defined beacuse $N(T)$ is always less than $1 + n_0$).

Consider now $T_0 \in \widetilde{U}$ such that $N(T_0) = N_0$. Let $\lambda_0 \in \mathbb{C} \setminus \mathbb{D}(0,\delta)$ and $f_0 \in C^\infty(M,\mathbb{C}) \setminus \set{0}$ be such that $(\mathcal{L}_{T_0} - \lambda_0)^{N_0 - 1} f_0 \neq 0$ and  $(\mathcal{L}_{T_0} - \lambda_0)^{N_0} f_0 = 0$. If $N_0 = 1$, we may impose in addition that the multiplicity $m_0$ of $\lambda_0$ as a resonance of $T_0$ is at least $2$ (since $n_0 > 0$). Let then $\nu$ be a coresonant state for $T_0$ associated to $\lambda_0$. When $N_0 = 1$, we may impose that $\nu(f_0) = 0$ (because the kernel of $\mathcal{L}_{T_0}^* - \lambda_0$ has dimension at least $2$). Let $g_0 = (\mathcal{L}_{T_0} - \lambda_0)^{N_0 - 1} f_0$ and choose $x_0 \in M$ such that $g_0(x_0) \neq 0$. Let then $V_0$ be an open neighbourhood of $x_0$ in $M$ diffeomorphic to a ball and such that $T_0$ induces a diffeomorphism from $V_0$ to $T_0(V_0)$ and $g_0$ does not vanish on $V_0$. According to Proposition \ref{proposition:full_support}, there is $\varphi \in C_0^\infty(M,\mathbb{R})$ supported in $T_0(V_0)$ such that $\int_M \varphi \mathrm{d}x = 0$ and $\nu(\varphi) \neq 0$ (we must have $\lambda_0 \neq 1$ since $1$ is always a simple resonance). Define a function $\psi : M \to \mathbb{R}$ by
\begin{equation*}
\psi(x) = \begin{cases} \varphi(T_0(x)) |\det T_0(x)| & \textup{ if } x \in V_0, \\ 0 & \textup{ if } x \in M \setminus V_0. \end{cases}
\end{equation*}
Notice that $\psi$ is smooth, supported in $V_0$ and satisfies $\mathcal{L}_{T_0} \psi = \varphi$. In particular, we have $\int_M \psi \mathrm{d}x = 0$. Since $T_0(V_0)$ is diffeomorphic to an open ball, there is a vector field $Y$ on $M$, supported in $V_0$, and such that $\psi = - \Div(Y)$. Let then $X$ be the vector field on $M$ defined by $X = g_0^{-1} Y$, which makes sense since $g_0$ does not vanish on the support of $Y$.

We let $(\phi_t)_{t \in \mathbb{R}}$ be the flow of $X$ and define a smooth family $(T_t)_{t \in \mathbb{R}}$ of maps from $M$ to itself by $T_t = T_0 \circ \phi_t$. Let then $\epsilon > 0$ be such that the only resonance of $T_0$ in $\overline{\mathbb{D}}(\lambda_0,\epsilon)$ is $\lambda_0$. It follows from Proposition \ref{proposition:stability_resonances} for $t$ small the number of resonances for $T_t$ in $\overline{\mathbb{D}}(\lambda_0,\epsilon)$ is $m_0$. It implies that $T_t$ has exactly one resonance $\lambda_t$ in $\mathbb{D}(\lambda_0,\epsilon)$ that has multiplicity $m_0$: otherwise the total contribution of the resonances in $\mathbb{D}(\lambda_0,\epsilon)$ to $n(T_t)$ would be strictly less than $m_0 - 1$, and since the contribution of the other resonances cannot increase (still by Proposition \ref{proposition:stability_resonances}), we would contradict the minimality of $n_0$. Notice that for $t$ small, the sum of the resonances for $T_t$ in $\mathbb{D}(\lambda_0,\epsilon)$ is $m_0 \lambda_t$. It follows consequently from Remark \ref{remark:useful_smoothness} that $t \mapsto \lambda_t$ is smooth on a neighbourhood of $0$. Let us define for $t$ small $f_t = \frac{1}{2 i \pi}\int_{\partial \mathbb{D}(\lambda_0,\epsilon)} R_{T_t}(z)f_0 \mathrm{d}z$ (notice that when $t = 0$, we retrieve indeed $f_0$). This is a generalized resonant state for $T_t$ associated to $\lambda_t$, and it depends smoothly on $t$ in any $C^k$ space according to Remark \ref{remark:useful_smoothness}. By continuity, we must have $(\mathcal{L}_{T_t} - \lambda_t)^{N_0 - 1} f_t \neq 0$ for $t$ small. We must also have $(\mathcal{L}_{T_t} - \lambda_t)^{N_0} f_t = 0$ (otherwise, we would contradict the maximality of $N_0$). Differentiating this last relation, we get
\begin{equation}\label{eq:differentiation_jordan}
\begin{split}
& \sum_{j = 0}^{N_0-1} (\mathcal{L}_{T_0} - \lambda_0)^{j}( P_{T_0}(X) - \frac{\mathrm{d}}{\mathrm{d}t}(\lambda_t)_{|t = 0})(\mathcal{L}_{T_0} - \lambda_0)^{N_0 - 1 - j} f_0 \\ & \qquad \qquad \qquad \qquad \qquad \qquad \qquad \qquad + (\mathcal{L}_{T_0} - \lambda_0)^{N_0} \frac{\mathrm{d}}{\mathrm{d}t}(f_t)_{|t = 0} = 0.
\end{split}
\end{equation}
Now, since $\nu$ is a coresonant state for $T_0$ associated to $\lambda_0$, we deduce from \eqref{eq:differentiation_jordan} that $\nu(P_{T_0}(X)(\mathcal{L}_{T_0} - \lambda_0)^{N_0 - 1} f_0) = 0$ but
\begin{equation*}
P_{T_0}(X)(\mathcal{L}_{T_0} - \lambda_0)^{N_0 - 1} f_0 = P_{T_0}(X) g_0 = - \mathcal{L}_{T_0}(\Div(g_0 X)) = \mathcal{L}_{T_0} \psi = \varphi.
\end{equation*}
This is a contradiction with $\nu(\varphi) \neq 0$.
\end{proof}

\section{A key deformation lemma}\label{section:key_lemma}

The point of this section is to explain how to produce enough interesting perturbations of resonant states for the proofs of Theorems \ref{theorem:generic_morse}, \ref{theorem:generic_density}, \ref{theorem:generic_real_resonance} and \ref{theorem:generic_complex_resonance}. The main result in this regard is Lemma \ref{lemma:key_deformation}, but see also Lemma \ref{lemma:derivative_simple_resonance} and Remarks \ref{remark:derivative_simple_resonance} and \ref{remark:key_deformation} for more context. 

We will fix some notations for the whole section. Let us consider a map $F_0$ in $\Exp(M)$. Let $\lambda_1,\dots,\lambda_p$ be distinct simple real resonances for $F_0$ and $\lambda_{p+1},\dots,\lambda_{p+q}$ be distinct simple non-real resonances for $F_0$. Assume in addition that $\lambda_j \neq \bar{\lambda}_k$ if $j,k \in \set{p+1,\dots,p+q}$.

Choose $\epsilon > 0$ such that the disks $\overline{\mathbb{D}}(\lambda_1,\epsilon), \dots , \overline{\mathbb{D}}(\lambda_{p+q},\epsilon)$ are disjoint and do not contain other resonances than $\lambda_1,\dots,\lambda_{p+q}$. According to Proposition \ref{proposition:stability_resonances}, there is a connected neighbourhood $U_0$ of $F_0$ in $\Exp(M)$ such that for every $T \in U_0$ and $j \in \set{1,\dots,p+q}$ there is exactly one resonance $\lambda_j(T)$ for $T$ in $\mathbb{D}(\lambda_j,\epsilon)$. Moreover, the resonance $\lambda_j(T)$ is simple, and it is real if and only if $j \leq p$.

If $j \in \set{1,\dots,p+q}$, let $f_{F_0,j}$ be a resonant state for $F_0$ associated to $\lambda_j$. If $j \leq p$, we choose $f_{F_0,j}$ real-valued. Then for $T \in U_0$, we set $f_{T,j} = \Pi_{T,\lambda_j(T)} f_{F_0,j}$. Up to taking $U_0$ smaller, we may assume that $f_{T,j}$ is non-zero for every $T \in U_0$ (and thus that it is a resonant state for $T$). For $T \in U_0$ and $j \in \set{1,\dots,p+q}$, we let $\nu_{T,j}$ be the coresonant state for $T$ associated to $\lambda_j$ such that\footnote{We have then $\Pi_{T,\lambda_j(T)}: f \mapsto \nu_{T,j}(f) f_{T,j}$.} $\nu_{T,j}(f_{T,j}) = 1$. Notice that $f_{T,j}$ depends smoothly on $T$ (according to Proposition \ref{proposition:smoothness_spectral_projector}) and that $f_{T,j}$ is real-valued when $j \leq p$, see Remark \ref{remark:real_valued}. The goal of this section is to produce curves $T \mapsto T_t$ in $U_0$ for which the derivative at $0$ of $t \mapsto f_{T_t,j}$ takes some specific values (see Lemma \ref{lemma:key_deformation} below).

We will start with an explicit formula for the derivative of $t \mapsto f_{T_t,j}$ at zero that we will also use in the proof of Lemma \ref{lemma:local_solvability}. We need further notation. If $T \in \Exp(M)$ and $\lambda$ is a simple resonance for $T$ we have
\begin{equation}\label{eq:first_step_expansion_resolvant}
R_T(z) \underset{z \to \lambda}{=} \frac{\Pi_{T,\lambda}}{z - \lambda} + H_{T,\lambda} + \mathcal{O}(|z - \lambda|)
\end{equation}
for some continuous operator $H_{T,\lambda} :C^\infty(M,\mathbb{C}) \to C^\infty(M,\mathbb{C})$. We have then the relations
\begin{equation}\label{eq:full_inverse}
I = \Pi_{T,\lambda} + (\lambda - \mathcal{L}_T)H_{T,\lambda} \textup{ and } H_{T,\lambda} \Pi_{T,\lambda} = 0.
\end{equation}
Using the operators from \eqref{eq:first_step_expansion_resolvant}, we can compute the derivative of $t \mapsto f_{T_t,j}$.

\begin{lemma}\label{lemma:derivative_simple_resonance}
Let $(T_t)_{t \in (-\epsilon_0,\epsilon_0)}$ be a smooth family of elements of $U_0$. Let $j \in \set{1,\dots,p+q}$. Let $X$ be the vector field on $M$ defined by $X(x) = D T_0(x)^{-1} \cdot \frac{\mathrm{d}}{\mathrm{d}t}T_t(x)_{|t = 0}$. Then\footnote{The first derivative holds in $C^\infty(M,\mathbb{C})$ according to Proposition \ref{proposition:smoothness_spectral_projector}.}:
\begin{equation}\label{eq:derivative_simple_resonance}
\frac{\mathrm{d}}{\mathrm{d}t}(f_{T_t,j})_{t = 0} = H_{T_0,\lambda_j(T_0)} P_{T_0}(X) f_{T_0,j} + \Pi_{T_0,\lambda_j(T_0)} P_{T_0}(X) H_{T_0,\lambda_j(T_0)} f_{F_0,j}
\end{equation}
and
\begin{equation*}
\frac{\mathrm{d}}{\mathrm{d}t}\lambda_j(T_t)_{|t= 0} = \nu_{T_0,j}(P_{T_0}(X)f_{T_0,j}).
\end{equation*}
If there is $N \geq 0$ such that $\mathcal{L}_{T_0}^N P_{T_0}(X) f_{T_0,j} = 0$, then the formulae above simplify into
\begin{equation*}
\begin{split}
\frac{\mathrm{d}}{\mathrm{d}t}(f_{T_t,j})_{t = 0} = & \sum_{k= 0}^{N-1} \lambda_j(T_0)^{-k-1} \mathcal{L}_{T_0}^k P_{T_0}(X) f_{T_0,j} \\ & \qquad \qquad \qquad \qquad \qquad + \Pi_{T_0,\lambda_j(T_0)} P_{T_0}(X) H_{T_0,\lambda_j(T_0)} f_{F_0,j}
\end{split}
\end{equation*}
and $\frac{\mathrm{d}}{\mathrm{d}t}\lambda_j(T_t)_{|t= 0} = 0$.
\end{lemma}

\begin{remark}\label{remark:derivative_simple_resonance}
Observe that the term $\Pi_{T_0,\lambda_j(T_0)} P_{T_0}(X) H_{T_0,\lambda_j(T_0)} f_{F_0,j}$ that appears in the formula for the derivative $\frac{\mathrm{d}}{\mathrm{d}t}(f_{T_t,j})_{t = 0}$ is equal to zero when $T_0 = F_0$ or $\lambda_j = 1$. Notice also that this term is a multiple of $f_{T_0,j}$, hence it does not really matter. Indeed, the object we care about is not the resonant state $f_{T_t,j}$ but the one-dimensional vector space $E_{T_t,\lambda_j(T_t)}$ that it spans. If $\lambda_j \neq 1$, then it follows from Proposition \ref{proposition:smoothness_spectral_projector} that the map $t \mapsto E_{T_t,\lambda_j(T_t)}$ is smooth from $(- \epsilon_0,\epsilon_0)$ to $PC_0^\infty(M,\mathbb{C})$. The tangent space to $PC_0^\infty(M,\mathbb{C})$ at $f_{T_0,j}$ identifies with $C_0^\infty(M,\mathbb{C}) / \langle f_{T_0,j} \rangle$ (see \S \ref{subsection:tangent_spaces}), and our formula then tells that the derivative of $t \mapsto E_{T_t,\lambda_j(T_t)}$ at $0$ identifies with the class of $H_{T_0,\lambda_j(T_0)} P_{T_0}(X) f_{T_0,j}$.
\end{remark}

\begin{proof}[Proof of Lemma \ref{lemma:derivative_simple_resonance}]
Recall the formula 
\begin{equation*}
f_{T_t,j} = \frac{1}{2 i \pi} \int_{\partial \mathbb{D}(\lambda_j,\epsilon)} R_{T_t}(z) f_{F_0,j} \mathrm{d}z.
\end{equation*}
It follows from Lemma \ref{lemma:useful_linear_response} and differentiation under the integral (that we can apply in $C^k(M,\mathbb{C})$ for each $k \in \mathbb{N}$) that $t \mapsto f_{T_t,j}$ is smooth from $(-\epsilon_0,\epsilon_0)$ to $C^\infty(M,\mathbb{C})$ with derivative at $0$:
\begin{equation*}
\begin{split}
\frac{1}{2 i \pi}\int_{\partial \mathbb{D}(\lambda_j,\epsilon)} R_{T_0}(z) P_{T_0}(X) R_{T_0}(z)f \mathrm{d}z & = H_{T_0,\lambda_j(T_0)} P_{T_0}(X) \Pi_{T_0,\lambda_j(T_0)}f_{F_0,j}  \\ & \qquad + \Pi_{T_0,\lambda_j(T_0)} P_{T_0}(X) H_{T_0,\lambda_j(T_0)}f_{F_0,j}.
\end{split}
\end{equation*}
We used \eqref{eq:first_step_expansion_resolvant} for the residue computation.

Differentiating with respect to $t$ the relation $\mathcal{L}_{T_t} f_{T_t,j} = \lambda_j(T_t) f_{T_t,j}$, we get
\begin{equation*}
P_T(X) f_{T_0,j} = (\lambda_j(T_0) - \mathcal{L}_{T_0}) \frac{\mathrm{d}}{\mathrm{d}t}(f_{T_t,j})_{t = 0}  + \frac{\mathrm{d}}{\mathrm{d}t}\lambda_j(T_t)_{|t= 0} f_{T_0,j}.
\end{equation*}
Applying $\nu_{T_0,j}$ to this relation, we get the formula for $\frac{\mathrm{d}}{\mathrm{d}t}\lambda_j(T_t)_{|t= 0}$. 

We deal now with the case in which there is an integer $N \geq 0$ such that $\mathcal{L}_{T_0}^N P_{T_0}(X)f_{T_0,j} = 0$. By the analytic continuation principle, we find that $R_{T_0}(z) P_{T_0}(X) f_{T_0,j} = \sum_{k = 0}^{N-1} z^{-k-1} \mathcal{L}_{T_0}^{k} P_{T_0}(X) f_{T_0,j}$ for $z \in \mathbb{C}^* \setminus \res(T_0)$. Comparing with \eqref{eq:first_step_expansion_resolvant}, we find that $$H_{T_0,\lambda} P_{T_0}(X) f_{T_0,j} = \sum_{k = 0}^{N-1} \lambda_j(T_0)^{-k-1} \mathcal{L}_{T_0}^{k} P_{T_0}(X) f_{T_0,j}$$ and $\Pi_{T_0,\lambda} P_{T_0}(X) f_{T_0,j} = 0$, from which the result follows (the second relation implies $\nu_{T_0,j}(P_{T_0}(X)f_{T_0,j}) = 0$).
\end{proof}

We are now ready to state a lemma that we will use to produce many useful perturbations of resonant states in the proofs of our different genericity results.

\begin{lemma}\label{lemma:key_deformation}
Let $T \in U_0$. Let $m \geq 1$ be an integer. Let $x_1,\dots,x_{m+1}$ be distinct points in $M$ such that $T(x_1) = \dots = T(x_{m+1})$. Let $E$ be a subset of $T^{-1}(\set{x_1})$ such that $T^{-1}(\set{x_1}) \setminus E$ contains at least a point that is not periodic for $T$. Let $(\mathfrak{f}_{j,k})_{\substack{1 \leq j \leq p \\ 1 \leq k \leq m}}$ be a family of elements of $C^\infty(M,\mathbb{R})$ and $(\mathfrak{f}_{j,k})_{\substack{p+1 \leq j \leq p+q \\ 1 \leq k \leq m}}$ be a family of elements of $C^\infty(M,\mathbb{C})$. Then, there is a vector field $X$ on $M$ such that:
\begin{enumerate}[label=(\roman*)]
\item $X$ is supported away from $x_1,\dots,x_{m+1}$, $T(x_1)$ and the antecedents of $x_1,\dots,x_{m+1}$ by $T$;\label{item:support}
\item there is $N \geq 0$ such that $\mathcal{L}_{T}^N P_T(X) f_{T,j} = 0$ for $j = 1,\dots,p+q$;\label{item:kernel}
\item for $j = 1,\dots,p+q$ and $k = 1,\dots,m$, the function $\mathfrak{f}_{j,k}$ coincides with $\sum_{\ell = 0}^{N-1} \lambda_j(T)^{-\ell-1} \mathcal{L}_T^{\ell} P_{T}(X) f_{T,j}$ on a neighbourhood of $x_k$;\label{item:local_perturbation}
\item for $j = 1,\dots,p+q$, if $x_1$ is not a fixed point for $T$ then the function $\sum_{\ell = 0}^{N-1} \lambda_j(T)^{-\ell-1} \mathcal{L}_T^{\ell} P_{T}(X) f_{T,j}$ is identically zero on a neighbourhood of $E$;\label{item:we_need_more}
\item for $j= 1,\dots,p+q$, if there is no fixed point of $T$ among $x_1,\dots,x_{m+1}$, then the function $\sum_{\ell = 0}^{N-1} \lambda_j(T)^{-\ell-1} \mathcal{L}_T^{\ell} P_{T}(X) f_{T,j}$ is identically zero on a neighbourhood of $T(x_1)$.\label{item:not_too_much}
\end{enumerate}
\end{lemma}

\begin{remark}\label{remark:key_deformation}
Lemma \ref{lemma:key_deformation} deserves some comments. Let us point out that \ref{item:kernel} implies that the function $\sum_{\ell = 0}^{N-1} \lambda_j(T)^{-\ell-1} \mathcal{L}_T^{\ell} P_{T}(X) f_{T,j}$ that appears in \ref{item:local_perturbation} is the relevant part of the first order variation of $f_{T,j}$ when $T$ is perturbed using $X$ (see Lemma \ref{lemma:derivative_simple_resonance} and Remark \ref{remark:derivative_simple_resonance}).

Lemma \ref{lemma:key_deformation} may seem a bit involved. This is because we stated a result with enough generality so that it can be used to produce all the deformations that are required for the transversality arguments in the rest of the paper. There is only one proof below in which we will use the full generality of Lemma \ref{lemma:key_deformation} (the proof of Lemma \ref{lemma:generic_solvability}).

The simplest application of Lemma \ref{lemma:key_deformation} (this is the only case that will appear in the proof of Theorem \ref{theorem:generic_morse} in \S \ref{section:first_generic_properties}) is the following: we have a point $x_1$ and we want to produce a deformation of a resonant state near that point. Then, we choose a point $x_2$ with the same image as $x_1$ and we apply Lemma \ref{lemma:key_deformation} with $m= 2$ and $E = \emptyset$, ignoring \ref{item:we_need_more} and \ref{item:not_too_much}. We end up with an arbitrary deformation of our resonant state near $x_1$, up to the direction of the resonant state itself (which does not matter as explained in Remark \ref{remark:derivative_simple_resonance}).
\end{remark}

\begin{remark}\label{remark:proof_key}
Let us explain the ideas behind the proof of Lemma \ref{lemma:key_deformation} in the case of a single real resonance (that is $p = 1$ and $q = 0$). In view of Lemma~\ref{lemma:derivative_simple_resonance}, the most natural way to produce explicit perturbations of $f_{T,1}$ would be to prove that the map $X \mapsto P_T(X)f_{T,1}$ is surjective. This is what we will do in the proof of Lemma \ref{lemma:local_solvability} below, but we will need extra assumptions on $T$ to do so. The issue with the derivative formula in Lemma \ref{lemma:derivative_simple_resonance} is the presence of the operator $H_{T,\lambda_1(T)}$ we do not know how to compute. This is why we will use the second part of this lemma and search for $X$ such that $\mathcal{L}_T^N P_T(X) f_{T,1} = 0$ for some $N \geq 1$. 

To do so, we first construct a function $g$, supported near $x_1,\dots,x_{m+1}$, with the value that we want near $x_1,\dots,x_m$ and such that $\mathcal{L}_T g = 0$. The values of $g$ near $x_1,\dots,x_m$ are imposed, which is why we need the extra point $x_{m+1}$ to be able to achieve $\mathcal{L}_T g = 0$. Then, we want to write $g$ as $\mathcal{L}_T^{N-1} P_T(X)$ for some $N \geq 1$ and $X \in \Gamma(\T M)$. We work locally near each point $x_1,\dots,x_{m+1}$. Consider for instance $x_1$. If there is an antecedent $\tilde{x}_1$ for $x_1$ by $T$ such that $f_{T,1}(\tilde{x}_1) \neq 0$ then we can solve $P_T(X)f_{T,1} = g$ at least near $x_1$ (recall the definition \eqref{eq:derivative_transfer} and notice that we can divide by $f_{T,1}$ near $\tilde{x}_1$). If all the antecedents of $x_1$ by $T$ vanishes, then we look at the antecedents of $x_1$ by $T^N$ (for some large $N$) and solve $\mathcal{L}_T^{N-1} P_T(X) f_{T,1} = g$ instead of $P_T(X)f_{T,1} = g$. Indeed, it follows from Lemma \ref{lemma:expanding_open_sets} that there is $N$ large enough such that $x_1,\dots,x_{m+1}$ all have an antecedents by $T^N$ at which $f_{T,1}$ does not vanish.

When dealing with several resonances, the condition ``having an antecedent at which $f_{T,1}$ does not vanish'' is replaced by a linear algebra condition given in Lemma \ref{lemma:linear_algebra}. 
\end{remark}

We prepare the proof of Lemma \ref{lemma:key_deformation} with a linear algebra fact.

\begin{lemma}\label{lemma:linear_algebra}
Let $T \in U_0$. Introduce the function $\F_T : M \to \mathbb{R}^p \times \mathbb{C}^q$ defined by $\F_T(x) = (f_{T,j}(x))_{1 \leq j \leq p + q}$ for $x \in M$. There is $N_0 \geq 0$ such that for every $N \geq N_0$ and $x \in M$ there are points $y_1,\dots,y_{p+2q} \in M$ such that $T^N(y_1) = \dots = T^N(y_{p+2q}) = x$ and $(\F_T(y_1),\dots,\F_T(y_{p+2q}))$ is a basis (over $\mathbb{R}$) of $\mathbb{R}^p \times \mathbb{C}^q$.
\end{lemma}

\begin{proof}
Let us start by proving that the functions $f_{T,1},\dots,f_{T,p}$, $\re f_{T,p+1}$, $\dots$, $\re f_{T,p+q}$, $\im f_{T,p+1}$,$\dots$, $\im f_{T,p+q}$ are linearly independent over $\mathbb{R}$. Consider real numbers $\alpha_1$,$\dots$,$\alpha_p$, $\beta_{p+1}$,$\dots$, $\beta_{p+q}$, $\gamma_{p+1}$,$\dots$,$\gamma_{p+q}$ such that
\begin{equation*}
\sum_{j = 1}^p \alpha_j f_{T,j} + \sum_{j = p+1}^{p+q} \beta_j \re f_{T,j} + \gamma_j \im f_{T,j} = 0.
\end{equation*}
This relation may be rewritten as
\begin{equation}\label{eq:other_relation}
\sum_{j = 1}^p \alpha_j f_{T,j} + \sum_{j = p+1}^{p+q} \frac{\beta_j - i \gamma_j}{2} f_{T,j} + \frac{\beta_j + i \gamma_j}{2} \bar{f}_{T,j} = 0.
\end{equation}
Notice then that the functions $f_{T,1},\dots,f_{T,p}, f_{T,p+1}, \bar{f}_{T,p+1},\dots,f_{T,p+q}, \bar{f}_{T,p+q}$ are eigenvectors for $\mathcal{L}_T$ associated to distinct eigenvalues. Here, we use our assumption that $U_0$ is connected and that $\lambda_j \neq \bar{\lambda}_k$ when $j,k \in \set{p+1,\dots,p+q}$. Consequently, $f_{T,1},\dots,f_{T,p}, f_{T,p+1}, \bar{f}_{T,p+1},\dots$, $f_{T,p+q}, \bar{f}_{T,p+q}$ are linearly independent over $\mathbb{C}$ and it follows from \eqref{eq:other_relation} that $\alpha_1 = \dots  = \alpha_p = \beta_{p+1} = \dots = \beta_{p+q} = \gamma_{p+1} = \dots = \gamma_{p+q} = 0$.

Let us now call $E$ the linear subspace (over $\mathbb{R}$) of $C^\infty(M,\mathbb{C})$ spanned by $f_{T,1},\dots,f_{T,p}, \re f_{T,p+1}$, $\im f_{T,p+1}$, $\dots, \re f_{T,p+q}, \im f_{T,p+q}$. For each $x \in M$, let $l_x$ be the linear form on $E$ defined by $l_x(f) = f(x)$ for $f \in E$. Since $\bigcap_{x \in M} \ker l_x = \set{0}$, we find that the $l_x$'s for $x \in M$ span the dual $E^*$ of $E$. Hence, there are $z_1,\dots,z_{p+2q} \in M$ such that $(l_{z_1},\dots,l_{z_{p+2q}})$ is a basis of $E^*$. Since $(f_{T,1},\dots,f_{T,p}, \re f_{T,p+1}$, $\im f_{T,p+1}$, $\dots, \re f_{T,p+q}, \im f_{T,p+q})$ is a basis for $E$, we find that the matrix
\begin{equation*}
\begin{bmatrix}
f_{T,1}(z_1) & \dots & f_{T,1}(z_{p+2q}) \\
\dots & \dots & \dots \\
f_{T,p}(z_1) & \dots & f_{T,p}(z_{p+2q}) \\
\re f_{T,p+1}(z_1) & \dots & \re f_{T,p+1}(z_{p+2q}) \\
\im f_{T,p+1}(z_1) & \dots & \im f_{T,p+1}(z_{p+2q}) \\
\dots & \dots & \dots \\
\re f_{T,p+q}(z_1) & \dots & \re f_{T,p+q}(z_{p+2q}) \\
\im f_{T,p+q}(z_1) & \dots & \im f_{T,p+q}(z_{p+2q}) \\
\end{bmatrix}
\end{equation*}
is invertible. Equivalently, $(\F_T(z_1),\dots,\F_T(z_{p+2q}))$ is a basis for $\mathbb{R}^p \times \mathbb{C}^q$.

At this point, we consider the set $\mathcal{V}$ of elements $(y_1,\dots,y_{p+2q})$ of $M^{p+2q}$ such that $(\F_T(y_1), \dots, \F_T(y_{p+2q}))$ is a basis for $\mathbb{R}^p \times \mathbb{C}^q$. The set $\mathcal{V}$ is open and we just proved that $\mathcal{V}$ is non-empty. Let us introduce the map $\widetilde{T} : M^{p+2q} \to M^{p+2q}$ defined by $\widetilde{T}(y_1,\dots,y_{p+2q}) = (T(y_1),\dots,T(y_{p+2q}))$. Notice that $\widetilde{T}$ is expanding. Hence, Lemma \ref{lemma:expanding_open_sets} implies that there is $N_0 \geq 0$ such that $\widetilde{T}^{N_0}(\mathcal{V}) = M^{p+2q}$. In particular, for every $N \geq N_0$ we have
\begin{equation*}
\set{(x,\dots,x) :x \in M} \subseteq \widetilde{T}^{N}(\mathcal{V})
\end{equation*}
which proves the lemma.
\end{proof}

From Lemma \ref{lemma:linear_algebra}, we deduce a crucial tool for the proof of Lemma~\ref{lemma:key_deformation}.

\begin{lemma}\label{lemma:long_range_construction}
Let $T \in U_0$. There is $N_0 \geq 0$ such that for every $N \geq N_0$ and $y_0 \in M$, if $V_0$ is a neighbourhood of $y_0$, then there is an open neighbourhood $V \subseteq V_0$ of $y_0$ in $M$ such that if $\mathfrak{g}_1,\dots,\mathfrak{g}_p$ are elements of $C_0^\infty(M,\mathbb{R})$ supported in $V$ and $\mathfrak{g}_{p+1},\dots,\mathfrak{g}_{p+q}$ are elements of $C_0^\infty(M,\mathbb{C})$ supported in $V$ then there is a vector field $X$ on $M$ such that:
\begin{enumerate}[label=(\roman*)]
\item for $j = 1,\dots,p+q$, we have $\mathcal{L}_T^{N} P_T(X) f_{T,j} = \mathfrak{g}_j$;
\item the vector field $X$ is supported in $T^{-N-1}(V)$.
\end{enumerate}
\end{lemma}

\begin{proof}
Let $N_0$ be as in Lemma \ref{lemma:linear_algebra}. Let $N \geq N_0$ and $y_0 \in M$. Let $z_1,\dots,z_{p+2q}$ be points in $M$ such that $(\F_T(z_1),\dots, \F_T(z_{p+2q}))$ is a basis for $\mathbb{R}^p \times \mathbb{C}^q$ and $T^{N+1}(z_1) = \dots = T^{N+1}(z_{p+2q}) = y_0$. Since $T^{N+1}$ is a local diffeomorphism, we may find open neighbourhoods $W_1,\dots,W_{p+2q}$ pf $z_1,\dots,z_{p+2q}$ such that $T^{N+1}(W_1) = \dots = T^{N+1}(W_{p+2q})$ and for $j = 1,\dots,p+2q$ the map $T^{N+1}$ induces a diffeomorphism from $W_j$ to $T^{N+1}(W_j)$. Up to taking $W_1,\dots,W_{p+2q}$ smaller, we may assume that $W_1,\dots,W_{p+2q}$ are disjoint, diffeomorphic to the open ball of dimension $\dim M$, that $T^{N+1}(W_1) \subseteq V_0$ and that $(w_1,\dots,w_{p+2q}) \in W_1 \times \dots \times W_{p+2q}$ implies that the family $(\F_T(w_1),\dots,\F_T(w_{p+2q}))$ is a basis for $\mathbb{R}^p \times \mathbb{C}^q$.

We set $V = T^{N+1}(W_1)$. Let then $\mathfrak{g}_1,\dots, \mathfrak{g}_{p+q}$ be as in the statement of the lemma. For $j = 1,\dots,p$, since $V$ is diffeomorphis to an open ball and $\mathfrak{g}_j$ has zero average, we may find a vector field $Y_j$ supported in $V$ such that $\mathfrak{g}_j = \Div(Y_j)$. For $j = p+1,\dots,q$, since the function $\mathfrak{g}_j$ is a priori complex valued, it is not always possible to get such a vector field. Instead, we get a section $Y_j$ of $\T M \otimes \mathbb{C}$ (a ``complex vector field'') supported in $V$ and such that $\mathfrak{g}_j = \Div(Y_j)$.

For $x \in V$, we have
\begin{equation*}
(Y_1(x),\dots,Y_{p+q}(x)) \in (\T_x M)^p \times (\T_x M \otimes \mathbb{C})^q \simeq \T_x M \otimes (\mathbb{R}^p \times \mathbb{C}^q). 
\end{equation*}
Since $(\F_T((T^{N+1}_{|W_1})^{-1}(x)),\dots,\F_T((T^{N+1}_{|W_{p+2q}})^{-1}(x)))$ is a basis for $\mathbb{R}^p \times \mathbb{C}^q$, we find that there are $Z_1(x),\dots,Z_{p+2q}(x) \in \T_x M$ such that 
\begin{equation}\label{eq:tensor_product_decomposition}
Y_j(x) = \sum_{k = 1}^{p+2q} f_{T,j}((T^{N+1}_{|W_k})^{-1}(x)) Z_k(x) \textup{ for } j = 1,\dots,p+q.
\end{equation}
From the uniqueness in \eqref{eq:tensor_product_decomposition}, we find that the vector fields $Z_1,\dots,Z_{p+2q}$ on $V$ defined in this way are smooth and compactly supported. For $k = 1,\dots,p+2q$, let $X_k$ be the vector field on $M$ defined by
\begin{equation*}
X_k(x) = \begin{cases} - |\det DT^{N+1}(x)| D T^{N+1}(x)^{-1} \cdot Z_k(T^{N+1}x) & \textup{ if } x \in W_k, \\ 0 & \textup{ if } x \in M \setminus W_k. \end{cases}
\end{equation*}
Notice that $X_k$ is a smooth vector field, supported in $W_k$, and that, for $j =1,\dots,p+q$, we have
\begin{equation*}
\mathcal{L}_T^N P_T(X_k)(f_{T,j}) = - \mathcal{L}_T^{N+1}(\Div(f_{T,j} X_k)).
\end{equation*}
Hence, $\mathcal{L}_T^N P_T(X_k)(f_{T,j})$ is supported in $V$ and for $x \in V$ we have
\begin{equation*}
\begin{split}
& \mathcal{L}_T^N P_T(X_k)(f_{T,j})(x) \\ & \qquad \qquad \qquad = \frac{\Div(|\det DT^{N+1}| f_{T,j} (DT^{N+1})^{-1} \cdot (Z_k \circ T^{N+1}))((T^{N+1}_{|W_k})^{-1}x)}{|\det T^{N+1} ((T^{N+1}_{|W_k})^{-1}x)|} \\
     & \qquad \qquad \qquad = \Div(f_{T,j} \circ (T^{N+1}_{|W_k})^{-1} Z_k)(x).
\end{split}
\end{equation*}
Letting $X = \sum_{k = 1}^{p+2q} X_k$ and summing over $k$, we find that for $j = 1,\dots,p+2q$ the function $\mathcal{L}_T^N P_T(X_k)(f_{T,j})$ is supported in $V$ and for $x \in V$ we have
\begin{equation*}
\mathcal{L}_T^N P_T(X)(f_{T,j})(x) = \Div(\sum_{k = 1}^{p+2q}f_{T,j} \circ (T^{N+1}_{|W_k})^{-1} Z_k)(x) = \Div(Y_j)(x) = \mathfrak{g}_j(x).
\end{equation*}
It remains to notice that $X$ is supported in $\bigcup_{k = 1}^{p+2q} W_k \subseteq T^{-N-1}(V)$.
\end{proof}

We need a last preparatory lemma before proving Lemma \ref{lemma:key_deformation}.

\begin{lemma}\label{lemma:contracting_fixed_point}
Let $V$ be a smooth manifold. Let $x_0 \in V$. Let $G : V \to V$ be a smooth map such that $G(x_0) = x_0$. Assume that the eigenvalues of $DG(x_0)$ belong to $\set{ z \in \mathbb{C} : 0 < |z| < 1}$.  Let $h : V \to \mathbb{C}$ be a smooth function such that $|h| < 1$. Let $w : V \to \mathbb{C}$ be a smooth function. There is an open neighbourhood $\widetilde{V}$ of $x_0$ in $V$ such that $G(\widetilde{V}) \subseteq \widetilde{V}$ and a smooth function $g : \widetilde{V} \to \mathbb{C}$ such that 
\begin{equation}\label{eq:contracting_fixed_point}
g - h. g \circ G = w
\end{equation}
on $\widetilde{V}$. Moreover, if $h$ and $w$ are real-valued, the function $g$ is real-valued.
\end{lemma}

\begin{proof}
Without loss of generality we may assume that $V$ is an open subset of $\mathbb{R}^d$ for some $d \geq 1$ and that $x_0 = 0$. We wtart by searching a power series $g_{\textup{ps}}$ at $0$ that solves \eqref{eq:contracting_fixed_point}. Beware that we do not claim that we can find a convergent power series solving \eqref{eq:contracting_fixed_point}, we only claim equality of power series at $0$.

If we write $g_{\textup{ps}}(x) = \sum_{m \geq 0} A_m(x)$ where $A_m(x)$ denotes a homogeneous polynomial of order $m$, then for every $m \geq 0$ we have, as power series at $0$:
\begin{equation*}
h(x) A_m(G(x)) = h(0) A_m(DG(0) \cdot x) + \sum_{k \geq m+1} B_{m,k}(A_m)(x)
\end{equation*}
where for $k \geq m+1$ the polynomial $B_{m,k}(A_m)$ is a homogeneous of order $k$. Letting $\sum_{m \geq 0} w_m(x)$ denote the Taylor series at $0$ of $w$ (where $w_m$ is homogeneous of order $m$), we define inductively
\begin{equation*}
A_m = (I - h(0) DG(0)^*)^{-1}(w_m + \sum_{k= 0}^{m-1} B_{k,m}(A_k)).  
\end{equation*} 
Here, $DG(0)^*$ denote the pullback by $DG(0)$ on homogeneous polynomials of order $m$. The eigenvalues of $DG(0)^*$ are products of $m$ eigenvalues of $DG(0)$. In particular, they have modulus less than $1$ and thus $(I - h(0) DG(0)^*)^{-1}$ is well-defined. With this definition, $g_{\textup{ps}}$ solves \eqref{eq:contracting_fixed_point} as a power series at $0$. 

Let $g_0$ be a smooth function near $0$ whose Taylor series at $0$ is given by $g_{\textup{ps}}$. We search for $g$ in the form $g = g_0 + g_1$. The function $g_1$ must then satisfy
\begin{equation}\label{eq:reduced_contracting_equation}
g_1 - h. g_1 \circ G = \tilde{w}
\end{equation}
where $\tilde{w} = w - (g_0 - h. g_0 \circ G)$ is a smooth function that vanishes at infinite order at $0$. Since $G$ has a contracting fixed point at $0$, there is an open neighbourhood $\widetilde{V}$ of $0$ such that $\overline{G(\widetilde{V})}$ is a compact subset of $\widetilde{V}$ and for every $x \in \widetilde{V}$ the sequence $(G^n(x))_{n \geq 0}$ converges exponentially fast to $0$. Let $L$ be the operator $ u \mapsto h. u \circ T$ acting on functions on $\widetilde{V}$. Notice that for $n \geq 1$, we have $L^n (\tilde{w}) = h. h \circ G \dots h \circ G^{n-1}. \tilde{w} \circ G^{n}$. Since $\tilde{w}$ vanishes at infinite order at $0$, it follows from the contracting property of $G$ that the $L^\infty$ norm of $L^n(\tilde{w})$ on $\widetilde{V}$ decays superexponentially fast with $n$. We can consequently define $g_1$ as $\sum_{n \geq 0} L^n(\tilde{w})$. We know that $g_1$ is continuous and solve \eqref{eq:reduced_contracting_equation} on $\widetilde{V}$. We want to prove that $g_1$ is smooth.

For $m \geq 0$, let $C^m_b(\widetilde{V})$ denote the space of $C^m$ functions on $\widetilde{V}$ with bounded derivatives of order up to $m$. For every $m \geq 0$, the operator $L$ is bounded on $C^{m+1}_b(\widetilde{V})$ and thus $\n{L^n(\tilde{w})}_{C_b^{m+1}}$ grows at most exponentially fast with $n$. Since $\n{L^n(\tilde{w})}_{C_b^0}$ decays superexponentially fast with $n$, we find by interpolation (we may assume that $\widetilde{V}$ is a smooth domain) that for every $m \geq 0$ the norm $\n{L^n(\tilde{w})}_{C_b^m}$ decays superexponentially fast with $n$, proving that $g_1$ is $C^m$. We proved that $g_1$ is smooth.
\end{proof}

We are finally ready for the proof of Lemma \ref{lemma:key_deformation}.

\begin{proof}[Proof of Lemma \ref{lemma:key_deformation}]
\underline{Step 1.} We start with some set up. Let $N_0 \geq 1$ be as in Lemma \ref{lemma:long_range_construction}. For $k = 1,\dots,m+1$, let $\tilde{x}_k$ be an antecedent of $x_k$ for $T$ that is not periodic (such a point always exist because each point has at most one periodic antecedent). If there is $k \in \set{1,\dots,m}$ such that $x_k$ is a fixed point for $T$, then we impose $\tilde{x}_k = x_{m+1}$. If $x_1$ is not a fixed point for $T$, then we impose $\tilde{x}_1 \notin E$. Let then $V_1,\dots,V_{m+1}$ be open neighbourhoods of $\tilde{x}_1,\dots,\tilde{x}_{m+1}$ respectively such that $T^2(V_1) = \dots = T^2(V_{m+1})$ and for $j = 1,\dots,m+1$ the map $T^2$ induces a diffeomorphims from $V_j$ to its image. Up to taking $V_1,\dots,V_{m+1}$ smaller, we assume that these sets are disjoint and that for $j = 1,\dots,m+1$ the set $V_j$ is contained in the set $V$ given by Lemma \ref{lemma:long_range_construction} applied with ``$y_0 = \tilde{x}_j$'' and ``$N = N_0$''. Eventually reducing even more the size of $V_1,\dots,V_{m+1}$, we may assume that the following properties hold:
\begin{enumerate}[label=\alph*)]
\item The set $\bigcup_{k =1}^{m+1}\overline{T^{-N_0 - 1}(V_k)}$ does not contain $x_1,\dots,x_{m+1}, T(x_1)$ and the antecedents of $x_1,\dots,x_{m+1}$ by $T$.\label{itema}
\item For $k = 1,\dots,m$, the point $x_k$ does not belong to $\bigcup_{n = 1}^{m+1} \bigcup_{\ell = 1}^{N_0} \overline{T^{-\ell}(V_n)}$.\label{itemb}
\item If $x_1$ is not a fixed point for $T$ then $E$ does not intersect $\bigcup_{n = 1}^{m+1}( \overline{T(V_n)} \cup \bigcup_{\ell = 0}^{N_0} \overline{T^{-\ell}(V_n)})$.\label{itemc}
\item If there is no fixed points among $x_1,\dots,x_{m+1}$ then $T(x_1)$ does not belong to $\bigcup_{n = 1}^{m+1}( \overline{T(V_n)} \cup \bigcup_{\ell = 0}^{N_0} \overline{T^{-\ell}(V_n)}$.\label{itemd}
\end{enumerate}
Let us justify that we can achieve these properties:
\begin{enumerate}[label=\alph*)]
\item Notice that $\bigcup_{k =1}^{m+1}\overline{T^{-N_0 - 1}(V_k)}$ is an arbitrarily small neighbourhood of the set $T^{-N_0 - 1}(\set{\tilde{x}_1,\dots,\tilde{x}_{m+1}})$, and the latter does not contain $x_1,\dots,x_{m+1}$, $T(x_1)$ and the antecedents of $x_1,\dots,x_{m+1}$ by $T$, since $\tilde{x}_1,\dots,\tilde{x}_{m+1}$ are not periodic for $T$.
\item The union $\bigcup_{n = 1}^{m+1} \bigcup_{\ell = 1}^{N_0} \overline{T^{-\ell}(V_n)}$ is an arbitrarily small neighbourhood of the set $\bigcup_{\ell = 1}^{N_0} T^{- \ell}(\set{\tilde{x}_1,\dots,\tilde{x}_{m+1}})$, and the latter does not contain $x_1,\dots$, $x_{m}$ because $\tilde{x}_1,\dots,\tilde{x}_{m+1}$ are not periodic.
\item The set $\bigcup_{n = 1}^{m+1}( \overline{T(V_n)} \cup \bigcup_{\ell = 0}^{N_0} \overline{T^{-\ell}(V_n)})$ is an arbitrarily small neighbourhood of $\set{x_1,\dots,x_{m+1}} \cup \bigcup_{\ell = 0}^{N_0} T^{- \ell}( \set{\tilde{x}_1,\dots,\tilde{x}_{m+1}})$. If $x_1$ is not a fixed point for $T$, then $E$ does not intersect $\set{x_1,\dots,x_{m+1}}$. Moreover, when $x_1$ is not a fixed point we imposed that $\tilde{x}_1 \notin E$, which implies that $E$ does not intersect $\set{\tilde{x}_1,\dots,\tilde{x}_{m+1}}$. The sets $E$ and $T^{-1}(\set{\tilde{x}_1,\dots,\tilde{x}_{m+1}})$ do not intersect because we imposed $\tilde{x}_k = x_{m+1}$ when $Tx_k = x_k$. Since $\tilde{x}_1,\dots,\tilde{x}_{m+1}$ are not periodic, the sets $E$ and $\bigcup_{\ell = 2}^{N_0} T^{- \ell} \set{\tilde{x}_1,\dots,\tilde{x}_{m+1}}$ never intersect.
\item As in the previous point, we use that $\bigcup_{n = 1}^{m+1}( \overline{T(V_n)} \cup \bigcup_{\ell = 0}^{N_0} \overline{T^{-\ell}(V_n)})$ is a small neighbourhood of $\set{x_1,\dots,x_{m+1}} \cup \bigcup_{\ell = 0}^{N_0} T^{- \ell}( \set{\tilde{x}_1,\dots,\tilde{x}_{m+1}})$. If there is no fixed point among $x_1,\dots,x_{m+1}$, then $T(x_1)$ does not belong to $\set{x_1,\dots,x_{m+1}}$. Since $\tilde{x}_1,\dots,\tilde{x}_k$ are not periodic points, $T(x_1)$ does not belong to $\bigcup_{\ell = 0}^{N_0} T^{- \ell}( \set{\tilde{x}_1,\dots,\tilde{x}_{m+1}})$.
\end{enumerate}

\underline{Step 2.} We want now to construct functions $(\mathfrak{g}_{j,k})_{\substack{1 \leq j \leq p+q \\ 1 \leq k \leq m}}$ to which we will apply Lemma \ref{lemma:long_range_construction}. These functions must have the following properties:
\begin{enumerate}[label=\arabic*)]
\item for $j = 1,\dots,p+q$ and $k = 1,\dots,m+1$, the function $\mathfrak{g}_{j,k}$ is supported in $V_k$ and has zero average;\label{item:first}
\item for $j = 1,\dots,p$ and $k = 1,\dots,m+1$, the function $\mathfrak{g}_{j,k}$ is real-valued;\label{item:second}
\item for $j = 1,\dots,p+q$, we have $\sum_{k = 1}^{m+1} \mathcal{L}_T^2 \mathfrak{g}_{j,k} = 0$;\label{item:third}
\item for $j = 1,\dots,p$ and $k = 1,\dots,m$, the restrictions of the functions $\mathfrak{f}_{j,k}$ and $\lambda_j(T)^{-N_0 -2} \sum_{\ell = 1}^{m+1} \mathfrak{g}_{j,\ell} + \lambda_j(T)^{-N_0-1} \sum_{\ell = 1}^{m+1} \mathcal{L}_T \mathfrak{g}_{j,\ell}$ to a neighbourhood of $x_k$ coincide. \label{item:fourth}
\end{enumerate} 
For this construction, we will distinguish three cases.

\noindent \textbf{Case 1: there is no fixed point for $T$ among $x_1,\dots,x_{m+1}$.}

In that case, for $\ell = 1,\dots,m+1$ and $k = 1,\dots,m$, the points $\tilde{x}_\ell$ and $x_k$ are distinct (otherwise, we would have $T x_\ell = T x_k = T \tilde{x}_\ell = x_\ell$). Hence, by imposing that for $\ell = 1,\dots,m+1$ and $j = 1,\dots,p+q$ the function $\mathfrak{g}_{j,\ell}$ is supported close enough to $\tilde{x}_\ell$, condition \ref{item:fourth} above becomes: for $j = 1,\dots,p+q$ and $k = 1,\dots,m$ the functions $\lambda_j(T)^{-N_0-1} \mathcal{L}_T\mathfrak{g}_{j,k}$ and $ \mathfrak{f}_{j,k}$ coincide near $x_k$. This is guaranteed by imposing that $\mathfrak{g}_{j,k} = \lambda_j(T)^{N_0 + 1} |\det DT| \mathfrak{f}_{j,k} \circ T$ near $\tilde{x}_k$.

Condition \ref{item:third} is obtained by setting $\mathfrak{g}_{j,m+1} = - |\det DT^2| \sum_{\ell = 1}^m \mathcal{L}_T^2 \mathfrak{g}_{j,\ell} \circ T^2$ on $V_{m+1}$ (and zero outside) for $j = 1,\dots,p+q$. Condition \ref{item:second} is simple to get (take the real part if needed). 

Concerning condition \ref{item:first}, the point that could be a difficult is the zero average property. It is obtained for $j = 1,\dots,p+q$ and $k = 1,\dots,m$ by modifying $\mathfrak{g}_{j,k}$ away from $\tilde{x}_k$  (so that we do not lose the fourth condition). It is then automatically satisfied for $k = m+1$ in view of our definition of $\mathfrak{g}_{j,m+1}$.

\noindent \textbf{Case 2: there is a fixed point for $T$ among $x_1,\dots,x_m$.}

Up to relabelling, assume that $x_1$ is the fixed point. Recall that in that case, we imposed $\tilde{x}_1 = x_{m+1}$. Notice then that condition \ref{item:fourth} does not include the case $k = m+1$. Hence, the reasoning above still applies.

\noindent \textbf{Case 3: $x_{m+1}$ is a fixed point for $T$.}

If $x_{m+1}$ has an antecedent that is not among $x_1,\dots,x_{m+1}$, then we can take $\tilde{x}_{m+1} \notin \set{x_1,\dots,x_{m+1}}$ and work as above. However, it may happen that $x_1,\dots,x_{m+1}$ are the only antecedents of $x_{m+1}$. In that case, up to relabelling, we may assume that $\tilde{x}_{m+1} = x_m$. For $j =1,\dots,p+q$ and $k = 1,\dots,m-1$, we can still define $\mathfrak{g}_{j,k}$ by $\mathfrak{g}_{j,k} = \lambda_j(T)^{N_0 + 1} |\det DT| \mathfrak{f}_{j,k} \circ T$ near $\tilde{x}_k$. To get condition \ref{item:third}, we still impose $\mathfrak{g}_{j,m+1} = - |\det DT^2| \sum_{\ell = 1}^m \mathcal{L}_T^2 \mathfrak{g}_{j,\ell} \circ T^2$ on $V_{m+1}$. 

We need now to find (for $j = 1,\dots,p+q$) a function $\mathfrak{g}_{j,m}$ for which \ref{item:fourth} holds. The condition is that the functions $\mathfrak{f}_{j,m}$ and $\lambda_j(T)^{-N_0 - 2} \mathfrak{g}_{j,m+1} + \lambda_j(T)^{-N_0 - 1} \mathcal{L}_{T} \mathfrak{g}_{j,m}$ coincides near $x_m = \tilde{x}_{m+1}$. Using our definition of $\mathfrak{g}_{j,m+1}$, we want to achieve
\begin{equation*}
\begin{split}
& - |\det DT^2| \mathcal{L}_T^2 \mathfrak{g}_{j,m} \circ T^2 + \lambda_j(T) \mathcal{L}_T \mathfrak{g}_{j,m} \\ & \qquad \qquad \qquad \qquad = \lambda_j(T)^{N_0 + 2} \mathfrak{f}_{j,m} +  |\det DT^2| \sum_{\ell = 1}^{m-1} \mathcal{L}_T^2 \mathfrak{g}_{j,\ell} \circ T^2
\end{split}
\end{equation*}
near $x_m$. Since we search for $\mathfrak{g}_{j,m}$ supported near $\tilde{x}_m$, we can rewrite the relation above as
\begin{equation*}
\begin{split}
& \mathfrak{g}_{j,m} \circ (T^2_{|V_m})^{-1} \circ T^3 - \lambda_j(T) \frac{|\det DT^2| \circ (T^2_{|V_m})^{-1} \circ T^3}{|\det DT^2| \circ T |\det DT|} \mathfrak{g}_{j,m} \\ & \qquad \qquad = - \frac{|\det DT^2| \circ (T^2_{|V_m})^{-1} \circ T^3}{|\det DT^2| \circ T} \\ & \qquad \qquad \qquad \qquad \times  \left(\lambda_j(T)^{N_0 + 2} \mathfrak{f}_{j,m} \circ T +  |\det DT^2| \circ T \sum_{\ell = 1}^{m-1} \mathcal{L}_T^2 \mathfrak{g}_{j,\ell} \circ T^3\right)
\end{split}
\end{equation*}
near $\tilde{x}_m$. Notice then that the map $(T^2_{|V_m})^{-1} \circ T^3$ has an expanding fixed point at $\tilde{x}_m$ (it is conjugated to the fixed point of $T$ at $x_{m+1}$ by $T^2$). Let $G$ be the local inverse of $(T^2_{|V_m})^{-1} \circ T^3$ near $\tilde{x}_m$. The map $G$ has a contracting fixed point at $\tilde{x}_m$, and the relation we want for $\mathfrak{g}_{j,m}$ can be rewritten as
\begin{equation}\label{eq:contracting_equation}
\mathfrak{g}_{j,m} - h \mathfrak{g}_{j,m} \circ G = w \textup{ near } \tilde{x}_m,
\end{equation}
where
\begin{equation*}
h = \lambda_j(T) \frac{|\det DT^2|}{|\det DT^2| \circ T \circ G |\det DT| \circ G}
\end{equation*}
and
\begin{equation*}
\begin{split}
w = - \frac{|\det DT^2| }{|\det DT^2| \circ T \circ G} & \Bigg(\lambda_j(T)^{N_0 + 2} \mathfrak{f}_{j,m} \circ T \circ G \\ & \qquad \qquad+ |\det DT^2| \circ T \circ G \sum_{\ell = 1}^{m-1} \mathcal{L}_T^2 \mathfrak{g}_{j,\ell} \circ T^3 \circ G\Bigg).
\end{split}
\end{equation*}
Notice that $h(\tilde{x}_m) = \frac{\lambda_j(T)}{|\det DT(x_{m+1})|}$ has modulus strictly less than $1$, and thus $|h| < 1$ on a neighbourhood of $\tilde{x}_m$. Hence, the equation \eqref{eq:contracting_equation} is solved by Lemma \ref{lemma:contracting_fixed_point}. Thus we get a $\mathfrak{g}_{j,m}$ that guarantees the validity of \ref{item:third} and \ref{item:fourth}. The conditions \ref{item:first} and \ref{item:second} are dealt with as in the other cases.

\underline{Step 3.} Now that the $\mathfrak{g}_{j,k}$'s are constructed, for $k= 1,\dots,m+1$, we apply Lemma \ref{lemma:long_range_construction} to the family $(\mathfrak{g}_{j,k})_{1 \leq j \leq p +q}$ with ``$y_0 = \tilde{x}_k$ and $N = N_0$''. We end up with a vector field $X_k$ supported in $T^{-N_0 - 1}(V_k)$ and such that $\mathcal{L}_T^{N_0} P_T(X_k) f_{T,j} = \mathfrak{g}_{j,k}$ for $j = 1,\dots,p+q$. We define then $X = \sum_{k = 1}^{m+1} X_k$.

Let us check that this choice is consistent with the five properties we are searching for.

\ref{item:support} The vector field $X$ is supported in $\bigcup_{k = 1}^{m+1} T^{-N_0 - 1}(V_k)$. This set does not contain $x_1,\dots,x_{m+1}, T(x_1)$ and the antecedents of $x_1,\dots,x_{m+1}$ according to \ref{itema}. 

\ref{item:kernel} With $N = N_0 + 2$, we have for $j = 1,\dots,p+q$ that $\mathcal{L}_T^N P_T(X)f_{T,j} = \mathcal{L}_T^2(\sum_{k = 1}^{m+1} \mathfrak{g}_{j,k}) = 0$.

\ref{item:local_perturbation} Let $j \in \set{1,\dots,p+q}$ and $k \in \set{1,\dots,m}$, and write 
\begin{equation}\label{eq:forget_past}
\begin{split}
& \sum_{\ell = 0}^{N-1} \lambda_j(T)^{- \ell - 1} \mathcal{L}_{T}^\ell P_T(X) f_{T,j} \\ & \qquad \qquad =  \lambda_j(T)^{-N_0 - 2} \sum_{n = 1}^{m+1} \mathcal{L}_T(\mathfrak{g}_{j,n}) + \lambda_j(T)^{-N_0 - 1} \sum_{n = 1}^{m+1} \mathfrak{g}_{j,n} \\ & \qquad \qquad \qquad \qquad \qquad \qquad \qquad \qquad \qquad + \sum_{\ell = 0}^{N_0 - 1} \lambda_j(T)^{- \ell - 1} \mathcal{L}_T^\ell P_T(X) f_{T,j}.
\end{split}
\end{equation}
The sum of the first two terms in the right hand side coincide with $\mathfrak{f}_{j,k}$ on a neighbourhood of $x_k$ by \ref{item:fourth}. The last sum is supported in $\bigcup_{n= 1}^{m+1} \bigcup_{\ell = 1}^{N_0} \overline{T^{-\ell}(V_n)}$ which is away from $x_k$ according to \ref{itemb}.

\ref{item:we_need_more} Assume that $x_1$ is not a fixed point. For $j = 1,\dots,p+q$, the function $\sum_{\ell = 0}^{N-1} \lambda_j(T)^{- \ell - 1} \mathcal{L}_T^\ell P_T(x) f_{T,j}$ is supported in $\bigcup_{n= 1}^{m+1}(T(V_n) \cup \bigcup_{\ell = 0}^{N_0} T^{-\ell}(V_n))$, which is away from $E$ according to \ref{itemc}. 

\ref{item:not_too_much} Assume that there is no fixed point for $T$ among $x_1,\dots,x_{m+1}$. For $j = 1,\dots,p+q$, the function $\sum_{\ell = 0}^{N-1} \lambda_j(T)^{- \ell - 1} \mathcal{L}_T^\ell P_T(x) f_{T,j}$ is supported in $\bigcup_{n= 1}^{m+1}(T(V_n) \cup \bigcup_{\ell = 0}^{N_0} T^{-\ell}(V_n))$, which is away from $T(x_1)$ according to \ref{itemd}. 
\end{proof}

\section{Proof of Theorem \ref{theorem:generic_morse}}\label{section:first_generic_properties}

We have enough material to prove Theorem \ref{theorem:generic_morse}. Indeed, with Lemma \ref{lemma:key_deformation}, the proof of Theorem \ref{theorem:generic_morse} is reduced to standard transversality arguments (see for instance \cite[\S 5]{laudenbach_cours}).

\begin{proof}[Proof of Theorem \ref{theorem:generic_morse}]
That $\mathcal{U}_\delta$ is open is a consequence of Proposition \ref{proposition:stability_resonances}. Let us prove that it is dense. Let $U$ be an open subset of $\Exp(M)$. We want to prove that $U \cap \mathcal{U}_\delta \neq \emptyset$. According to Theorem \ref{theorem:generic_simple}, there is $F_0 \in U$ such that all resonances for $F_0$ of modulus larger than or equal to $\delta$ are simple. Let $\lambda_1,\dots,\lambda_p$ be the real resonances for $F_0$ that have absolute value larger than or equal to $\delta$ and $\lambda_{p+1},\dots,\lambda_{p+q}$ be the non-real resonances for $F_0$ with modulus larger than or equal to $\delta$ and positive imaginary part. Let then $U_0 \subseteq U$ be a neighbourhood of $F_0$ in $\Exp(M)$ as in \S \ref{section:key_lemma}. In the rest of the proof, we will use the notations introduced in the beginning of \S \ref{section:key_lemma}. Up to taking $U_0$ smaller, we assume that for every $T \in U_0$, the resonances of modulus larger than or equal to $\delta$ of $T$ are among $\lambda_1(T),\dots,\lambda_{p+q}(T)$.

Let $d$ be the dimension of $M$. Consider a point $x_1$ in $M$. Since $F_0$ has degree at least $2$, there is a point $x_2$ in $M$ with $F_0(x_2) = F_0(x_1)$. Let $g_{x_1,1},\dots,g_{x_1,d}$ be a smooth functions from $M$ to $\mathbb{R}$ such that $\mathrm{d}g_{x_1,1}(x_1),\dots,\mathrm{d}g_{x_1,d}(x_1)$ span $\T_{x_1}^* M$ and $g_{x_1,1}(x_1) \neq 0$. For $n = 1,\dots,d$, let us apply Lemma \ref{lemma:key_deformation} with $T = F_0,m = 1$ and $\mathfrak{f}_{j,1} = g_{x_1,n}$ for $1 \leq j \leq p+q$. We denote by $X_{x_1,n}$ the resulting vector field. We construct a last vector field $X_{x_1,d+1}$ by applying Lemma \ref{lemma:key_deformation} with $T = F_0,m = 1$ and $\mathfrak{f}_{j,1} = i g_{x_1,1}$ for $p+1 \leq j \leq p+q$ (the value of $\mathfrak{f}_{j,1}$ for $1 \leq j \leq p$ does not matter). Notice (see Lemma \ref{lemma:derivative_simple_resonance}) that there is a neighborhood $V_{x_1}$ of $x_1$ in $M$ such that for every $x \in V$ and $1 \leq j \leq p$ the maps $\mathrm{d} H_{F_0,\lambda_j} P_{F_0}(X_{x_1,n}) f_{F_0,j} (x) \neq 0$ for $n = 1,\dots,d$ span $\T_{x_1}^* M$ and $H_{F_0,\lambda_j} P_{F_0}(X_{x_1,1}) f_{F_0,j} (x) \neq 0$. For $p+1 \leq j \leq p+q$, and $x \in V_{x_1}$, the numbers $H_{F_0,\lambda_j} P_{F_0}(X_{x_1,1}) f_{F_0,j} (x)$ and $H_{F_0,\lambda_j} P_{F_0}(X_{x_1,d+1}) f_{F_0,j} (x)$ span $\mathbb{C}$ as a real vector space.

Since $M$ is compact, there are $y_1,\dots,y_N$ such that $M = \bigcup_{k = 1}^N V_{y_k}$. For $k = 1,\dots,N$ and $n = 1,\dots,d$, let $(\phi_{k,n}^t)_{t \in \mathbb{R}}$ be the flow of $X_{y_k,n}$. For $\bar{t} = (t_{k,n})_{\substack{1 \leq k \leq N \\ 1 \leq n \leq d}} \in \mathbb{R}^{N \times d}$, let $T_{\bar{t}} = F_0 \circ \phi_{1,1}^{t_{1,1}} \circ \dots \phi_{1,d}^{t_{1,d}} \circ \dots \circ \phi_{N,1}^{t_{N,1}}\circ \dots \circ \phi_{N,d}^{t_{N,d}}$. There is $\epsilon_0$ such that for every $\bar{t} \in (-\epsilon_0,\epsilon_0)^{N \times d}$ the map $T_{\bar{t}}$ belongs to $U_0$. For $j = 1,\dots,p$, we can consequently define a map
\begin{equation*}
\begin{array}{ccccc}
h_j &:& M \times (-\epsilon_0,\epsilon_0)^N & \to &\mathbb{R} \\
 & & (x,\bar{t}) & \mapsto & f_{T_{\bar{t}},j}(x).
\end{array}  
\end{equation*}  
According to Remark \ref{remark:useful_smoothness}, this is a smooth map. Moreover, it follows from Lemma \ref{lemma:derivative_simple_resonance} and Remark \ref{remark:derivative_simple_resonance} that for every $x \in M$ the derivative of $h_j$ at $(x,0)$ is
\begin{equation*}
(u,(\tau_{k,n})_{\substack{1 \leq k \leq N \\ 1 \leq n \leq d}}) \mapsto \mathrm{d} f_{F_0,j}(x) \cdot u + \sum_{k = 1}^N \sum_{n = 1}^d \tau_{k,n} H_{F_0,\lambda_j} P_{F_0}(X_{y_k,n}) f_{F_0,j}(x).
\end{equation*}
This derivative is non-zero (the term corresponding to $k$ such that $x \in V_{y_k}$ and $n = 1$ is non-zero) and thus there is $\epsilon_1 > 0$ such that $0$ is a regular value of the restriction of $h_j$ to $M \times (-\epsilon_1,\epsilon_1)^{N \times d}$. By a standard consequence of Sard's Theorem (see e.g. \cite[\S 5.3.1]{laudenbach_cours}), the set of $\bar{t} \in (-\epsilon_1,\epsilon_1)^{N \times d}$ such that zero is a regular value of $h_j(\cdot,\bar{t}) = f_{T_{\bar{t},j}}$ has full Lebesgue measure.

To get the Morse property, we apply the same reasoning to the map $\mathrm{d}_x h_j$ (partial derivative with respect to the variable in $M$) which is a map from $M \times (-\epsilon_0,\epsilon_0)^{N \times d}$ to $\T^* M$. Our construction ensures that $\mathrm{d}_x h_j$ is transverse to the zero section of $\T^* M$ along $M \times \set{0}$.  Indeed, for $x \in M$, the derivative with respect to $\bar{t}$ of $\mathrm{d}_x h_j$ at $(x,0)$ is
\begin{equation*}
(\tau_{k,n})_{\substack{1 \leq k \leq N \\ 1 \leq n \leq d}} \mapsto \sum_{k = 1}^N \sum_{n = 1}^d \tau_{k,n} \mathrm{d} H_{F_0,\lambda_j} P_{F_0}(X_{y_k,n}) f_{F_0,j}(x)
\end{equation*}
and the covectors that appear in this sum span $\T_x^* M$. Hence, we get (up to taking $\epsilon_1$ smaller) that for almost all $\bar{t}$ in $(-\epsilon_1,\epsilon_1)^{N \times d}$ the derivative $\mathrm{d}f_{T_{\bar{t},j}}$ is transverse to the zero section, i.e. $f_{T_{\bar{t}},j}$ is a Morse function.

Of course, we can get the same $\epsilon_1$ for each $j \in \set{1,\dots,p}$ (just take the minimal one). We deal with the complex resonant states similarly (our construction of vector fields ensures that, for $p+1 \leq j \leq p+q$, the derivative of $(x,\bar{t}) \mapsto f_{T_{\bar{t},j}}(x)$ at a point of $M \times \set{0}$ is surjective onto $\mathbb{C}$). Hence, we find that there is $\epsilon_1 > 0$ such that for every $\bar{t} \in (-\epsilon_1,\epsilon_1)^{N \times d}$ the map $T_{\bar{t}}$ belongs to $\mathcal{U}_\delta$. In particular, $U \cap \mathcal{U}_\delta \neq \emptyset$, ending the proof of the theorem.
\end{proof}

\begin{remark}
In the proof of Theorem \ref{theorem:generic_morse}, we worked with all the resonant states (corresponding to resonances larger than or equal to $\delta$) simultaneously because Lemma \ref{lemma:key_deformation} is general enough to do so, but we could have dealt with them one by one.
\end{remark}

\section{General results}\label{section:general_results}

This section is devoted to the proofs of Theorems \ref{theorem:generic_density}, \ref{theorem:generic_real_resonance} and \ref{theorem:generic_complex_resonance}. Concerning Theorems \ref{theorem:generic_real_resonance} and \ref{theorem:generic_complex_resonance}, we will actually prove a more detailed result, that allows to work with several resonances simultaneously:

\begin{theorem}\label{theorem:general_statement}
Let $n_0 \in \set{0,1}$ and $n_1,n_2 \in \mathbb{N}$. Let $\delta \in (0,1)$. Let $U$ be an open and dense subset of\footnote{The topology on the space $PC^\infty(M,\mathbb{R})$ of lines in $C^\infty(M,\mathbb{R})$ is defined by identifying it with the quotient of $C^\infty(M,\mathbb{R}) \setminus \set{0}$ under the action of $\mathbb{R}^*$, as we did for $PC_0^\infty(M,\mathbb{R})$ and $PC_0^\infty(M,\mathbb{C})$.} $PC^\infty(M,\mathbb{R})^{n_0} \times PC_0^\infty(M,\mathbb{R})^{n_1} \times PC_0^\infty(M,\mathbb{C})^{n_2} \times \mathbb{R}^{n_1} \times \mathbb{C}^{n_2}$. Let $m = \dim M + n_0 +n_1 + 2 n_2$. Let $\mathcal{U}$ be the set of $T \in \Exp_{\geq m}(M)$ such that if $\lambda_{1},\dots,\lambda_{n_0+n_1 + n_2}$ are resonances for $T$ such that:
\begin{itemize}
\item if $n_0 = 1$ then $\lambda_1 = 1$;
\item $\lambda_{n_0+1},\dots,\lambda_{n_0 + n_1}$ are distinct real numbers, not equal to $1$ and of absolute value larger than or equal to $\delta$;
\item $\lambda_{n_0 + n_1 + 1},\dots,\lambda_{n_0 + n_1 + n_2}$ are non-real numbers of modulus larger than or equal to $\delta$, distinct from each other and from the complex conjugate of each other
\end{itemize}
then $\lambda_1,\dots,\lambda_{n_0+n_1 + n_2}$ are simple resonances and\footnote{Here, and in the rest of this section, we make a slight abuse of notation and use $E_{T,\lambda}$ to denote the real eigenspace asssociated with $\lambda$, when $\lambda$ is a real resonance. It would be more rigorous to write $E_{T,\lambda} \cap C^\infty(M,\mathbb{R})$, but it seemed to us that the simplification of the notations was worth the small ambiguity.} $$(E_{T,\lambda_1},\dots, E_{T,\lambda_{n_0 + n_1 + n_2}},\lambda_{1 +n_0},\dots,\lambda_{n_0 + n_1 +n_2}) \in U.$$ The set $\mathcal{U}$ is open and dense in $\Exp_{\geq m}(M)$.
\end{theorem}

The idea behind the proof of Theorem \ref{theorem:general_statement} is the following. That $\mathcal{U}$ is open follows from Proposition \ref{proposition:stability_resonances}. In order to prove that $\mathcal{U}$ is open, we consider a map $T_0 \in \Exp(M)$ and associated resonances $\lambda_1,\dots,\lambda_{n_0 + n_1 + n_2}$ as in the statement of Theorem \ref{theorem:general_statement}. For $T$ a small perturbation of $T_0$, the resonances $\lambda_1,\dots,\lambda_{n_0 + n_1 + n_2}$ for $T_0$ are deformed into resonances $\lambda_1(T),\dots,\lambda_{n_0 + n_1 + n_2}(T)$. We would like to prove that the map
\begin{equation*}
\Phi : T \mapsto (E_{T,\lambda_1(T)},\dots,E_{T,\lambda_{n_0+n_1+n_2}(T)}, \lambda_{1 + n_0}(T),\dots,\lambda_{n_0 +n_1 + n_2}(T))
\end{equation*}
has a local right inverse near $\Phi(T_0)$. In \S \ref{subsection:local_solvability}, we study the surjectivity of the derivative of $\Phi$ at $T_0$ and use Nash--Moser theory to give a sufficient condition on $T_0$ for the existence of a right inverse for $\Phi$ near $T_0$ (Lemma \ref{lemma:local_solvability}). This condition is always satisfied when $n_0 = 1$ and $n_1 = n_2 = 0$, which proves Theorem \ref{theorem:generic_density}. In \S \ref{subsection:generic_local_solvability}, we prove that our sufficient condition for the existence of a right inverse for $\Phi$ is generically satisfied, provided the degree of $T_0$ is large enough (Lemma \ref{lemma:generic_solvability}).

Before studying the map $\Phi$, we make a few remarks on the infinite dimensional manifolds that are involved in Theorem \ref{theorem:general_statement} in \S \ref{subsection:tangent_spaces}.

\subsection{Infinite dimensional manifolds}\label{subsection:tangent_spaces}

The point of this section is not to give a full account of the notions of Nash--Moser theory that we are going to use. We will apply this theory as it is exposed in \cite{hamilton_ift}. In particular, we will use the same terminology as in this reference. We let the curious reader refer to \cite{hamilton_ift} for an in-depth exposition of Nash--Moser theory, including in particular the definitions of smooth tame map, tame Fréchet space and tame Fréchet manifold that we will use below.

The point of this subsection is to clarify certain facts about the structure of tame Fréchet manifolds of certain sets involved in the following subsections. 

\subsubsection{The space of expanding maps}

We already discussed the structure of Fréchet manifold of $\Exp(M)$ in \S \ref{subsection:linear_response_theory}. The same process endows $\Exp(M)$ with a structure of tame Fréchet manifold: $\Exp(M)$ is an open subset of $C^\infty(M,M)$, and the latter has a standard structure of tame Fréchet manifold, see \cite[Corollary II.2.3.2]{hamilton_ift}. As in \S \ref{subsection:linear_response_theory}, if $F_0 \in \Exp(M)$, then we can define a map $X \mapsto F_X$ by \eqref{eq:coordinate_expm}. This map is a smooth tame diffeomorphism from a neighbourhood of $0$ in $\Gamma(\T M)$ and a neighbourhood of $F_0$ in $\Exp(M)$. Actually, considering the proof of \cite[Corollary II.2.3.2]{hamilton_ift}, we see that we could use such maps to define an atlas for $\Exp(M)$.

As explained in \S \ref{subsection:linear_response_theory}, the derivative of $X \mapsto F_X$ at $0$ induces an identification between $\Gamma(\T M)$ and $\T_{F_0} M$. Hence, we have an identification of $\Exp(M) \times \Gamma(\T M)$ with $\T \Exp(M)$. In coordinates of the form $X \mapsto F_X$, the identification $\T \Exp(M) \simeq \Exp(M) \times \Gamma(\T M)$ is given by nonlinear partial differential operators, which are smooth tame maps according to \cite[Corollary II.2.2.7]{hamilton_ift} and the remark following it. Consequently, $\T \Exp(M)$ identifies with $\Exp(M) \times \Gamma(\T M)$ as a smooth tame vector bundle (which is crucial when checking the hypotheses of \cite[Theorem III.1.1.3]{hamilton_ift}). We will systematically use this identification in the rest of the section.

\subsubsection{Spaces $PC_0^\infty(M,\mathbb{K})$}

The other spaces that we need to consider are $PC_0^\infty(M,\mathbb{K})$ for $\mathbb{K} = \mathbb{R}$ or $\mathbb{C}$ and $PC^\infty(M,\mathbb{R})$. We recall that we endowed $PC_0^\infty(M,\mathbb{K})$ with a topology by identifying it with the quotient of $C_0^\infty(M,\mathbb{K}) \setminus \set{0}$ under the action of $\mathbb{K}^*$. Let us now endow $PC_0^\infty(M,\mathbb{K})$ with a structure of tame Fréchet manifold. We will do it through the use of ``affine'' charts.

Consider $f_0 \in C_0^\infty(M,\mathbb{K}) \setminus \set{0}$ and a continuous $\mathbb{K}$-linear form $l_0$ on the space $C_0^\infty(M,\mathbb{K})$ such that $l_0(f_0) \neq 0$. Define then the map
\begin{equation}\label{eq:affine_parametrization}
\begin{array}{ccccc}
\Psi_{f_0,l_0} & : & \ker l_0 & \to & PC_0^\infty(M,\mathbb{K}) \\
 & & g & \mapsto & [f_0 + g].
\end{array}
\end{equation}
Notice that $\Psi_{f_0,l_0}$ is a homeomorphism on its image, that we shall denote by\footnote{The set $\mathfrak{U}_{l_0}$ indeed only depends on $l_0$: it is the set of lines in $C_0^\infty(M,\mathbb{K})$ on which the restriction of $l_0$ is non-trivial.} $\mathfrak{U}_{l_0}$. We claim that $\mathcal{A} = \set{(\mathfrak{U}_l,\Psi_{f,l}^{-1}) : f \in C_0^\infty(M,\mathbb{K}), l \in C_0^\infty(M,\mathbb{K})^*, l(f) \neq 0}$ is an atlas of tame Fréchet manifold for $PC_0^\infty(M,\mathbb{K})$.

Indeed, $\ker l_0$ is a tame direct summand of $C^\infty(M,\mathbb{K})$ and thus a tame Fréchet space. Moreover, if $f_1 \in C_0^\infty(M,\mathbb{K})$ and $l_1 \in C_0^\infty(M,\mathbb{K})^*$ are such that $l_1(f_1) \neq 0$ the change of coordinates $\Psi_{f_1,l_1}^{-1} \circ \Psi_{f_0,l_0}$ is defined on the open set $\set{g \in \ker l_0 : l_1(g) \neq - l_1(f_0)} = \Psi_{f_0,l_0}^{-1}(\mathfrak{U}_{l_0} \cap \mathfrak{U}_{l_1})$ by $g \mapsto \frac{l_1(f_1)}{l_1(f_0 + g)}(f_0 +g) - f_1$. This is a smooth tame map as a composition of smooth tame maps.

Notice that in the chart $(\mathfrak{U}_{l_0},\Psi_{f_0,l_0}^{-1})$, the projection $C_0^\infty(M,\mathbb{K}) \setminus \set{0} \to PC_0^\infty(M,\mathbb{K})$ becomes $f \mapsto \frac{l_0(f_0)}{l_0(f)}f - f_0$ from $\set{ f \in C_0^\infty(M,\mathbb{K}) \setminus \set{0} : l_0(f) \neq 0}$ to $\ker l_0$. We may parametrize $\set{ f \in C_0^\infty(M,\mathbb{K}) \setminus \set{0} : l_0(f) \neq 0}$ by $\mathbb{K}^* \times \ker l_0$ using the map $(\lambda,g) \mapsto \lambda(f_0 + g)$. Using these coordinates, the projection $C_0^\infty(M,\mathbb{K}) \setminus \set{0} \to PC_0^\infty(M,\mathbb{K})$ turns into the map $(\lambda,g) \mapsto g$. Hence, the projection $C_0^\infty(M,\mathbb{K}) \setminus \set{0} \to PC_0^\infty(M,\mathbb{K})$ is a smooth tame map and a submersion, which justifies our choice of tame Fréchet manifold structure on $PC_0^\infty(M,\mathbb{K})$. It follows also that the tangent space to $PC_0^\infty(M,\mathbb{K})$ at $[f]$ may be identified with the kernel of the derivative of the projection $C_0^\infty(M,\mathbb{K}) \setminus \set{0} \to PC_0^\infty(M,\mathbb{K})$ at $f$, that is with $C_0^\infty(M,\mathbb{K}) / \langle f \rangle$.

The space $PC^\infty(M,\mathbb{R})$ also plays a small role in Theorem \ref{theorem:general_statement}. Its structure of tame Fréchet manifold may be described as above (taking $\mathbb{K} = \mathbb{R}$ and dropping the index $0$).

\subsection{Sufficient condition for the existence of a right inverse}\label{subsection:local_solvability}

In this section, we use the setting and notations of \S \ref{section:key_lemma}. If $1$ is among $\lambda_1,\dots,\lambda_p$ then we assume that $\lambda_1 = 1$. In that case, we set $p_0 = 1$ and $p_1 = p-1$. Otherwise, we set $p_0 = 0$ and $p_1 = p$. Consider the map
\begin{equation}\label{eq:map_phi}
\begin{array}{ccccc}
\Phi & : & U_0 & \to & P C^\infty(M,\mathbb{R})^{p_0} \times P C_0^\infty(M,\mathbb{R})^{p_1} \times PC_0^\infty(M,\mathbb{C})^{q} \times \mathbb{R}^{p_1} \times \mathbb{C}^q \\
    & & T & \mapsto & (E_{T,\lambda_1(T)},\dots,E_{T,\lambda_{p+q}(T)}, \lambda_{1 + p_0}(T),\dots,\lambda_{p+q}(T)).
\end{array}
\end{equation}
We recall that for $T \in U_0$, the map $\F_T : M \mapsto \mathbb{R}^p \times \mathbb{C}^q$ is defined by $\F_T(x) = (f_{T,j}(x))_{1 \leq j \leq p+q}$ for $x \in M$. The main result of this subsection is:

\begin{lemma}\label{lemma:local_solvability}
Let $T_0 \in U_0$. Assume that for every point $x$ in $M$ the family $(\F_T(y))_{y \in T_0^{-1}(\set{x})}$ spans $\mathbb{R}^p \times \mathbb{C}^q$. Then there is a neighbourhood $V_0$ of $\Phi(T_0)$ in $P C^\infty(M,\mathbb{R})^{p_0} \times P C^\infty(M,\mathbb{R})^{p_1} \times PC^\infty(M,\mathbb{C})^{q} \times \mathbb{R}^{p_1} \times \mathbb{C}^q$ and a smooth tame map $\Psi : V_0 \to U_0$ such that $\Phi(\Psi(\alpha)) = \alpha$ for every $\alpha \in V_0$.
\end{lemma}

We will use Lemma \ref{lemma:local_solvability} in conjonction with Lemma \ref{lemma:generic_solvability} in order to prove Theorem \ref{theorem:general_statement} (and thus Theorems \ref{theorem:generic_real_resonance} and \ref{theorem:generic_complex_resonance}). Theorem \ref{theorem:generic_density} however follows directly from Lemma \ref{lemma:local_solvability} as we explain now.

\begin{proof}[Proof of Theorem \ref{theorem:generic_density}]
The set $\mathcal{U}$ is clearly open, let us prove that it is dense. Let $F_0$ be any element of $\Exp(M)$. We consider the setting of \S \ref{section:key_lemma} in the specific case $p = 1,q= 0$ and $\lambda_1 = 1$, so that the map $\Phi$ is just $T \mapsto E_{T,1}$. Notice that $\F_{F_0}$ is a multiple of the density of the absolutely continuous invariant measure for $F_0$. In particular, the function $\F_{F_0}$ never vanishes, and thus for every $y \in M$ the number $\F_{F_0}(y)$ spans $\mathbb{R}$. Hence, the hypothesis of Lemma \ref{lemma:local_solvability} is satisfied. 

Let $\Psi$ be given Lemma \ref{lemma:local_solvability}. We identify $\set{ \rho \in C^\infty(M,\mathbb{R}_+^*) : \int_M \rho \mathrm{d}x = 1}$ with an open subset of $PC^\infty(M,\mathbb{R})$ and consider a sequence $(\alpha_n)_{n \geq 0}$ of elements of $U \cap V_0$ that converges to $E_{F_0,1}$. For every $n \geq 0$, we have $\Psi(\alpha_n) \in \mathcal{U}$, and thus $F_0$ belongs to the closure of $\mathcal{U}$. We proved that $\mathcal{U}$ is dense.
\end{proof}

The proof of Lemma \ref{lemma:local_solvability} is based on an application of the inverse function theorem of Nash and Moser, or rather to the corresponding result for maps with surjective derivatives \cite[Theorem III.1.1.3]{hamilton_ift}. A crucial feature of this result as it is exposed in \cite{hamilton_ift} is the notion of smooth tame map (see section II.2 of this reference). Hence, we need to update the results presented in \S \ref{subsection:linear_response_theory} by replacing smoothness by ``smooth tameness''. We start with the analogues of Lemmas \ref{lemma:explicit_derivative} and \ref{lemma:useful_linear_response} respectively.

\begin{lemma}\label{lemma:transfer_smooth_tame}
The map
\begin{equation*}
\begin{array}{ccc}
\Exp(M) \times C^\infty(M,\mathbb{C}) & \to & C^\infty(M,\mathbb{C}) \\
(T, f) & \mapsto & \mathcal{L}_T f
\end{array}
\end{equation*}
is smooth tame.
\end{lemma}

\begin{lemma}\label{lemma:tame_resolvent}
Let $T_0 \in \Exp(M)$. Let $V$ be an open relatively compact subset of $\mathbb{C}^*$ such that $\overline{V} \cap \res(T_0) = \emptyset$. There is an open neighbourhood $U$ of $T_0$ in $\Exp(M)$ such that for every $T \in U$ the map $T$ has no resonance in $\overline{V}$. Moreover, the map $(T,z,f) \mapsto R_T(z) f$ is smooth tame from $U \times V \times C^\infty(M,\mathbb{C})$ to $C^\infty(M,\mathbb{C})$.
\end{lemma}

In order to prove these two results, we revisit a famous estimate.

\begin{lemma}[Tame Doeblin--Fortet--Lasota--Yorke inequality]\label{lemma:dfly}
Let $T_0$ be an expanding map on $M$. For $X \in \Gamma(\T M)$, let $T_X$ be the map defined by \eqref{eq:coordinate_expm} with $F_0$ replaced by $T_0$. There are $\theta,\theta_0 > 0$ and a neighbourhood $U$ of $0$ in $\Gamma(\T M)$ such that:
\begin{itemize}
\item for every $X \in U$, the map $T_X$ belongs to $\Exp(M)$;
\item for every $r \in \mathbb{N}^*$, there is a constant $C_r$ such that for every $n \geq 0$ there is a constant $C_{r,n}$ such that for every $f \in C^\infty(M,\mathbb{C})$ and $X \in U$ we have
\begin{equation}\label{eq:tame_lasota_yorke}
\n{\mathcal{L}_{T_X}^n f}_{C^r} \leq C_r e^{n(\theta_0 - r \theta)} \n{f}_{C^r} + C_{r,n}(1 + \n{X}_{C^{r+1}})\n{f}_{C^0}.
\end{equation}
\end{itemize}
\end{lemma}

\begin{remark}
In Lemma \ref{lemma:dfly}, we use the coordinates on $\Exp(M)$ given by \eqref{eq:coordinate_expm} because it is more convenient in order to state a tame estimate. It is crucial in Lemma \ref{lemma:dfly} that the neighbourhood $U$ of $T_0$ does not depend on $r$ and that the constants $C_r$ and $C_{r,n}$ do not depend on $X$. We did not specify a definition for the $C^r$ norm because it is irrelevant.
\end{remark}

\begin{proof}[Proof of Lemma \ref{lemma:dfly}]
By taking $U$ small enough, we ensure that for every $X \in U$ the map $T_X$ belongs to $\Exp(M)$ and satisfies \eqref{eq:definition_expanding} with constants $C$ and $\theta$ that do not depend on $X$. Moreover, we assume that $U$ is connected, which ensures that for every $X \in U$ the degree of $T_X$ coincides with the degree of $T_0$. We also assume that $U$ is $C^2$ bounded.

Let $r \in \mathbb{N}^*$. In the following $C_r$ denotes a constant that could be fixed at this point of the proof (i.e. it may depend on $r$ and $U$ only), but whose actual value may change from line to line. We will also consider a parameter $n$ and denote by $C_{r,n}$ a constant that could be fixed after the choice of $r$ and $n$ (i.e. it only depends on $U, r$ and $n$), but whose actual value may change from line to line.

Pick $f \in C^\infty(M,\mathbb{C})$. For $x \in M, X \in \Gamma(\T M)$ and $n \in \mathbb{N}$, we have
\begin{equation*}
\mathcal{L}_{T_X}^n f(x) = \sum_{T_X^n(y) = x} \frac{f(y)}{|\det D T_X^n(y)|}.
\end{equation*}
Consider $\alpha = (\alpha_1,\dots,\alpha_d) \in \mathbb{N}^d$ such that $|\alpha| \leq r$, where $d$ denotes the dimension of $M$. We can compute $\partial^\alpha(\mathcal{L}_{T_X}^n f)(x)$ for $x$ in some coordinate patches by iterated application of the chain rule and Leibniz formula. We find
\begin{equation}\label{eq:derivative_transfer_coordinates}
\begin{split}
& \partial^\alpha(\mathcal{L}_{T_X}^n f)(x) \\ & \qquad \qquad = \underbrace{\sum_{T_X^n(y) = x} \frac{\mathrm{d}^{|\alpha|} f(y) \cdot ((D T_X^n(y))^{-1} \cdot e_{j_1},\dots,(D T_X^n(y))^{-1} \cdot e_{j_{|\alpha|}} )}{|\det D T_X^n(y)|}}_{= A_{X,n,\alpha}f (x)} \\ & \qquad \qquad \qquad \qquad \qquad \qquad \qquad \qquad \qquad \qquad \qquad \qquad \qquad \qquad + B_{X,n,\alpha}f(x),
\end{split}
\end{equation}
where $(e_1,\dots,e_d)$ is the canonical basis for $\mathbb{R}^d$ and $j_1,\dots,j_{|\alpha|} \in \set{1,\dots,d}$ are such that $\alpha = e_{j_1} + \dots + e_{j_r}$. The term $B_{X,n,\alpha}f(x)$ is a sum of expressions of the form\footnote{In order to keep relatively simple notations, we decided to ignore the change of variable maps when working in coordinates in this proof. This is not an issue because we can cover $M$ by finitely many coordinates patches. Notice however that when we evaluate a term of the form \eqref{eq:other_terms_lasota_yorke} or \eqref{eq:other_terms_doeblin_fortet} the point $y$ will change coordinate patch (as $x$ moves) more and more often as $n$ goes to $+ \infty$. One could be afraid that this could make us lose uniformity in the some estimates (e.g. when applying an interpolation argument), but it is harmless eventually because we keep track of this dependence on $n$ in our notation and that in the end it only impacts the second term in \eqref{eq:tame_lasota_yorke}.}:
\begin{equation}\label{eq:other_terms_lasota_yorke}
\partial^{\beta_1} h_{i_1,j_1}(y) \dots \partial^{\beta_{|\alpha|}}h_{i_{|\alpha|},j_{|\alpha|}}(y) \partial^\gamma \left( \frac{f}{|\det D T_X^n|} \right)(y),
\end{equation}
where $y$ is an antecedent of $x$ by $T_X^n$, the $h_{i,j}(y)$'s are the entries of the matrix $D T_X^n(y)^{-1}$ and $\beta_1 + \dots + \beta_{|\alpha|} + \gamma = \alpha$ with $\gamma \neq \alpha$. Applying Leibniz rule to expand the last factor in \eqref{eq:other_terms_lasota_yorke}, we find that $B_{X,n,\alpha}f(x)$ is a sum of terms of the form:
\begin{equation}\label{eq:other_terms_doeblin_fortet}
g(y) \partial^{\beta_1} h_{i_1,j_1}(y) \dots \partial^{\beta_{|\alpha|+d}}h_{i_{|\alpha|+d},j_{|\alpha|+d}}(y) \partial^\gamma f(y)
\end{equation}
with $\beta_1 + \dots + \beta_{|\alpha|+d} + \gamma \leq \alpha$ and $\gamma \neq \alpha$. The factor $g(y)$ is a smooth function that comes from the comparison of the Euclidean metric and the metric on $M$ when computing the Jacobian determinant of $T_X^n$. By interpolation inequalities and recalling that $|\gamma| \leq r-1$, we can bound the term \eqref{eq:other_terms_doeblin_fortet} by\footnote{When we take a $C^k$ norm of $(DT_X^n)^{-1}$, we mean the norm of the map $y \mapsto (DT_X^n(y))^{-1}$ defined in a domain of a coordinates patch.}
\begin{equation*}
C_{r,n} \n{(DT_X^n)^{-1}}_{C^0}^{|\alpha| + d-1} \left( \n{(D T_X^n)^{-1}}_{C^1} \n{f}_{C^{r-1}} + \n{(D T_X^n)^{-1}}_{C^r} \n{f}_{C^0} \right).
\end{equation*}
Since $U$ is $C^2$ bounded, this quantity is less than
\begin{equation*}
C_{r,n}\left( \n{f}_{C^{r-1}} + \n{(DT_X^{n})^{-1}}_{C^r} \n{f}_{C^0} \right)
\end{equation*}
Since $y \mapsto (DT_X^n (y))^{-1}$ is the composition of $y \mapsto DT_X^n(y)$ and matrix inversion, using the chain rule and an interpolation argument as above (or more directly as in the proof of \cite[Lemma II.2.3.4]{hamilton_ift}), we can bound $\n{(DT_X^{n})^{-1}}_{C^r}$ by $C_{r,n}(1+\n{(DT_X^{n})}_{C^r})$. The map $T_X$ is also a composition of $X$ and a fixed smooth function, so that $\n{(DT_X^{n})}_{C^r} \leq C_{r,n}(1 + \n{X}_{C^{r+1}})$. Consequently, we get that the term $B_{X,n,\alpha}f(x)$ in \eqref{eq:derivative_transfer_coordinates} is bounded by 
\begin{equation*}
C_{r,n}\left( \n{f}_{C^{r-1}} + (1 + \n{X}_{C^{r+1}}) \n{f}_{C^0} \right).
\end{equation*}
The term $A_{X,n,\alpha} f(x)$ on the other hand can be bounded using \eqref{eq:definition_expanding} by
\begin{equation*}
C_r e^{- n (|\alpha|+d) \theta} (\deg T_0)^{n} \n{f}_{C^{|\alpha|}}.
\end{equation*}
Gathering the estimates above we get
\begin{equation*}
\begin{split}
\n{\mathcal{L}_{T_X}^n f}_{C^r} & \leq C_r e^{- n (r+d) \theta} (\deg T_0)^{n} \n{f}_{C^{r}} \\ & \qquad \qquad \qquad \qquad + C_{r,n} \n{f}_{C^{r-1}} + C_{r,n}(1 + \n{X}_{C^{r+1}}) \n{f}_{C^0}.
\end{split}
\end{equation*}
Using that $\n{f}_{C^{r-1}} \leq C_r e^{- n (r+d) \theta} (\deg T_0)^{n} \n{f}_{C^{r}} + C_{r,n} \n{f}_{C^0}$, we get rid of the term $C_{r,n} \n{f}_{C^{r-1}}$, and the result follows with $\theta_0 = \log \deg T_0$.
\end{proof}

Let us now prove Lemmas \ref{lemma:transfer_smooth_tame} and \ref{lemma:tame_resolvent}.

\begin{proof}[Proof of Lemma \ref{lemma:transfer_smooth_tame}]
It follows from Lemma \ref{lemma:dfly} that $(T,f) \mapsto \mathcal{L}_T f$ is tame. We recall that, using the identification of the tangent bundle of $\Exp(M)$ with $\Exp(M) \times \Gamma(\T M)$, the partial derivative of $(T,f) \mapsto \mathcal{L}_T f$ with respect to $T$ is $(T,f,X) \mapsto - \mathcal{L}_T(\Div(fX))$ (we do not need to consider the partial derivative with respect to $f$ since the operator that we study is linear in $f$). Notice that $(f,X) \mapsto \Div(fX)$ is tame linear and thus smooth tame. It follows that $(T,f,X) \mapsto - \mathcal{L}_T(\Div(fX))$ is tame. By induction, we find that $(T,f) \mapsto \mathcal{L}_T f$ is smooth tame. 
\end{proof}

\begin{proof}[Proof of Lemma \ref{lemma:tame_resolvent}]
The existence of $V$ is guaranteed by Lemma \ref{lemma:useful_linear_response}. In order to prove that $(T,z,f) \mapsto R_T(z) f$ is smooth tame, it follows from \cite[Theorem II.3.1.1]{hamilton_ift} that we only need to prove that $(T,z,f) \mapsto R_T(z) f$ is tame. 

Let $z_0 \in V$ and $\widetilde{T}_0 \in V$. We saw in the proof of Lemma \ref{lemma:useful_linear_response} that there are neighbourhoods $V_0$ of $z_0$ in $V$ and $\widetilde{U}_0$ of $\widetilde{T}_0$ in $U_0$ such that there are $C > 0$ and $\ell \in \mathbb{N}^*$ such that for every $z \in V_0, T \in \widetilde{U}$ and $f \in C^\infty(M,\mathbb{C})$ we have $\n{R_{T}(z)f}_{C^0} \leq C \n{f}_{C^\ell}$. Up to making $\widetilde{U}_0$ smaller, we may assume that it has the form $\widetilde{U}_0 = \set{ \widetilde{T}_X : X \in \widetilde{U}}$ where $\widetilde{U}$ is the neighbourhood of $0$ in $\Gamma(\T M)$ obtained by applying Lemma \ref{lemma:dfly} to $\widetilde{T}$ (and $\widetilde{T}_X$ defined by \eqref{eq:coordinate_expm} with $F_0$ replaced by $\widetilde{T}_0$).

For $X \in \widetilde{U},z \in V_0$ and $n \in \mathbb{N}^*$, let $G_{X,n}(z) = \sum_{k =0}^{n-1} z^{-k-1} \mathcal{L}_{\widetilde{T}_X}^k$, and notice that $I = G_{X,n}(z)(z - \mathcal{L}_{\widetilde{T}_X}) + z^{-n} \mathcal{L}_{\widetilde{T}_X}^n$. Consequently, for $r \in \mathbb{N}^*, n \in \mathbb{N}, X \in \widetilde{U}, z \in V_0$ and $f \in C^\infty(M,\mathbb{C})$ we have applying Lemma \ref{lemma:dfly}:
\begin{equation}\label{eq:vers_tame}
\begin{split}
\n{f}_{C^r} \leq \n{G_{X,n}(z)(z - \mathcal{L}_{\widetilde{T}_X})f}_{C^r} & + C_r \left(\frac{e^{\theta_0}}{|z|}\right)^{n} e^{-n r \theta}\n{f}_{C_r} \\ & \qquad \qquad + C_{r,n}(1 + \n{X}_{C^{r+1}}) \n{f}_{C^0}.
\end{split}
\end{equation}
Using Lemma \ref{lemma:dfly}, we get that\footnote{We use the same conventions concerning the constant $C_r$ and $C_{r,n}$ in this proof as in the proof of Lemma \ref{lemma:dfly}.}
\begin{equation*}
\begin{split}
\n{G_{X,n}(z)(z - \mathcal{L}_{\widetilde{T}_X})f}_{C^r} & \leq C_{r,n} \n{(z - \mathcal{L}_{\widetilde{T}_X})f}_{C^r} \\ & \qquad \qquad + C_{r,n}(1+\n{X}_{C^{r+1}}) \n{(z - \mathcal{L}_{\widetilde{T}_X})f}_{C^0}.
\end{split}
\end{equation*}
We can then impose $r \geq r_0$ with $r_0$ large enough so that $e^{\theta_0 - r_0 \theta}|z|^{-1} < 1/2$ for every $z \in V_0$. Hence, for $n$ large enough we have $C_r \left(\frac{e^{\theta_0}}{|z|}\right)^{n} e^{-n r \theta} \leq 1/2$, which allows us to get rid of the corresponding term in \eqref{eq:vers_tame}. Hence, we proved that for $r \in \mathbb{N}^*$ large enough there is a constant $C_r > 0$ such that for every $X \in \widetilde{U},z \in V_0$ and $f \in C^\infty(M,\mathbb{C})$ we have
\begin{equation*}
\n{f}_{C^r} \leq C_r \n{(z - \mathcal{L}_{\widetilde{T}_X})f}_{C^r} + C_r(1 + \n{X}_{C^{r+1}})\left( \n{(z- \mathcal{L}_{\widetilde{T}_X})f}_{C_0} + \n{f}_{C_0} \right). 
\end{equation*}
Replacing $f$ by $R_{\widetilde{T}_X}(z) f$ in this estimate yields:
\begin{equation*}
\begin{split}
\n{R_{\widetilde{T}_X}(z)f}_{C^r} & \leq C_r \n{f}_{C^r} + C_r(1 + \n{X}_{C^{r+1}}) \left(\n{f}_{C^0} + \n{R_{\widetilde{T}_X}(z)f}_{C_0} \right) \\ & \leq C_r \n{f}_{C^r} + C_r(1 + \n{X}_{C^{r+1}}) \n{f}_{C_\ell}.
\end{split}
\end{equation*}
This is the tame estimate we were looking for.
\end{proof}

Let us now consider the implication of Lemmas \ref{lemma:transfer_smooth_tame} and \ref{lemma:tame_resolvent} in terms of the map $\Phi$ from \eqref{eq:map_phi}.

\begin{lemma}\label{lemma:the_derivative}
The map $\Phi$ is smooth tame. Moreover, if $T_0 \in U_0$ then the derivative of $\Phi$ at $T_0$ is the map from $\Gamma(\T M)$ to $$(C^\infty(M,\mathbb{R})/ \langle f_{T,1} \rangle)^{p_0} \times \prod_{j = 1+p_0}^{p+q} C^\infty_0(M,\mathbb{K}_j) / \langle f_{T,j} \rangle \times \mathbb{R}^{p_1} \times \mathbb{C}^q,$$ where $\mathbb{K}_j = \mathbb{R}$ for $1 \leq j \leq p$ and $\mathbb{K}_j = \mathbb{C}$ for $p+1 \leq j \leq p + q$, given by
\begin{equation*}
X  \mapsto  (([H_{T_0,\lambda_j(T_0)} P_{T_0}(X) f_{T,j}])_{1 \leq j \leq p + q},(\nu_{T_0,j}(P_{T_0}(X)f_{T_0,j}))_{1+ p_0 \leq j \leq p+q}).
\end{equation*}
Here, we used the identifications of the different tangent spaces discussed in \S \ref{subsection:tangent_spaces}.
\end{lemma}

\begin{proof}
For $j = 1,\dots,p+q$, the smoothness of $T \mapsto \lambda_j(T)$ is given by Remark \ref{remark:useful_smoothness}, and since it is valued in a Banach space it is smooth tame. The formula for the derivative is given in Lemma \ref{lemma:derivative_simple_resonance}.

Let $j \in \set{1,\dots,p+q}$. Recall that for $T \in U_0$ we have
\begin{equation*}
f_{T,j} = \frac{1}{2 i \pi} \int_{\partial \mathbb{D}(\lambda_j,\epsilon)} R_{T}(z) f_{F_0,j} \mathrm{d}z
\end{equation*}
The smoothness of $T \mapsto f_{T,j}$ is given by Proposition \ref{proposition:smoothness_spectral_projector}, and the derivatives of this map may be computed by differentiation under the integral as explained in the proof of Proposition \ref{proposition:smoothness_spectral_projector}. In view of the formula given in Lemma \ref{lemma:useful_linear_response} for the derivative of $T \mapsto R_{T}(z)$, it follows from Lemmas \ref{lemma:transfer_smooth_tame} and \ref{lemma:tame_resolvent} that the derivatives (of any order) of $T \mapsto f_{T,j}$ are tame. Since the projection $C_0^\infty(M,\mathbb{K}_j) \setminus \set{0} \to PC_0^\infty(M,\mathbb{K}_j)$ is smooth tame, the result follows. The formula for the first derivative of $T \mapsto E_{T,\lambda_j(T)}$ is a consequence of Lemma \ref{lemma:derivative_simple_resonance} (see also Remark \ref{remark:derivative_simple_resonance}).
\end{proof}

We deduce from Lemma \ref{lemma:the_derivative} that in order to prove Lemma \ref{lemma:local_solvability} we only need to establish:

\begin{lemma}\label{lemma:inverse_after_identification}
Under the assumptions of Lemma \ref{lemma:local_solvability}, there is a neighbourhood $U_1$ of $T_0$ in $U_0$ and a smooth tame map $Q : U_1 \times (C_0^\infty(M,\mathbb{R})^p \times C_0^\infty(M,\mathbb{C})^q) \to \Gamma(\T M)$, linear in its second argument such that for every $T \in U_1$, $(g_j)_{1 \leq j \leq p + q} \in C_0^\infty(M,\mathbb{R})^p \times C_0^\infty(M,\mathbb{C})^q$ and $k \in \set{1,\dots,p+q}$ we have
\begin{equation*}
P_T(Q(T,(g_j)_{1 \leq j \leq p + q}))f_{T,k} = g_{T,k}.
\end{equation*}
\end{lemma}

The idea behind the proof of Lemma \ref{lemma:inverse_after_identification} is very similar to our solution to the equation ``$P_T(X) f_{T,1} = g$'' in Remark \ref{remark:proof_key}, except that the hypothesis in Lemma \ref{lemma:local_solvability} implies that we are always in the case in which this solution applies. Before proving Lemma \ref{lemma:inverse_after_identification}, let us explain why it implies Lemma \ref{lemma:local_solvability}.

\begin{proof}[Proof of Lemma \ref{lemma:local_solvability} from Lemma \ref{lemma:inverse_after_identification}]
Assume that Lemma \ref{lemma:inverse_after_identification} holds and let $U_1$ and $Q$ be as in this lemma. For $j = 1+p_0,\dots,p+q$ let $l_j$ be a continuous linear form on $C_0^\infty(M,\mathbb{K}_j)$ such that $l_j(f_{T_0,j}) = 1$. If $p_0 = 1$, let $l_1 : g \mapsto \int_M g \mathrm{d}x$ be a continuous linear form on $C^\infty(M,\mathbb{R})$. For $j = 1,\dots,p+q$, let then $\Psi_j = \Psi_{f_{T_0,j},l_j}$ be the parametrization for $PC^\infty(M,\mathbb{R}), PC_0^\infty(M,\mathbb{R})$ or $PC_0^\infty(M,\mathbb{C})$ given by \eqref{eq:affine_parametrization}, and let $W_j$ be its image (this is a neighbourhood of $E_{T_0,\lambda_j(T_0)}$). For $j = 1,\dots,p+q$ and $T$ close to $T_0$ define $\tilde{f}_{T,j} = \Psi_j^{-1}(E_{T,\lambda_j(T)})$. Instead of $\Phi$, we will invert the map $\widetilde{\Phi} : T \mapsto ((\tilde{f}_{T,j})_{1 \leq j \leq p+q},\lambda_j(T))_{1+p_0 \leq j \leq p+q})$ in a neighbourhood of $T_0$ (which is just $\Phi$ ``in coordinates'').

For $j = 1+p_0,\dots,p+q$ and $T$ near $T_0$, there is an isomorphism between $\ker l_j \times \mathbb{K}_j$ and $C_0^\infty(M,\mathbb{K}_j)$ given by $(f,\tau) \mapsto l_j(f_{T,j}) (\lambda_j(T) - \mathcal{L}_T) f + \tau f_{T,j}$. The inverse of this isomorphism is $$h \mapsto (l_j(f_{T,j})^{-1} H_{T,\lambda_j(T)} h - l_j(f_{T,j})^{-2} l_j(H_{T,\lambda_j(T)}h) f_{T,j},\nu_{T,j}(h)),$$ see \eqref{eq:full_inverse}. If $p_0 = 1$, then notice that $\ker l_1 = C_0^\infty(M,\mathbb{R})$. Put all these isomorphisms together to get an identification
\begin{equation*}
\mathfrak{A}(T) : \left( \prod_{j = 1}^{p+q} \ker l_j \right) \times \mathbb{R}^{p_1} \times \mathbb{C}^{q} \to C_0^\infty(M,\mathbb{R})^p \times C_0^\infty(M,\mathbb{C})^q.
\end{equation*}
It follows from Lemma \ref{lemma:the_derivative} that, for $T$ near $T_0$, the map $\mathcal{A}(T) D \widetilde{\Phi}(T)$ is $X \mapsto (P_T(X) f_{T,j})_{1 \leq j \leq p +q}$, mapping $\Gamma(\T M)$ into $C_0^\infty(M,\mathbb{R})^p \times C_0^\infty(M,\mathbb{C})^q$. Lemma \ref{lemma:inverse_after_identification} asserts that this map has a left inverse (which is a smooth tame function of $T$ and $X$). Since $(T,\alpha) \mapsto \mathfrak{A}(T)^{-1}\alpha$ is a smooth tame map (it follows from the explicit formula for the inverse and Lemma \ref{lemma:tame_resolvent}), we find that $X \mapsto D \widetilde{\Phi}(T) \cdot X$ has a right inverse which is a smooth tame map of $T$ and $X$. The result then follows from \cite[Theorem III.1.1.3]{hamilton_ift}. 
\end{proof}

\begin{remark}
The isomorphism (depending on $T$) between the spaces $\ker l_j \times \mathbb{K}_j$ and $C_0^\infty(M,\mathbb{K}_j)$, for $j = 1+p_0,\dots,p+q$, that we use in the proof of Lemma \ref{lemma:local_solvability} may seem complicated. This is because we took coordinates on $PC_0^\infty(M,\mathbb{K}_j)$. In a more intrinsic fashion, the proof relies on the isomorphism between $C_0^\infty(M,\mathbb{K}_j)$ and $C_0^\infty(M,\mathbb{K}_j) / \langle f_{T,j} \rangle \times \mathbb{K}_j$ given by
\begin{equation*}
 g \mapsto ([H_{T,\lambda_j(T)} g],\nu_{T,j}(g))
\end{equation*}
whose inverse is $([g],\tau) \mapsto (\lambda_j(T) - \mathcal{L}_T) g + \tau f_{T,j}$, see \eqref{eq:full_inverse}. We need to take coordinates in the proof of Lemma \ref{lemma:local_solvability} in order to make sense of the fact that the map $(T,\alpha) \mapsto \mathfrak{A}(T)^{-1} \alpha$ is smooth tame for instance.
\end{remark}

The goal of the following lemma is to simplify the proof of Lemma \ref{lemma:inverse_after_identification}.

\begin{lemma}\label{lemma:tame_homology}
There is a tame linear map $A : C_0^\infty(M,\mathbb{R}) \to \Gamma(\T M)$ such that for every $f \in C_0^\infty(M,\mathbb{R})$ we have $\Div(Af) = f$.
\end{lemma}

\begin{proof}
Up to working on the bundle of orientation, we may assume that $M$ is orientable. Hence, the density $\mathrm{d}x$ identifies with a volume form. Consider the Hodge Laplacian $\Delta = \D \D^* + \D^* \D$, associated to the Riemannian metric on $M$, acting on form of degree $d = \dim M$ (so that the term $\D^* \D$ is actually not needed here). Let $H$ be the holomorphic part of the resolvent $(z - \Delta)^{-1}$ at $0$. If $f \in C_0^\infty(M,\mathbb{R})$, we have $f \mathrm{d}x = \omega - \Delta H f = \omega - \D \D^* H f$, where $\omega$ belongs to the kernel of $\Delta$. We have
\begin{equation*}
\int_M \omega = \int_M f \mathrm{d}x + \int_M \D \D^* H f = 0.
\end{equation*}
It follows that $\omega$ is cohomologically trivial, but then Hodge theory implies that $\omega = 0$. Hence, we have $f \mathrm{d}x = \D \D^* H f$. To the $(d-1)$-form $\D^* H f$ it corresponds a unique vector field $X_f$ such that $\D^* H f = i_{X_f} \mathrm{d}x$, and we have then $f = \Div(X_f)$. The map $f \mapsto \D^* H f$ is tame from $C_0 ^\infty(M,\mathbb{R})$ to the space of smooth $(d-1)$-form on $M$ because it is a pseudo-differential operator ($\Delta$ is elliptic). The identification between smooth $(d-1)$-forms and vector fields using $\mathrm{d}x$ also defines a tame operator (we just apply a linear map in each fiber, that depends smoothly on the point).
\end{proof}

For $T \in \Exp(M)$, define the operator $L_T$ from $\Gamma(\T M \otimes \mathbb{C})$ to itself by
\begin{equation*}
L_T(X) (x) = \sum_{Ty = x} \frac{D T(y) \cdot X(y)}{|\det DT(y)|}, \quad X \in \Gamma(\T M \otimes \mathbb{C}), \, x \in M.
\end{equation*}
Notice that if $X \in \Gamma(\T M \otimes \mathbb{C})$, then we have
\begin{equation*}
\mathcal{L}_T(\Div(X)) = \Div(L_T(X)). 
\end{equation*}
Hence, it follows from Lemma \ref{lemma:tame_homology} that, instead of Lemma \ref{lemma:inverse_after_identification}, we only need to prove:

\begin{lemma}\label{lemma:inverse_final_version}
Under the assumptions of Lemma \ref{lemma:local_solvability}, there is an open neighbourhood $U_1$ of $T_0$ in $U_0$ and a smooth tame map, linear in its second argument, $Q: U_1 \times (\Gamma(\T M)^p \times \Gamma(\T M \otimes \mathbb{C})^q) \to \Gamma(\T M)$ such that for every $T \in U_1$ and $\alpha \in \Gamma(\T M)^p \times \Gamma(\T M \otimes \mathbb{C})^q$ we have
\begin{equation}\label{eq:inverse_final_version}
(L_T(f_{T,j} Q(T,\alpha)))_{1 \leq j \leq p+q} = \alpha.
\end{equation}
\end{lemma}

\begin{proof}
Let $x_0 \in M$. We will prove that there are an open neighbourhood $W$ of $x_0$ in $M$, an open neighbourhood $U_1$ of $T_0$ in $U_0$ and a smooth tame map, linear in its second argument, $(T,\alpha) \mapsto Q(T, \alpha)$ from $U_1 \times (\Gamma(\T M)^p \times \Gamma(\T M \otimes \mathbb{C})^q)$ to $\Gamma(\T M)$ such that for every $T \in U_1$ and $\alpha \in \Gamma(\T M)^p \times \Gamma(\T M \otimes \mathbb{C})^q$ the relation \eqref{eq:inverse_final_version} holds if $\alpha$ is supported in $W$. The full result follows then by a partition of unity argument.

By assumption, there are points $y_1,\dots,y_{p+2q} \in M$ such that $T_0(y_1) = \dots = T_0(y_{p+2q}) = x_0$ and $(\F_{T_0}(y_j))_{1 \leq j \leq p+2q}$ is a basis of $\mathbb{R}^p \times \mathbb{C}^q$. For $j = 1,\dots,p+2q$, let $V_j$ be a neighbourhood of $y_j$ in $M$ such that $T_0$ induces a diffeomorphism from a neighbourhood of $\overline{V}_j$ to its image. Up to taking the $V_j$'s smaller, we may assume that for every $(z_1,\dots,z_{p+2q}) \in \overline{V}_1 \times \dots \times \overline{V}_{p+2q}$, the family $(\F_{T_0}(z_j))_{1 \leq j \leq p+2q}$ is a basis of $\mathbb{R}^p \times \mathbb{C}^q$ (in particular, the $\overline{V}_j$'s are disjoint). Choose $W$ an open neighbourhood of $x_0$ such that $\overline{W} \subseteq \bigcap_{j = 1}^{p+2q} T_0(V_j)$.

There is a neighbourhood $U_1$ of $T_0$ in $U_0$ such that for every $T \in U_1$ we have:
\begin{itemize}
\item $\overline{W} \subseteq \bigcap_{j = 1}^{p+2q} T(V_j)$;
\item for $j = 1,\dots,p+2q$ the map $T$ induces a diffeomorphism from a neighbourhood of $\overline{V}_j$ to its image;
\item for every $(z_1,\dots,z_{p+2q}) \in \overline{V}_1 \times \dots \times \overline{V}_{p+2q}$, the family $(\F_T(z_j))_{1 \leq j \leq p+2q}$ is a basis of $\mathbb{R}^p \times \mathbb{C}^q$.
\end{itemize}
Now, consider $T \in U_1$ and $\alpha = (X_1,\dots,X_{p+q}) \in \Gamma(\T M)^p \times \Gamma(\T M \otimes \mathbb{C})^q$ supported in $W$. For every $x \in W$, since $(\F_{T}(T_{|V_j}^{-1}(x)))_{1 \leq j \leq p+2q}$ is a basis of $\mathbb{R}^p \times \mathbb{C}^q$, there are $Y_1(x),\dots,Y_{p+2q}(x) \in \T_x M$, uniquely defined, such that
\begin{equation*}
X_j(x) = \sum_{k = 1}^{p+2q} f_{T,j}(T_{|V_k}^{-1}(x)) Y_k(x)
\end{equation*}
for $j = 1,\dots,p+2q$. We set then $Q(T, \alpha) = Z$, where $Z$ is the vector field defined by
\begin{equation*}
Z(x) = \begin{cases} |\det DT(x)| D T(x)^{-1} \cdot Y_k(x)  & \textup{ if } x \in T_{|V_k}^{-1}(W), k \in \set{1,\dots,p+2q}, \\
         0 & \textup{ if } x \in M \setminus \bigcup_{k= 1}^n V_k,
       \end{cases}
\end{equation*}
for $x \in M$.

Finally $(T,\alpha) \mapsto Q(T,\alpha)$ is a smooth tame map because for $j = 1,\dots,p+q$ the map $T \mapsto f_{T,j}$ is smooth tame (as a consequence of Lemma \ref{lemma:tame_resolvent}) and the other components of the construction are local inverses of $T$, composition of smooth maps and linear algebra formulae (the first two can be dealt with as in the proof of Lemma \ref{lemma:dfly}, see also \cite[Lemmas II.2.3.4 and II.2.3.6]{hamilton_ift}).
\end{proof}

As explained above, Lemma \ref{lemma:inverse_final_version} implies Lemma \ref{lemma:inverse_after_identification} which itself implies Lemma \ref{lemma:local_solvability}.

\subsection{Generic existence of a right inverse}\label{subsection:generic_local_solvability}

In this subsection, we use the setting and notations of \S \ref{section:key_lemma}. Our goal is the following lemma:

\begin{lemma}\label{lemma:generic_solvability}
Let $\mathcal{V}$ be the set of $T \in U_0$ such that for every $x \in M$ the family $(\F_T(y))_{y \in T^{-1}(\set{x})}$ spans $\mathbb{R}^p \times \mathbb{C}^q$. Assume that $\deg F_0 \geq \dim M + p + 2q$. The set $\mathcal{V}$ is open and dense in $U_0$.
\end{lemma}

Before starting the proof of Lemma \ref{lemma:generic_solvability}, let us explain why it implies Theorem \ref{theorem:general_statement}.

\begin{proof}[Proof of Theorem \ref{theorem:general_statement}]
Proposition \ref{proposition:stability_resonances} implies that $\mathcal{U}$ is open. Let us prove that it is dense in $\Exp_{\geq m}(M)$. Let $V$ be an open subset of $\Exp_{\geq m}(M)$. According to Theorem \ref{theorem:generic_simple}, there is $F_0 \in V$ such that all resonances of $F_0$ of modulus larger than or equal to $\delta$ are simple.

Let $\lambda_1,\dots,\lambda_{n_0 + n_1 + n_2}$ be resonances for $F_0$ as in the statement of Theorem~\ref{theorem:general_statement}. With these resonances, we are in the setting of \S \ref{section:key_lemma} with $p = n_0 + n_1$ and $q = n_2$. It follows from Lemma \ref{lemma:generic_solvability} that there is a $T_0$, arbitrarily close to $F_0$ such that the map $\Phi$ from \eqref{eq:map_phi} has a local right inverse near $\Phi(T_0)$. Using the local right inverse for $\Phi$, we find as in the proof of Theorem \ref{theorem:generic_density} that there is $T_1$, arbitrarily close to $T_0$ and thus to $F_0$ (in particular it can be chosen in $V$) such that
\begin{equation}\label{eq:first_condition}
(E_{T_1,\lambda_1(T_1)},\dots,E_{T_1,\lambda_{p+q}(T_1)},\lambda_{1 + n_0}(T_1),\dots,\lambda_{p+q}(T_1)) \in U.
\end{equation}
This is not enough to get $T_1 \in \mathcal{U}$, because we considered only one family of resonances. However, expanding maps near $F_0$ only have finitely many family of resonances as in the statement of Theorem \ref{theorem:general_statement}. Hence, we can repeat the process above, starting this time from $T_1$, with another family of resonances. If we stay close enough from $T_1$ the condition \eqref{eq:first_condition} will be preserved. Iterating this process, we end up with an element of $\mathcal{U}\cap V$, and prove that $\mathcal{U}$ is dense.
\end{proof}

\begin{remark}\label{remark:sketch_proof}
A significant part of the technicality in the proof of Lemma \ref{lemma:generic_solvability} comes from the fact that we are considering several resonances at once, and that we allow complex resonances. Consequently, let us give a sketch of the proof of Lemma \ref{lemma:generic_solvability} in the case of a single real resonance (i.e. $p_0 = q = 0$ and $p = p_1 = 1$). To lighten notations, we will write $f_T$ instead of $f_{T,1}$. The condition on $m$ in that case is $m \geq d+1$.

We want to find $T \in U_0$ such that for every $x \in M$ there is $y \in M$ such that $T(y) = x$ and $f_{T}(y) \neq 0$. A natural idea to do so is to consider for $T \in U_0$ a map $\widetilde{\mathcal{G}}_T : M \to \mathbb{R}^m$ that maps a point $x$ to the collection of the values of $f_T$ at the antecedents of $x$ by $T$. By working locally on $M$, we can lift the ambiguity in the ordering of the antecedents to get a map that actually takes value in $\mathbb{R}^m$. Hence, we need to work with several maps defined on different subsets of $M$, we will ignore this issue in this sketch of proof. Notice that by assumption $m > \dim M$, hence if $0$ was a regular value for $\widetilde{\mathcal{G}}_T$ we would know that $\widetilde{\mathcal{G}}_T$ does not vanish and we would be done. Unfortunately, we cannot produce enough perturbations of $\widetilde{\mathcal{G}}_T$ using Lemma \ref{lemma:key_deformation} to write a transversality argument that would make $0$ a regular value of $\widetilde{\mathcal{G}}_T$ when $m = \dim M+1$. 

To tackle this difficulty, we drop components of $\widetilde{\mathcal{G}}_T$ (at least one) to get a map $\mathcal{G}_T$ from $M$ to $\mathbb{R}^{\dim M}$. We can now use Lemma \ref{lemma:key_deformation} and a transversality argument to find $T_0 \in U_0$ such that $0$ is a regular value of $\mathcal{G}_{T_0}$. According to Theorem \ref{theorem:generic_morse}, we may also assume that $0$ is a regular value for $f_{T_0}$. Notice that there are finitely many points $x \in M$ such that $f_{T_0}(y) = 0$ for every $y$ such that $T_0(y) = x$, because such an $x$ is a zero of $\mathcal{G}_{T_0}$. Consider such a point $x_0$ (if there are none, we are already done). Since $x_0$ is a regular zero for $\mathcal{G}_{T_0}$, for $T$ near $T_0$ there is a unique point $x_0(T)$ near $x_0$ such that $\mathcal{G}_T(x_0(T)) = 0$. Hence, $x_0(T)$ is the only point near $x_0$ that can have the property that for every $y \in M$ such that $T y = x_0(T)$ we have $f_T(y) = 0$. Using Lemma \ref{lemma:key_deformation}, we can construct a deformation $t \mapsto T_t$ of $T_0$ which induces any first order deformation of $\mathcal{G}_{T_0}$ at $x_0$. Hence, we can impose $\frac{\mathrm{d}}{\mathrm{d}t}(x_0(T_t))_{t= 0}$. If we can arrange so that $\frac{\mathrm{d}}{\mathrm{d}t}(f_{T_t})_{t= 0}$ is zero near $x_0$, then since $0$ is a regular value of $f_{T_0}$, we can make it so that $f_{T_t}(x_0(T_t)) \neq 0$ for $t$ small. By the eigenvalue equation for $f_{T_t}$, it implies that there is $y \in M$ such that $T_t(y) =x_0(T_t)$ and $f_{T_t}(y) \neq 0$. Hence, we reduced the number of problematic points, iterating this argument ends the proof.

In order to achieve the condition $\frac{\mathrm{d}}{\mathrm{d}t}(f_{T_t})_{t= 0} = 0$ near $x_0$, we need to ensure that $x_0$ is not a fixed point for $T_0$ (see Lemma \ref{lemma:key_deformation}). This possibility is proven by contradiction, using Lemma \ref{lemma:key_deformation} again.
\end{remark}

The rest of this subsection is dedicated to the proof of Lemma \ref{lemma:generic_solvability}. Hence, we assume that $\deg F_0 \geq \dim M + p + 2q$. It follows from Proposition \ref{proposition:stability_resonances} that $\mathcal{V}$ is open, so we only need to prove that it is dense. Let $U$ be an open subset of $U_0$. We want to prove that $\mathcal{V} \cap U \neq \emptyset$. Let $d = \dim M$. Applying the implicit function theorem as in the proof of Lemma \ref{lemma:explicit_derivative} and up to making $U$ smaller, we may cover $M$ by open subsets $M_1,\dots,M_R$ such that for $\ell = 1,\dots,R$ there are smooth maps $G_{\ell,1},\dots,G_{\ell,p+2q-1} : U \times M_j \to M$ such that for every $T \in U$ and $x \in M_j$ the points $G_{\ell,1}(T,x),\dots,G_{\ell,p+2q+d-1}(T,x)$ are distinct antecedents of $x$ by $T$. We use here our assumptions on the degree of $F_0$. Consider compact subsets $K_1,\dots,K_R$ of $M_1,\dots,M_R$ whose interiors cover $M$.

For $T \in U$ and $\ell \in \set{1,\dots,R}$, we introduce the map $\mathcal{G}_{T,\ell} : M \to (\mathbb{R}^p \times \mathbb{C}^q)^{p + 2q + d -1}$ defined by
\begin{equation*}
\mathcal{G}_{T,\ell}(x) = (\F_T(G_{\ell,1}(T,x)),\dots,\F_T(G_{\ell,p+2q+d-1}(T,x))) \textup{ for } x \in M.
\end{equation*}
For $k = 0,\dots,p+2q$, let 
\begin{equation*}
\begin{split}
N_k = \{(C_1,\dots,C_{p+2q + d-1}) \in & (\mathbb{R}^p \times \mathbb{C}^q)^{p + 2q + d -1} : \\ & \qquad \qquad \dim \langle C_1,\dots,C_{p+2q +d-1} \rangle = k\}.
\end{split}
\end{equation*}
It is useful to notice that:

\begin{lemma}\label{lemma:submanifold}
If $k \in \set{0,\dots,p+2q}$, then $N_k$ is a smooth submanifold of $(\mathbb{R}^p \times \mathbb{C}^q)^{p+2q + d-1}$ of codimension $(p+2q + d-1-k)(p+2q-k)$.
\end{lemma}

\begin{proof}
Let $k \in \set{0,\dots,p+2q}$. Let $C = (C_1,\dots,C_{p+2q+d-1})$ be an element of $N_k$. There is a subset $I$ of $\set{1,\dots,p+2q + d-1}$ of cardinal $k$ such that $(C_i)_{i \in I}$ is a linearly independent family. For notational simplicity, let us assume that $I = \set{1,\dots,k}$. There is a neighbourhood $W$ of $C$ in $(\mathbb{R}^p \times \mathbb{C}^q)^{p+2q + d-1}$ such that if $c = (c_1,\dots,c_{p+2q + d-1}) \in W$ then $(c_1,\dots,c_k)$ is a linearly independent family. For such a $c$, let $\Pi(c)$ be the orthogonal projector on the orthogonal of the subspace of $\mathbb{R}^p \times \mathbb{C}^q$ spanned by $c_1,\dots,c_k$. Here, we endow $\mathbb{R}^p \times \mathbb{C}^q$ with an Euclidean structure for instance by identifying it with $\mathbb{R}^{p+2q}$. Consider then the map
\begin{equation*}
\begin{array}{ccccc}
\psi & :&  W & \to & (\im \Pi(C))^{p+2q + d-1-k} \\
& & c = (c_1,\dots,c_{p+2q + d-1}) & \mapsto & (\Pi(C) \Pi(c)c_{k+j})_{1 \leq j \leq p + 2q + d-1-k}.
\end{array}
\end{equation*}
Up to making $W$ smaller, for every $c \in W$, the map $\Pi(C)$ induces an isomorphism between the range of $\Pi(c)$ and the range of $\Pi(C)$. It follows that $N_k \cap W = \psi^{-1}(\set{0})$. The map $\psi$ is a submersion at $C$ (one only needs to act on the $p+2q+d-1-k$ last columns to get surjectivity of the derivative), and thus $N_k$ is a submanifold of $(\mathbb{R}^p \times \mathbb{C}^q)^{p+2q + d-1}$ of codimension $(p+2q + d-1-k)(p+2q-k)$.

Let us mention that the tangent space to $N_k$ at $C$ is
\begin{equation*}
\begin{split}
\T_C N_k = \{(v_1,\dots,v_{p+2q}) \in (\mathbb{R}^p \times & \mathbb{C}^q)^{p+2q+d-1}: \\ & \qquad \qquad v_{k+i} - \sum_{j = 1}^k \mu_{i,j}v_j \in \langle C_1,\dots,C_k \rangle\}
\end{split}
\end{equation*}
where the coefficients $\mu_{i,j}$ are defined by $C_{k+i} = \sum_{j = 1}^k \mu_{i,j} C_j$.
\end{proof}

Notice that if $T \in U, \ell \in \set{1,\dots,R}$ and $x \in M_j$ are such that $(\F_T(y))_{Ty = x}$ does not span $\mathbb{R}^p \times \mathbb{C}^q$, then $\mathcal{G}_{T,\ell}(x)$ belongs to $N_k$ for a $k$ between $0$ and $p+2q-1$. We progress toward the proof of Lemma \ref{lemma:generic_solvability} by showing that:

\begin{lemma}\label{lemma:transversality}
Let $U_1$ be the set of $T \in U$ such that for $\ell = 1,\dots,R$ the map $\mathcal{G}_{T,\ell}$ is tranvserse to $N_k$ along $K_j$ for $k = 0,\dots,p+2q-1$. The set $U_1$ is open and dense in $U$.
\end{lemma}

Notice that the formula for the codimension of $N_k$ in Lemma \ref{lemma:submanifold} implies that if $T \in U_1$ and $\ell \in \set{1,\dots,R}$ then $\mathcal{G}_{T,\ell}(K_\ell)$ does not intersect $N_0,\dots,N_{p+2q-2}$ and intersects $N_{p+2q-1}$ at a finite number of points.

\begin{proof}[Proof of Lemma \ref{lemma:transversality}]
For $\ell = 1,\dots,R$, let $U_1^{(\ell)}$ be the set of $T \in U$ such that $\mathcal{G}_{T,\ell}$ is transverse to $N_k$ along $K_\ell$ for $k = 0,\dots,p+2q-1$. We will prove that the $U_1^{(\ell)}$'s are open and dense in $U$, and it will then follows that $U_1 = \bigcap_{\ell = 1}^R U_1^{(\ell)}$ is open and dense.

Let us fix $\ell \in \set{1,\dots,R}$ for the end of the proof. Notice that $T \in U_1^{(\ell)}$ if and only if the restriction to $K_\ell$ of $\mathcal{G}_{T,\ell}$ takes value in the open set $N_{p+2q-1} \cup N_{p+2q}$ and has only transverse intersections with $N_{p+2q-1}$. Since $N_{p+2q-1}$ is a closed submanifold of $N_{p+2q-1} \cup N_{p+2q}$, we find that $U_1^{(\ell)}$ is open. To prove that it is dense, let us pick $T \in U$ and construct a nearby map in $U_1^{(\ell)}$.

Consider a point $x \in K_\ell$. According to Lemma \ref{lemma:key_deformation}, we can find vector fields $X_1^x,\dots,X_{(p+2q + d-1)(p+2q)}^x$ such that the\footnote{We consider the elements of $(\mathbb{R}^p \times \mathbb{C}^q)^{p+2q +d-1}$ as $(p+q) \times (p+2q+d-1)$ matrix whose first $p$ rows have real entries and last $q$ rows have complex entries}
\begin{equation*}
(H_{T,\lambda_j(T)} P_T(X_{r}^x) f_{T,j} (G_{\ell,s}(T,x)))_{\substack{1 \leq j \leq p+q \\ 1 \leq s \leq p + 2q + d-1}}
\end{equation*}
for $r = 1,\dots, (p+2q + d-1)(p+2q)$ form a basis of $(\mathbb{R}^p \times \mathbb{C}^q)^{p + 2q + d -1}$. Here, we use the fact that there is an antecedent for $x$ which is not among $(G_{\ell,s}(T,x))_{1 \leq s \leq p + 2q + d-1}$, which follows from our assumption on the degree of $F_0$. Notice in addition that the vector fields $X_r^x, r = 1,\dots,(p+2q + d-1)(p+2q)$ are supported away from the antecedents of $x$ by $T$. Hence, if we use these vector fields to deform $T$, the maps $G_{\ell,s}(T,\cdot), s = 1,\dots,p+2q+d-1$ will be unchanged near $x$. There is then an open neighbourhood $V_x$ for $x$ in $M$ such that for every $y \in \overline{V}_x$ we have that the 
\begin{equation*}
(H_{T,\lambda_j(T)} P_T(X_{r}^x) f_{T,j} (G_{\ell,s}(T,y)))_{\substack{1 \leq j \leq p+q \\ 1 \leq s \leq p + 2q + d-1}}
\end{equation*}
for $r = 1,\dots, (p+2q + d-1)(p+2q)$ form a basis of $(\mathbb{R}^p \times \mathbb{C}^q)^{p + 2q + d -1}$, and that for $r = 1,\dots,(p+2q+d-1)(p+2q)$, the vector field $X_r^x$ vanishes on a neighbourhood of the antecedents of $y$ by $T$.

Since $K_\ell$ is compact, there are $x_1,\dots,x_N \in M$ such that $K_\ell \subseteq \bigcup_{n = 1}^N V_{x_n}$. Let $Y_1,\dots,Y_K$ be an enumeration of all the vector fields of the form $X_r^x$ for $r \in \set{1,\dots,(p+2q)(p+2q+d-1)}$ and $x \in \set{x_1,\dots,x_N}$. For $n = 1,\dots, K$, let $(\phi_n^t)_{t \in \mathbb{R}}$ be the flow of $Y_n$. For $\bar{t} = (t_1,\dots,t_K) \in \mathbb{R}^K$, let $T_{\bar{t}} = T \circ \phi_{1}^{t_1} \circ \dots \circ \phi_{K}^{t_K}$. There is then an $\epsilon_0 > 0$ such that for every $\bar{t} \in (-\epsilon_0,\epsilon_0)^K$ the map $T_{\bar{t}}$ belongs to $U$. We define then the map
\begin{equation*}
\begin{array}{ccccc}
\psi & : & M \times (-\epsilon_0,\epsilon_0)^{K} & \to & (\mathbb{R}^p \times \mathbb{C}^q)^{p + 2q + d-1} \\
 & & (x,\bar{t}) & \mapsto & \mathcal{G}_{T_{\bar{t}}}(x).
\end{array}
\end{equation*}
It follows from Lemma \ref{lemma:useful_linear_response} and Remark \ref{remark:useful_smoothness} that this map is smooth. Let us prove that there is $\epsilon_1 > 0$ such that for $k = 0,\dots,p+2q-1$, the map $\psi$ restricted to $\bigcup_{n = 1}^N V_{x_n} \times (-\epsilon_1,\epsilon_1)^K$ is transverse to $N_k$.

Consider $k \in \set{0,\dots,p+2q-1}, y \in \bigcup_{n = 1}^N V_{x_n}$ and $\bar{t} \in (-\epsilon_0,\epsilon_0)^K$ close to $0$. If $\mathcal{G}_{T_{\bar{t}},\ell}(y)$ does not belong to $N_k$, there is nothing to prove. Otherwise, let $n \in \set{1,\dots,N}$ be such that $y \in V_{x_n}$. Among the partial derivatives of $\psi$ with respect to $\bar{t}$ at $(y,\bar{t})$, there are, according to Lemma \ref{lemma:derivative_simple_resonance}, the
\begin{equation}\label{eq:complicated_derivative}
(H_{T_{\bar{t}},\lambda_j(T_{\bar{t}})} P_{T_{\bar{t}}} (\widetilde{X}_{r,\bar{t}}^{x_n}) f_{T_{\bar{t}},j} (G_{\ell,s}(T_{\bar{t}},y)))_{\substack{1 \leq j \leq p+q \\ 1 \leq s \leq p + 2q + d-1}} + D_r \mathcal{G}_{T_{\bar{t}},\ell}(y)
\end{equation}
for $r = 1,\dots,p+2q+d-1$.  Here, $\widetilde{X}_{r,\bar{t}}^{x_n}$ is a vector field that depends continuously on $\bar{t}$ and coincides with $X_{r}^{x_n}$ when $\bar{t} = 0$. The matrix $D_r$ is a diagonal, i.e. $D_r \mathcal{G}_{T_{\bar{t}},\ell}(y)$ is deduced from $\mathcal{G}_{T_{\bar{t}},\ell}(y)$ by mutliplying its $p$ first rows by a real number (that may depend on the row) and the last $q$ rows by a complex number. The term $D_r \mathcal{G}_{T_{\bar{t}},\ell}(y)$ comes from the second term in \eqref{eq:derivative_simple_resonance}. Since $N_k$ is invariant by dilation of the rows of an element of $(\mathbb{R}^p \times \mathbb{C}^q)^{p + 2q + d-1}$ seen as a matrix, we find that $D_r \mathcal{G}_{T_{\bar{t}},\ell}(y)$ is tangent to $N_k$ at $\mathcal{G}_{T_{\bar{t}},\ell}(y)$. Since the first term in \eqref{eq:complicated_derivative} span $(\mathbb{R}^p \times \mathbb{C}^q)^{p + 2q + d-1}$ when $r$ goes from $1$ to $(p+2q+d-1)(p+q)$, provided $\bar{t}$ is small enough, we get that the range of $\mathrm{d}\psi(y,\bar{t})$ is transverse to $N_k$ at $\mathcal{G}_{T_{\bar{t}},\ell}(y)$.

Now that we have $\epsilon_1 > 0$ such that for $k = 0,\dots,p+2q-1$, the map $\psi$ restricted to $\bigcup_{n = 1}^N V_{x_n} \times (-\epsilon_1,\epsilon_1)^K$ is transverse to $N_k$, we find that the set of $\bar{t} \in (-\epsilon_1,\epsilon_1)^K$ such that $T_{\bar{t}} \in U_1^{(\ell)}$ is residual \cite[Theorem 2.7 p.79]{hirsch_differential_topology}. In particular, $T$ belongs to the closure of $U_1^{(\ell)}$.
\end{proof}

For every $T \in U_1$, we let $B(T)$ be the set of points $x \in M$ such that $(\F_T(y))_{y \in T^{-1}(\set{x})}$ does not span $\mathbb{R}^p \times \mathbb{C}^q$. Our goal is to prove that there is $T \in U_1$ such that $B(T) = \emptyset$. A direct consequence of Lemma \ref{lemma:transversality} is:

\begin{lemma}\label{lemma:usc}
For every $T \in U_1$, the set $B(T)$ is finite. Moreover, if $K$ is a closed subset of $M$, then the map $T \mapsto |B(T) \cap K|$ is upper-semi-continuous on $U_1$.
\end{lemma}

\begin{proof}
For $T \in U_1$ we have $B(T) \subseteq \bigcup_{\ell = 1}^{R} \mathcal{G}_{T,\ell}^{-1}(N_{p+2q-1}) \cap K_\ell$. The latter is a union of finite set (see the codimension formula in Lemma \ref{lemma:submanifold}) and thus a finite set.

Let us move to the upper semi-continuity statement. Let $K$ be a closed subset of $M$. Let $T \in U_1$. Consider a point $x$ of $K$. If $x \notin B(T)$, then there is an open neighbourhood $V_x$ for $x$ in $M$ and a neighbourhood $\widetilde{U}_{x}$ for $T$ in $U_1$ such that for every $y \in V_x$ and $\widetilde{T} \in \widetilde{U}_x$ we have $y \notin B(\widetilde{T})$. If $x \in B(T)$, there is $\ell \in \set{1,\dots,R}$ such that $x \in K_\ell$. Since $x \in B(T)$, we must have $\mathcal{G}_{T,\ell}(x) \in N_{p+2q-1}$. Since  $\mathcal{G}_{T,\ell}$ is transverse to $N_{p+2q-1}$ at $x$, there is an open neighbourhood $\widetilde{U}_x$ of $T$ in $U_1$, a neighbourhood $\widetilde{V}$ of $x \in M_\ell$ and a continuous map $\psi : \widetilde{U}_x \to \widetilde{V}$ such that for every $T' \in \widetilde{U_x}$, the point $\psi(T')$ is the only point in $\widetilde{V}$ mapped within $N_{p+2q-1}$ by $\mathcal{G}_{T',\ell}$. Hence, $\psi(T')$ is the only point in $\widetilde{V}$ that can belong to $B(T')$. Consequently, up to making $\widetilde{U}_x$ smaller, there is a neighbourhood $V_x$ of $x$ such that for every $T' \in \widetilde{U}_x$, there is at most one point of $B(T')$ in $V_x$.

Since $K$ is compact, we can find $x_1,\dots,x_N \in K$ such that $K \subseteq \bigcup_{\ell = 1}^N V_{x_\ell}$. For $T' \in \bigcap_{\ell = 1}^{N} \widetilde{U}_{x_\ell}$, we have then that 
\begin{equation*}
|B(T') \cap K| \leq \sum_{\ell = 1}^{N} |B(T') \cap V_{x_\ell}| \leq |\set{\ell \in \set{1,\dots,N} : x_\ell \in B(T)}| \leq |B(T) \cap K|. 
\end{equation*} 
\end{proof}

We are now ready to end the proof of Lemma \ref{lemma:generic_solvability}.

\begin{proof}[Proof of Lemma \ref{lemma:generic_solvability}]

The proof is by contradiction. Let $F_2$ be an element $U_1$ such that $|B(F_2)| = \min\limits_{T \in U_1} |B(T)|$ and assume that $B(F_2) \neq \emptyset$. 

\underline{Step 1.} Pick $x_0 \in B(F_2)$ and let $V_0$ be an open neighbourhood of $x_0$ such that $B(F_2) \cap \overline{V}_0 = \set{x_0}$. It follows from Lemma \ref{lemma:usc} that there is an open neighbourhood $U_2$ of $F_2$ in $U_1$ such that for every $T \in U_2$ we have $|B(T) \cap (M \setminus V_0)| \leq |B(F_2)\cap (M \setminus V_0)| = |B(F_2)| - 1$ and $|B(T) \cap \overline{V}_0| \leq 1$. Hence, for $T \in U_2$, there is a unique point $x_0(T)$ in $B(T) \cap V_0$. Notice also that for every $T \in U_2$ we have $|B(T)| = |B(F_2)|$. Thus, the argument we just used implies that the map $T \mapsto x_0(T)$ is continuous on $U_2$. Hence, up to taking $U_2$ smaller, there is $\ell \in \set{1,\dots,R}$ such that $x_0(T) \in K_\ell$ for every $T \in U_2$. Notice also that $x_0(F_2) = x_0$.

For every $T \in U_2$, since $x_0(T) \in B(T) \cap K_\ell$, the range of $\mathcal{G}_{T,\ell}(x_0(T))$ is spanned by $p+2q-1$ of its columns. Up to taking $U_2$ smaller, we may assume that the indexes of these columns do not depend on $T$. Up to relabelling, we can even assume that the range of $\mathcal{G}_{T,\ell}(x_0(T))$ is spanned by its $p+2q-1$ first columns

\underline{Step 2.} We claim that, up to making $U_2$ smaller, we may assume that for every $T \in U_2$ the point $x_0(T)$ is not a fixed point for $T$. To do so, we only need to find an element $T$ of $U_2$ such that $T(x_0(T)) \neq x_0(T)$ (we can then replace $U_2$ by a small neighbourhood of this $T$). We prove by contradiction that such a $T$ exists and assume that for every $T \in U_2$ we have $T(x_0(T)) = x_0(T)$. Pick any $T_0 \in U_2$. As in the proof of Lemma \ref{lemma:transversality}, we may use Lemma \ref{lemma:key_deformation} to produce a smooth deformation $(T_t)_{t \in (-\epsilon_0,\epsilon_0)}$ of $T_0$ such that $\frac{\mathrm{d}}{\mathrm{d}t} \mathcal{G}_{T_t,\ell}(x_0(T_0))_{t = 0}$ takes any value in $(\mathbb{R}^p \times \mathbb{C}^q)^{p+2q}$, up to the tangent space of $N_{p+2q-1}$. Notice that $T_t$ coincides with $T_0$ near $x_0(T_0)$ and its antecedents. Since $\mathcal{G}_{T_0}$ is transverse to $N_{p+2q-1}$ at $x_0(T_0)$ and $\mathcal{G}_{T_t,\ell}(x_0(T_t)) \in N_k$ for every $t \in (- \epsilon_0,\epsilon_0)$, the map $t \mapsto x_0(T_t)$ is smooth and we can impose the value of $\frac{\mathrm{d}}{\mathrm{d}t}x_0(T_t)_{t = 0}$. In particular, we can make it non-zero. But since $T_t$ coincides with $T_0$ near $x_0(T_0)$, we must have $x_0(T_t) = x_0(T_0)$ (because the fixed point of $T_0$ are isolated), a contradiction.

\underline{Step 3.} For $T \in U_2$, let us define the function $\mathcal{H}_T$ from on a neighbourhood of $x_0(T)$ to $(\mathbb{R}^p \times \mathbb{C}^q)^{p+2q}$ by $$\mathcal{H}_T(x) = (\widetilde{\F}_T(G_{\ell,1}(T,x)),\dots,\widetilde{\F}_T(G_{\ell,p+2q-1}(T,x)),\F_T(x)),$$ where the function $\widetilde{\F}_T$ is defined on $M$ by $\widetilde{\F}_T(x) = (\lambda_j(T)^{-1} f_{T,j}(x))_{1 \leq j \leq p + 2q}$. From the definition of resonant states, we find that $\F_T(x)$ belongs to the span of $(\widetilde{\F}_T(y))_{y \in T^{-1}(\set{x})}$. Hence, $\mathcal{H}_T(x_0(T))$ belongs to 
\begin{equation*}
\begin{split}
\mathcal{N} = \{(C_1,\dots & ,C_{p+2q}) \in (\mathbb{R}^{p} \times \mathbb{C}^q)^{p+2q} : \\ & \dim \langle C_1,\dots,C_{p+2q}\rangle = \dim \langle C_{1},\dots,C_{p+2q-1}\rangle = p+2q-1\}.
\end{split}
\end{equation*}
Otherwise $(\widetilde{\F}_T(y))_{y \in T^{-1}(\set{x})}$ would span $\mathbb{R}^p \times \mathbb{C}^q$, which implies that the family $(\F_T(y))_{y \in T^{-1}(\set{x})}$ spans $\mathbb{R}^p \times \mathbb{C}^q$. 
 
Notice that the tangent space to $\mathcal{N}$ at a point $C = (C_1,\dots,C_{p+2q})$ is
\begin{equation*}
\begin{split}
\T_C \mathcal{N} = \Big\{(v_1,\dots,v_{p+2q}) \in (\mathbb{R}^p & \times \mathbb{C}^q)^{p+2q} \\ &: v_{p+2q} - \sum_{j = 1}^{p+2q-1} \mu_j v_{j} \in \langle C_1,\dots,C_{p+2q-1} \rangle\Big\},
\end{split}
\end{equation*}
where $\mu_1,\dots,\mu_{p+2q-1}$ are the real numbers such that $C_{p+2q} = \sum_{j = 1}^{p+2q-1} \mu_j C_j$. According to Lemma \ref{lemma:key_deformation}, we may make any first order deformation of $\F_T$ near $x_0(T)$ without changing $T$ (at all) or $\F_T$ (at first order) near the points $G_{\ell,1}(T,x_0(T)),\dots,G_{\ell,p+2q-1}(T,x_0(T))$. Here, we apply Lemma \ref{lemma:key_deformation} with $E$ being $\set{G_{\ell,1}(T,x_0(T)),\dots,G_{\ell,p+2q-1}(T,x_0(T))}$. There are at least two antecedents for $x_0(T)$ outside of $E$ (by the degree condition) and thus at least one which is not periodic. The resonances are also unchanged at first order, see Lemma \ref{lemma:derivative_simple_resonance}. Working as in the proof of Lemma \ref{lemma:transversality} and in view of the tangent space to $\mathcal{N}$ (we only need deformations along the last column to get transversality), we can find $\widetilde{T} \in U_2$ such that $\mathcal{H}_{\widetilde{T}}$ is transverse to $\mathcal{N}$ at $x_0(\widetilde{T})$. There is then an open neighbourhood $\widetilde{U}_2$ of $\widetilde{T}$ in $U_2$ such that for every $T \in \widetilde{U}_2$ the map $\mathcal{H}_T$ is transverse to $\mathcal{N}$ at $x_0(T)$.

\underline{Step 4.} Recall the formula for the tangent space to $N_{p+2q-1}$ given in the proof of Lemma \ref{lemma:submanifold}. From this formula, we find that, starting from any map $T_0 \in \widetilde{U}_2$, we can use Lemma \ref{lemma:key_deformation} to produce a smooth deformation $(T_t)_{t \in (-\epsilon_0,\epsilon_0)}$ with $\frac{\mathrm{d}}{\mathrm{d}t} F_{T_t}(x)_{t = 0} = 0$ for $x$ near $G_{\ell,k}(T,x_0(T_0)), k = 1,\dots,p+2q-1$ and such that $\frac{\mathrm{d}}{\mathrm{d}t} \mathcal{G}_{T_t,\ell}(x_0(T_0))_{t= 0}$ projects to an arbitrarily element of the quotient $(\mathbb{R}^p \times C^q)^{p+2q}/ \T_{\mathcal{G}_{T_0,\ell}(x_0(T_0))} N_{p+2q-1}$. Indeed, we only need deformations along the last $d$ columns to produce any element of this space. Since $\mathcal{G}_{T_0,\ell}$ is transverse to $N_{p+2q-1}$ at $x_0(T_0)$, it means that we can impose the value of $\frac{\mathrm{d}}{\mathrm{d}t} x_0(T_t)_{|t=0}$ (using that $\mathcal{G}_{T_t,\ell}(x_0(T_t))$ stays within $N_{p+2q-1}$). Here, we use the fact that the deformation $(T_t)_{t \in (-\epsilon_0,\epsilon_0)}$ is trivial near $x_0(T_0)$ and its antecedents by $T_0$. The condition that we impose on $\frac{\mathrm{d}}{\mathrm{d}t} x_0(T_t)_{|t=0}$ is that $\mathrm{d}\mathcal{H}_{T_0}(x_0(T_0)) \cdot \frac{\mathrm{d}}{\mathrm{d}t} x_0(T_t)_{|t=0}$ is not tangent to $\mathcal{N}$ (which is possible because $\mathcal{H}_{T_0}$ is transverse to $\mathcal{N}$ at $x_0(T_0)$).

Since $x_0(T_0)$ is not a fixed point for $T_0$ (thanks to step 2 of the current proof), we may arrange the perturbation so that $\F_{T_t}$ is unchanged at first order near $x_0(T_0)$ (point \ref{item:not_too_much} of Lemma \ref{lemma:key_deformation}). It implies that the map $\mathcal{H}_{T_t}$ is unchanged at first order near $x_0(T_0)$ by the deformation $t \mapsto T_t$, and thus that $\frac{\mathrm{d}}{\mathrm{d}t}(\mathcal{H}_{T_t}(x_0(T_t)))_{|t= 0} = \mathrm{d}\mathcal{H}_{T_0}(x_0(T_0)) \cdot \frac{\mathrm{d}}{\mathrm{d}t} x_0(T_t)_{|t=0}$. Hence, for $t$ small but non-zero we have $\mathcal{H}_{T_t}(x_0(T_t)) \notin \mathcal{N}$, a contradiction.
\end{proof}

\appendix

\section{Existence of simple real resonances}\label{appendix:existence_simple_real_resonances}

We prove here a result that implies that Theorem \ref{theorem:generic_real_resonance} is not empty.

\begin{proposition}\label{proposition:existence_simple_resonance}
Let $M$ be smooth compact connected manifold. Assume that $\Exp(M)$ is not empty. There is $T \in M$ with a simple real resonance distinct from $1$.
\end{proposition}

\begin{proof}[Proof of Proposition \ref{proposition:existence_simple_resonance}]
\underline{Step 1.} We start by proving that there is a smooth expanding map $T_0$ on $M$ with a resonance distinct from $1$.

By assumption, there is $T_0 \in \Exp(M)$. We may assume that the sum $\sum_{ T_0 x = x} \frac{1}{|\det(I - D T_0(x))|}$ is not equal to $1$ (otherwise, we can just modify $T_0$ near a fixed point to achieve it). Then, we choose a real-analytic structure on $M$ and we may assume that $T_0$ is real-analytic (since real-analytic maps are dense within smooth maps). Since $T_0$ is analytic, the sum of the resonances of $T_0$ is\footnote{This fact follows for instance from \cite[\S 6.6 and 6.7]{ruelle_expanding_maps} and \cite[Theorem 3.4]{jezequel_local_global}.} $\sum_{T_0x =x } \frac{1}{|\det (I -DT_0(x))|} \neq 1$. Hence, $T_0$ must have a resonance distinct from $1$.

\underline{Step 2.} We have $T_0 \in \Exp(M)$ with a resonance $\lambda_0$ distinct from $1$. Thanks to Theorem \ref{theorem:generic_simple}, we may assume that $\lambda_0$ is simple. If $\lambda_0$ is real, we are done. Otherwise, we prove that there is $F_0 \in \Exp(M)$ with a real resonance $\mu_0$ of multiplicity $2$ with no Jordan block, and such that there is no other resonance of modulus $|\mu_0|$ for $F_0$.

To do so, we replace $T_0$ by an iterate of itself to ensure $\deg T_0 \geq \dim M +5$. Then we use Lemmas \ref{lemma:local_solvability} and \ref{lemma:generic_solvability}, maybe several times, to replace $T_0$ by a nearby map and assume that the argument of $\lambda_0$ is a rational multiple of $\pi$ and that there are no other resonances for $T_0$ with modulus $|\lambda_0|$ except $\bar{\lambda}_0$. It follows that there is $n \geq 2$ such that $\lambda_0^n \in \mathbb{R}^*$ is a resonance of $T_0^n$ of multiplicity $2$ without Jordan blocks. We set $F_0 = T_0^n$ and $\mu_0 = \lambda_0^n$.

\underline{Step 3.} Let $f_1,f_2$ be a basis of $E_{F_0,\mu_0}$ and $\nu_1,\nu_2$ be coresonant states associated to $\mu_0$ for $F_0$ such that $\nu_{j}(f_k) = \delta_{j,k}$ for $j,k \in \set{1,2}$. Working as in the proof of Lemma \ref{lemma:linear_algebra}, we find that there is $N \geq 1$ such that for every $x \in M$ the family $((f_1(y),f_2(y)))_{y \in F_0^{-N}(\set{x})}$ spans $\mathbb{R}^2$. If $N > 1$, we replace $F_0$ by $F_0^N$ to get $N = 1$. The proof of Lemma \ref{lemma:inverse_after_identification} implies then that there is $X \in \Gamma(\T M)$ such that $P_T(X) f_1 = f_2$ and $P_T(X) f_2 = f_1$.

Let $(\phi_t)_{t \in \mathbb{R}}$ be the flow of $X$ and for $t$ small let $F_t = F_0 \circ \phi_t$. We let $\Pi(t)$ be the spectral projector on resonances for $F_t$ near $\mu_0$. For $t$ small, let $A(t)$ be the matrix for the operator induced by $\mathcal{L}_{F_t}$ on the range of $\Pi(t)$ in the basis $(\Pi(t) f_1, \Pi(t)f_2)$. Notice that 
\begin{equation*}
A(0) = \begin{bmatrix}
\mu_0 & 0 \\ 0 & \mu_ 0
\end{bmatrix} \textup{ and }
A'(0) = \begin{bmatrix}
0 & 1\\
1 & 0
\end{bmatrix}.
\end{equation*}
Hence, we see that for $t > 0$ small the resonance $\mu_0$ splits into two simple real resonances (for $t< 0$ it splits into two simple non-real resonances). Thus, for $t > 0$ small, $F_t$ has a simple real-resonance.
\end{proof}

\section{Weighted transfer operators}\label{appendix:weighted_transfer_operator}

We only studied in this paper transfer operators and resonances associated to absolutely continuous invariant measures for smooth expanding maps. However, more general weighted transfer operators are interesting when studying Gibbs states. Let us fix a smooth expanding map $T \in \Exp(M)$ and let $\mathbb{K} = \mathbb{R}$ or $\mathbb{C}$. For $g \in C^\infty(M,\mathbb{K})$, we define an operator $\mathcal{M}_g$ by
\begin{equation*}
\mathcal{M}_g f(x) = \sum_{\substack{y \in M \\ T y = x}} g(y) f(y) \textup{ for } x \in M
\end{equation*}
when $f$ is a function from $M$ to $\mathbb{C}$. There is a theory of Ruelle resonances for such operators, see \cite{ruelle_expanding_maps} and \cite[Part I]{baladi_book2} (more generally, one could consider transfer operator associated to vector bundle extensions of $T$). The case which is relevant for the study of Gibbs states is $g$ valued in $\mathbb{R}_+^*$.

We expect that most of the analysis of the current paper may be adapted to study the resonant states for $\mathcal{M}_g$ for a generic $g$ (eventually restricting to the open set of non-vanishing $g$). Let us mention some differences that need to be taken into account to deal with this other case:
\begin{itemize}
\item Remark \ref{remark:1_resonance} does not hold anymore. Hence, $1$ is not necessarily a resonance, and the resonant states associated to resonances distinct from $1$ do not need to have zero average anymore. Consequently, the spaces $C_0^\infty(M,\mathbb{R})$ and $C_0^\infty(M,\mathbb{C})$ should be replaced in $C^\infty(M,\mathbb{R})$ and $C^\infty(M,\mathbb{C})$ in the analysis.
\item If $\mathbb{K} = \mathbb{R}$, then Remark \ref{remark:real_valued} still holds, but it is not the case anymore if $\mathbb{K} = \mathbb{C}$. Hence, in the latter case, one should not distinguish real and complex resonances (a real resonance may deformed into a complex one in that case). Each time we had to exclude pair of complex conjugates in the analysis (as in \S \ref{section:key_lemma} or Theorem \ref{theorem:general_statement}), this condition should drop when $\mathbb{K} = \mathbb{C}$.
\item The perturbation theory for the family of operators $g \mapsto \mathcal{M}_g$ is simpler than for $T \mapsto \mathcal{L}_T$, one can indeed rely on standard perturbation theory \cite{kato_book}. Due to this fact, it is likely that some results can even be stated in finite regularity without excessive technicalities.
\item The adaptation of Proposition \ref{proposition:full_support} is not obvious. If $g$ does not vanish, then the proof of Proposition \ref{proposition:full_support} can easily be adapted to prove that the coresonant states for $\mathcal{M}_g$ have full support. However, if $g$ is allowed to vanish then a new argument and some condition on $g$ are needed. Notice that it may happen that a coresonant state for $\mathcal{M}_g$ does not have full support (consider $g$ supported on a small neighbourhood of a fixed point for $T$).
\item The map $(g,f) \mapsto \mathcal{M}_g f$ is bilinear. Hence, we are interested in the surjectivity of the map $\varphi \mapsto \mathcal{M}_\varphi f = \mathcal{M}_f \varphi$ when $f$ is a resonant state for $\mathcal{M}_g$. We expect that some version of Lemma \ref{lemma:key_deformation} should still hold. There is however some noticeable point: if we try to adapt directly the proof of Lemma \ref{lemma:key_deformation}, then the condition $|h| < 1$ will not be guaranteed anymore when applying Lemma \ref{lemma:contracting_fixed_point}. However, looking at the proof of Lemma \ref{lemma:contracting_fixed_point}, we only need to know that $h(x_0)$ does not belong to some discrete set that depends on $G$. Hence, we could probably just ask some extra (generic) conditions on the resonances and weights to deal with this issue. Notice however, that the possibilty that $g$ might vanish is also a difficulty here.
\end{itemize}

\bibliographystyle{alpha}
\bibliography{biblio_expanding_generic.bib}

\end{document}